\documentclass[12pt,reqno]{amsart}
\usepackage{color,soul}
\sethlcolor{yellow}
\usepackage[margin=1.25in]{geometry} 
\usepackage{kpfonts}
\usepackage{enumerate}
\usepackage[mathscr]{eucal}
\usepackage{color} 
\usepackage[colorlinks=true, linkcolor=red, citecolor=blue, menucolor=black]{hyperref}

\usepackage{amsmath,amsthm, amsfonts, amssymb,MnSymbol}
\numberwithin{equation}{section}
\usepackage{graphicx}
\usepackage{listings}
\usepackage{enumerate}
\usepackage{centernot}
\usepackage{enumitem}
\usepackage{xfrac}
\usepackage{leftindex}
\usepackage{hyperref}

\theoremstyle{definition}

\newtheorem{main}{Theorem}

\newtheorem{thma}[main]{Theorem}
\newtheorem{cora}[main]{Corollary}
\newtheorem{theorem}{Theorem}[section]
\newtheorem{definition}[theorem]{Definition}
\newtheorem{lemma}[theorem]{Lemma}
\newtheorem{corollary}[theorem]{Corollary}
\newtheorem{proposition}[theorem]{Proposition}

\newtheorem{conjecture}[theorem]{Conjecture}

\newtheorem{claim}[theorem]{Claim}
\newtheorem{problem}[theorem]{Open Problem}

\newcommand{\A}{\mathcal A}
\newcommand{\B}{\mathcal B}
\newcommand{\C}{\mathcal C}
\newcommand{\M}{\mathcal M}
\newcommand{\ca}{\curvearrowright}
\newcommand{\Ll}{\mathcal L}
\newcommand{\Nn}{\mathcal N}
\renewcommand{\P}{\mathcal P}
\newcommand{\Q}{\mathcal Q}
\newcommand{\sK}{\mathscr K}
\newcommand{\sN}{\mathscr N}
\newcommand{\sP}{\mathscr P}
\newcommand{\sU}{\mathscr U}
\newcommand{\sZ}{\mathscr Z}
\newcommand{\sR}{\mathscr R}

\newcommand{\tQ}{\widetilde{\mathcal Q}}
\newcommand{\tM}{\widetilde{\mathcal M}}

\title[relative solidity of factors associated with relative hyperbolic groups]{Relative solidity results and their applications to computations of some II$_1$ factor invariants}
\author{J.\ F.\ Ariza Mejia, D.\ N.\ Amaraweera, I.\ Chifan, and K.\ Khan}

\begin{document}

\begin{abstract} In this paper we prove that whenever $G$ is hyperbolic relative to a family of exact, ressidually finite subgroups $\{H_1, \ldots, H_n\}$, the corresponding von Neumann algebra $\mathcal L(G)$ is solid relative to the family of subalgebras $\{\mathcal L(H_1),\ldots ,\mathcal L(H_n)\}$. Building on this result and combining it with findings from geometric group theory, we construct a continuum of icc property (T) relative hyperbolic groups that give rise to pairwise non virtually isomorphic	
factors, each of which has trivial one-sided fundamental semigroup.  
    
\end{abstract}

\maketitle

\section{Introduction}

In their fundamental work \cite{MvN43}, Murray and von Neumann associated in a natural way a von Neumann algebra, denoted by $\mathcal L(G)$, to every countable discrete group $G$. More precisely, $\mathcal L(G)$ is defined as the weak operator topology (WOT) closure of the complex group algebra $\mathbb C [G]$ acting by left convolution on the Hilbert space $\ell^2 (G)$ of square-summable functions on $G$. As noted in \cite{MvN43}, $\mathcal L(G)$ has a trivial center (i.e. is a factor) precisely when the group $G$ is icc. The structural theory of group von Neumann algebras has since been a central theme in operator algebras, driven by the following fundamental question: \emph{What structural properties of $G$ are retained by $\mathcal L(G)$?}

 An important aspect in the structural theory of $\mathcal{L}(G)$ is understanding their commuting von Neumann subalgebras. In this direction, Popa was able to prove in \cite{Pop83} that the (nonseparable) factor $\mathcal{L}(\mathbb{F})$ associated with the free group $\mathbb{F}$ with uncountably many generators is prime; i.e., it cannot be decomposed as $\mathcal{L}(\mathbb{F}) = \mathcal{P} \bar{\otimes} \mathcal{Q}$ for any diffuse factors $\mathcal{P}, \mathcal{Q}$. Using Voiculescu's free probability theory, Ge established the primeness of factors $\mathcal{L}(\mathbb{F}_n)$ associated with all free groups $\mathbb{F}_n$, where $2 \leq n \leq \infty$ \cite{Ge96}. Using C$^*$-algebraic techniques, Ozawa significantly generalized Ge's result by showing that the von Neumann algebra $\mathcal{L}(G)$ of any hyperbolic group $G$ is solid; i.e., for any diffuse von Neumann subalgebra $\mathcal{A} \subseteq \mathcal{L}(G)$, its relative commutant $\mathcal{A}' \cap \mathcal{L}(G)$ is amenable \cite{Oza03}. Popa was able to provide an alternative proof of Ozawa's solidity theorem using his deformation/rigidity theory \cite{Pop06b}. Using unbounded deformations, Peterson showed that $\mathcal{L}(G)$ is prime for any nonamenable group with positive first $L^2$-Betti number \cite{Pet06}. These results and the techniques involved were further refined and developed within the framework of deformation/rigidity theory by many mathematicians, leading to unprecedented progress on primeness/solidity results for large classes of groups \cite{Sin10, Vae10, CS11, CSU11, Ioa11, PV11, HV12, PV12, Ioa12, Bou12, BHR12, BV13, Vae13, Iso14, IM19, VV14, BC14, CKP14, Iso16, DHI19, CdSS17, Dri20}, just to enumerate a few.

In this paper, we study commuting subalgebras of $\mathcal L(G)$, when $G$ is a relative hyperbolic group. Specifically, we will make new progress on a relative solidity conjecture concerning these algebras, which can be traced to the work of Ozawa and has also been on the radar of several mathematicians. Using our relative solidity results, we will establish several W$^*$-rigidity results for these algebras, including the construction of a continuum of icc property (T) relative hyperbolic groups that give rise to pairwise non-virtually isomorphic factors, each of which has a trivial one-sided fundamental semigroup. In particular, this provides new progress on the Connes Rigidity Conjecture and other related problems, adding new examples to the recent developments \cite{CDHK20, CIOS22a, CIOS23b,CIOS24, CFQOT24}. 

\section{Main results}

The first goal of our paper is to investigate the following conjecture on relative solidity of von Neumann algebras associated with relative hyperbolic groups. 

\begin{conjecture}\label{relsolconj} Let $G$ be a group that is hyperbolic relative to a finite family of subgroups $\{H_1, \ldots, H_n\}$. Denote by $\mathcal M=\mathcal L(G)$ the von Neumann algebra corresponding to $G$. Let $p\in \mathcal M$ be a nonzero projection and let $\mathcal A \subseteq p\mathcal M p$ be a von Neumann subalgebra whose relative commutant $\mathcal A '\cap p\mathcal M p$ has no amenable direct summand.  Then one can find $1\leq i\leq n$ such that a corner of $\mathcal A$ intertwines into $\mathcal L(H_i)$ inside $\mathcal M$, in the sense of Popa; i.e., 
$\mathcal A\prec_{\mathcal M} \mathcal L(H_i)$ (see Section \ref{popaint}).  
    
\end{conjecture}

While at this time we cannot answer this positively in its full generality, we were, however, able to identify a fairly large class of relatively hyperbolic groups for which it holds true. Specifically, building on the prior methods in deformation/rigidity theory primarily from  \cite{IPP05,CS11,CSU11} and using Oyakawa's interesting recent work on group arrays \cite{Oya23} we are able to prove the following.

\begin{thma}\label{relsol1} Conjecture \ref{relsolconj} holds true whenever all peripheral subgroups $H_j$, $1 \leq j\leq n$ are exact and residually finite.
    \end{thma}
We note that, using the almost malnormality of the peripheral subgroups $H_i$ of a relative hyperbolic group $G$ \cite{DGO17}, the intertwining provided by Theorem \ref{relsol1} further yields a more complete description of the ``position'' of all commuting von Neumann subalgebras in $\mathcal{L}(G)$.

\begin{thma}\label{relsol2} Let $G$ be a group that is hyperbolic relative to a finite family of subgroups $\{H_1,\ldots, H_n\}$. Assume in addition $H_j$ is exact and residually finite for all $1\leq j \leq n$. Denote by $\mathcal M=\mathcal L(G)$ the group von Neumann algebra corresponding to $G$. Let $p\in \mathcal M$ be a nonzero projection and let $\mathcal A \subseteq p\mathcal M p$ be a diffuse von Neumann subalgebra. Then one can find a maximal projection $z_0\in \mathscr Z(\mathcal A'\cap  p \mathcal M p)$ such that \begin{enumerate}
\item the von Neumann algebra $(\mathcal A'\cap p \mathcal M p)z_0$ is amenable, and 
\item for every projection $q \in (\mathcal A'\cap p\mathcal M p)(1-z_0) $ one can find $1\leq i \leq n$, projections $a\in \mathcal A$ and $r \in q(\mathcal A'\cap p\mathcal M p)q$  with $ar\neq 0$ and a unitary $u\in \mathcal M$ such that  $u ar(\mathcal A\vee (\mathcal A'\cap p\mathcal M p)) ar u^*\subseteq \mathcal L(H_i)$.  
 \end{enumerate}   \end{thma}

Furthermore, a similar statement can also be derived at the level of algebras generated by various normalizing elements. We state our next result recall that for any inclusion of tracial von Neumann algebras $\mathcal A\subseteq \mathcal M$
the semigroup of one-sided quasi-normalizers $\mathcal {QN}^{1}_{\mathcal M}(\mathcal A)$ is defined as the set of all elements $x\in \M$ for which there exist $x_1,\ldots ,x_n\in \M$ such that $\mathcal A x\subseteq \sum_i x_i \mathcal A$. Then we define inductively a tower of one-sided normalizers von Neumann algebras for all ordinals $\beta,$ as follows: Let $\mathcal{A}^0 = \mathcal{A}.$ Now if $\beta$ is a successor ordinal then let $\mathcal{A}^{\beta}\subseteq \M $ be the von Neumann subalgebra generated by $\mathcal{QN}_{\mathcal M}^{1}(\mathcal{A}^{\beta - 1})$. Notice that $\mathcal{A}^{\beta -1} \subseteq \mathcal{A}^{\beta}.$ If $\beta$ is a limit ordinal then let $\mathcal{A}^{\beta} = \overline{\bigcup_{\gamma < \beta}\mathcal{A}^{\gamma}}^{\rm WOT}.$ Now let $\alpha$ be the first ordinal when the chain $(\mathcal{A}^{\beta})_{\beta}$ stabilizes, i.e. $\mathcal{A}^{\alpha}=\mathcal{A}^{\alpha +1}.$ Then denote by $\overline{\mathcal {QN}}_{\M}^{1}(\mathcal{A})= \mathcal{A}^{\alpha}.$

\begin{thma}\label{relsol3} Let $G$ be a group that is hyperbolic relative to a finite family of subgroups $\{H_1,\ldots, H_n\}$. Assume in addition $H_j$ is icc, exact and residually finite for all $1\leq j \leq n$. Denote by $\mathcal M=\mathcal L(G)$ the group von Neumann algebra corresponding to $G$. Let $p\in \mathcal M$ be a nonzero projection and let $\mathcal A \subseteq p\mathcal M p$ be a diffuse von Neumann subalgebra. Then one can find projections $z_0\in \mathscr Z(\mathcal A'\cap  p \mathcal M p)$, $ z_1, \ldots , z_n\in \mathscr Z(\overline{\mathcal{QN}}^1_{p\mathcal M p }(\mathcal A))$ with  $\sum^n_{i=0} z_i =p$ and unitaries $u_1,\ldots , u_n \in \mathcal M$ such that \begin{enumerate}
\item the von Neumann algebra $(\mathcal A'\cap p\mathcal M p)z_0$ is amenable, and 
\item  $u_i \overline{\mathcal {QN}}^{1}_{p\mathcal M p} (\mathcal A )z_i u_i^*\subseteq \mathcal L(H_i)$ for all $1\leq i \leq n$.  
 \end{enumerate}   \end{thma}

To this end we notice these statements are very similar in nature with several prior deep Kurosh-type results for amalgamanted free products of tracial von Neumann algebras from \cite{IPP05, Oza05, CH08}. 

\vskip 0.07in 
There are natural examples of relative hyperbolic groups  $G$ with exact, residually finite peripheral structure that emerged from geometric group theory to which our prior relative solidity results apply and are not covered by previous results in the literature. For instance, this is the case of all free-by-cyclic groups $G=\mathbb F_n\rtimes_\alpha \mathbb Z$ for any $n\geq 2$ and $\alpha\in {\rm Aut} (\mathbb F_n)$. Indeed, in \cite{Gho23} it was proved there is a maximal family $\{A_1,\ldots ,A_r\}$ of maximal subgroups of $ \mathbb F_n$ whose elements grow polynomially under $\alpha$  and  $G=\mathbb F_n\rtimes_\alpha \mathbb Z$ is hyperbolic relative to the collection $\{H_1,\ldots,H_r\}$, where $H_i$ is the suspensions of $A_i$ under $\alpha$ (see Subsection \ref{Examples}). Using the recent works \cite{DK23, GG23} this covers, more generally, many classes of hyperbolic-by-cyclic groups, free-by-free groups, etc (see Subsection \ref{Examples}).

\vskip 0.07in
These results provide new insight to a well-known open question concerning primeness of von Neumann algebras of relatively hyperbolic groups. 

A von Neumann algebra  $\mathcal M$ is called \emph{$s$-prime} if for every projection $p\in \mathcal M$ any two diffuse commuting von Neumann subalgebras $\mathcal A,\mathcal B\subset p\mathcal M p$ generate a von Neumann subalgebra $A \vee B \subset p \mathcal M p$ which has infinite index, in the sense of Popa-Pimsner \cite{PP86}, inside $p\mathcal M p$. In particular, s-prime von Neumann algebras cannot be decomposed as tensor product of diffuse von Neumann algebras. 

In \cite{CKP14} it was conjectured that $\mathcal L(G)$ is s-prime whenever $G$ is a non-elementary relative hyperbolic group.  This conjecture was confirmed in \cite[Theorem A]{CKP14}, for all $G$ which additionally have exact, residually finite peripheral structure. Here we obtain a further strengthening of this result, beyond the residually finite assumption.

\begin{thma}\label{primness2}  Let $G$ be a group hyperbolic relative to exact subgroups $\{ H_1, \ldots, H_n \}.$ 
Assume there are finite index subgroups $N_i \trianglelefteq H_i$ such that for every finite set $F \subset G$ the subgroup generated by the conjugates  $\{ N_i^h \,:\,1\leq i\leq n, h\in F\}$ has infinite index in $G$; e.g.\ this is the case when the normal closure $\llangle \cup_i N_i\rrangle^G$ has infinite index in $G$. \newline Then $\mathcal{L}(G)$ is s-prime.   \end{thma}

This already recovers well-known results such as the s-primeness of $\mathcal L(G)$ for any free product $G = H_1\ast\cdots \ast H_n\ast K$, where $H_i$s are arbitrary nontrivial exact groups and $K$ is any infinite hyperbolic group. However, the theorem yields  many new examples of s-prime group von Neumann algebras, beyond the free product case, which are not covered by any of the prior results. Many of them are in fact quotients of such free products obtained via deep methods in geometric group theory. 

Indeed, using small cancellation and Rips construction techniques, \cite{Osi06, BO06}  one can show for any finitely generated group $H$, there exists a property (T) group $G>H$ which is hyperbolic relative to $H$ and satisfies $[G:\llangle H\rrangle^G]=\infty$; see Theorem \ref{rips1} for a precise statement.  Thus, working with suitable groups $H$, this can be used to construct a continuum of property (T) relative hyperbolic groups satisfying the assumptions in Theorem \ref{primness2}. 

Additionally, there are also many $C'(1/6)$-small cancellation groups over free products that satisfy the assumptions in Theorem \ref{primness2}; see for instance Proposition \ref{smallcancoverf} and Corollary \ref{sprimeexamples}.
\vskip 0.1in

In the second part of the paper we use the prior relative solidity results to prove several strong rigidity results for II$_1$ factors of relative hyperbolic groups and also compute some of their invariants, e.g. their Murray-von Neumann fundamental group, \cite{MvN43}.  

\begin{thma}\label{rigidemb} For every $1 \leq i\leq 2$ assume that $G_i$ is hyperbolic relative to subgroups $H_i<G_i$. Assume in addition that $H_i =A_i \wr B_i$ where $A_i$ are infinite abelian and $B_i = B_i^1\times B_i^2$ where $B_i^j$  ICC non-amenable biexact for all $1\leq i,j\leq 2$.

\noindent For a scalar $t>0$ assume there exists a $\ast$-embedding $\Theta: \mathcal L(G_1)\rightarrow \mathcal L(G_2)^t$. 

\noindent  Then $t\in\mathbb N$ and there is a unitary $w\in \mathcal L(G_2)^t=\mathcal L(G_2)\otimes\mathbb M_t(\mathbb C)$ such that ${\rm ad}(w)\circ \Theta : \mathcal L(H_1)\rightarrow \mathcal L(H_2)^t$ is a $\ast$-embedding. 

\noindent In addition,  one can find $t_1,\ldots, t_q\in\mathbb N$ with $t_1+\cdots+t_q=t$, for some $q\in\mathbb N$, a finite index subgroup $K<B_1$, an injective homomorphism $\delta_i:K\rightarrow B_2$, and a unitary representation $\rho_i:K\rightarrow\sU_{t_i}(\mathbb C)$, for every $1\leq i\leq q$ such that 
$$\text{${\rm ad}( w)\circ \Theta(u_g)=\emph{diag}(v_{\delta_1(g)}\otimes\rho_1(g),
\ldots, v_{\delta_q(g)}\otimes\rho_q(g))
$, for every $g\in K$.}$$

\noindent We also have the following (finite index) inclusion $${\rm ad}(w)\circ \Theta (\mathcal L(A_1^{(B_1)}))\subseteq \mathcal L(A_2^{(B_2)})\otimes \mathbb D_t(\mathbb C).$$

\noindent Moreover, if $\Theta$ is a virtual $\ast$-isomorphism then the inclusion $\delta_i(K)<B_2$ has finite index for all $i$. In particular, $B_1$ is commensurable with $B_2$. 
\end{thma}

The prior von Neumann alegbraic rigidity results combined with various results in geometric group theory yield concrete computations of the one-sided fundamental semigroup of several classes of group II$_1$ factors, including many which have property (T). To properly introduce this invariant we recall first the notion of $t$-amplification $\mathcal M^t$, of a given II$_1$ factor $\mathcal M$  by an arbitrary positive scalar, $t>0$. When $1\geq t>0$, $\mathcal M^t$ is defined as the isomorphism class of the compression $p\mathcal M p$ for a projection $p\in \mathcal M$ of trace $\tau(p)=t$ and when $t>1$, it is defined as the isomorphism class of $p(\mathcal M \otimes \mathbb M_n(\mathbb C))p$ for an integer $n$ with $t/n\leq 1$ and a projection $p\in \mathcal M\otimes \mathbb M_n(\mathbb C)$  of trace $Tr\otimes \tau(p)=t/n$. One can see that, up to isomorphism, $\mathcal M^t$ does not depend on $n$ or $p$ but only on the value of $t$. Moreover classical results in the theory show that $(\mathcal M^t)^s \cong \mathcal M^{ts}$, for all $s,t>0$.

Following \cite{Ioa11,PV21}, the \emph{one-sided fundamental semigroup of $\mathcal M$} is denoted by $\mathcal F_s(\M)$ and consists of all scalars  $t\in \mathbb R_+$ for which there exists a $*$-embedding $\theta: \mathcal{M}\hookrightarrow\mathcal{M}^t$. It is not difficult to see that $\mathcal F_s(\mathcal{M})$ is closed under both addition and multiplication; in particular, we always have $\mathbb N \subseteq \mathcal F_s(\M)$. 

Moreover, $\mathcal F_s(\mathcal{M})$ naturally contains two important classical factor invariants: \begin{enumerate} \item [1)] the Murray-von Neumann fundamental group $\mathcal F(\mathcal M)$ defined as the group of all $t\in \mathbb R_+$ such that $\mathcal M^t\cong \mathcal M$ \cite{MvN43}, and 
\item [2)] VFR Jones invariant $\mathcal I(\M)$, i.e.\ the collection of all $r \geq 1$ for which there exists a II$_1$ subfactor $\mathcal N \subseteq \mathcal M$ which has Jones index  $[\mathcal M: \mathcal N]=r.$ 
\end{enumerate}



Several impressive results regarding the computation of the fundamental group, Jones invariant, and the one-sided fundamental semigroup for large classes of II$_1$ factors have been obtained over the last two decades via Popa's deformation/rigidity theory \cite{Pop01, IPP05, PV06, Vae08, Ioa11,PV21}. Remarkably, in a systematic study of embeddings between II$_1$ factors, Popa and Vaes showed in \cite{PV21} that any subset of $\mathbb N$ which is closed under addition and multiplication can be realized as the one-sided fundamental semigroup of a II$_1$ factor.
\vskip 0.05in
However, significantly less is known about these invariants when $\M=\mathcal L(G)$ when $G$ is an icc property (T) group. In this direction, Popa conjectured \cite{Pop06} that always $\mathcal F(\M)=\{1\}$. This has been verified only recently for a few classes of (cocycle) semidirect product groups \cite{CDHK20, CIOS22a, CIOS23b}.

In \cite{CIOS24} it was proposed the following strengthening of Popa's conjecture: 

\begin{conjecture}\label{trivialosfgconj} For every icc property (T) group $G$ we have $\mathcal F_s(\mathcal L(G))=\mathbb N$.\end{conjecture} Thus far, this has been verified only for a continuum family of property (T) wreath-like product groups, \cite{CIOS24}. 

Pairing Theorem \ref{rigidemb} with results in group theory \cite{AMO,NS21} we obtain additional families of groups satisfying Conjecture \ref{trivialosfgconj}, which are very different in nature from the ones introduced in \cite{CIOS24}.

\begin{cora} There exist a continuum of icc, relative hyperbolic groups $G$ with property (T) which give rise to pairwise non-virtually isomorphic factors $\mathcal L(G)$, satisfying $\mathcal F_s(\mathcal L(G))=\mathbb N$.
    
\end{cora}

In particular, this result  provides new insight into an open problem posed by P. de la Harpe \cite[Problem 5]{dlH95}, which asks: \emph{What are the possible values of the Jones invariant $\mathcal I(\mathcal L(G))$ when $G$ is an icc property (T) group?} Specifically, our results imply that only positive integers could occur. Incidentally, we would like to mention that there is no concrete computation available in the literature of $\mathcal I(\mathcal L(G))$ for any icc property (T) groups. The upcoming work \cite{CL25} makes the first progress in this direction.     


 \section{Preliminaries on groups}   

\subsection{Arrays on relatively hyperbolic groups}


Using a very interesting recent work of K.\ Oyakawa \cite{Oya23},  we highlight in this section a construction of  weakly-$\ell^2$ arrays on a fairly large class of relative hyperbolic groups. In some sense, these capture the peripheral length of group elements and will be heavily used to derive our main relative solidity results. 

We start by recalling the following result.  
\begin{proposition}\cite[Proposition~3.1]{Oya23}\label{array-from-the-group}
    Let $G$ be a finitely generated group that is hyperbolic relative to a collection $\{H_{\lambda}\}_{\lambda\in\Lambda}$ of subgroups of $G$. Then there exist a finite generating set $X$ of $G$, a 2-vertex-connected fine hyperbolic graph $\mathcal G$ on which $G$ acts without edge inversion and an array $\textbf{r}:G\rightarrow (\ell^2\mathcal G,\pi_{G})$ satisfying the following: 
    \begin{enumerate}
        \item The edge stabilizer is trivial for any edge $E(\mathcal G)$.
        \item For any $g\in G$, we get $d_{\tilde{\Gamma}}(1,g)\leq \frac{1}{2}\|\textbf{r}(g)\|^2_2$,
        where $\tilde{\Gamma}$ is the coned-off Cayley graph of $G$ with respect to $X$.
        \item The representation $(\ell^2\mathcal G,\pi_{G})$ is weakly contained in the left regular representation $(\ell^2G,\lambda_{G})$.
    \end{enumerate}
    
\end{proposition}

We continue with a few, very minor modifications of the results from \cite{Oya23}, that will be needed to construct our arrays. The first result concerns the structure of finite radius balls associated with an array in relatively hyperbolic groups.

\begin{proposition}\label{prop2.2}
     Let $G$ be a hyperbolic group relative to a finite collection of subgroups $\{H_i\}_{i\in\Lambda}$. Then there exist a weakly-$\ell^2$ representation $\sigma: G\rightarrow \mathscr U(\mathcal H)$ and an array $\textbf{s} :G\rightarrow \mathcal H$ such that for every $N\geq 0$ there exists $C_N>0$ for which we have $B^{\textbf{s}}_N:=\{ g\in G: \|\textbf{s}(g)\|\leq N\}\subseteq H_1^{a_1}H_2^{a_2}\cdots H_k^{a_k}F$ for some finite subset $F\subseteq G$ where $|F|,|a_j|<C_N$.
\end{proposition}

\begin{proof}
 By Proposition~\ref{array-from-the-group}, there is an array $\mathbf{r}:G\rightarrow \ell^2\mathcal G$. Let $\mathbf{r}_i: H_i\rightarrow \ell^2H_i$ be arrays, not necessarily proper, for all $i\in\Lambda$. Then for every $i$, consider the array $\mathbf{r}_i:G\rightarrow \ell^2G$ corresponding to $\mathbf{r}_i$ of constant $K_i\geq 0$, as constructed in \cite[Proposition~3.8]{Oya23}. Consider the representation $\sigma:G\rightarrow \mathscr U(\mathcal H)$ with $ \mathcal H:=\ell^2(\mathcal G)\oplus(\oplus_{i\in \Lambda}(\ell^2G))$, defined by  $
\sigma=\pi_{G}\oplus(\oplus_{i\in\Lambda} \lambda_G)$. Observe that $\sigma$ is weakly contained in the left regular representation $\lambda_{G}.$ Furthermore, let $\mathbf{s}:G\rightarrow \mathcal H$ be the array defined by $\mathbf{s}(g)=\mathbf{r}(g)\oplus (\oplus_{i\in \Lambda} \mathbf{r}_i)(g
)$ for all $g\in G$.   
\vskip 0.06in
Next, for a fixed $N\geq 0$, we analyze the structure of the finite ball, $A_N:=\{ g\in \Gamma: \lVert \mathbf{s}(g) \rVert \leq N\}$. First, denote by  $B_N:=\{g\in G: \|\mathbf{r}(g)\|\leq N\}$. Since, $\|\mathbf{s}(g)\|^2_2=\|\mathbf{r}(g)\|^2_2+\sum_{i\in\Lambda}\|\mathbf{r}_i(g)\|^2_2$, for all $g\in A_N$ we have $g\in B_N$ and $\|\mathbf{r}_i(g)\|\leq N$ for all $i\in\Lambda$. Using part (2) of Proposition~\ref{array-from-the-group} in combination of the fact that the identity map from the coned-off Cayley graph to $(G, X\cup H)$ where $H = \bigcup_{i\in \Lambda}H_i\setminus \{1\}$ is bi-Lipschitz, we get that there exists a constant $c_0$ such that; $d_{X\cup H}(1,g)\leq \frac{1}{2}c_0N^2=:C_N$. 
 Let $g=g_1h_1g_2h_2\cdots g_kh_kg_{k+1}$ be the label of a geodesic path from the identity to $g$ in the Cayley graph $\Gamma(G,X\cup H)$, where $g_i$'s are words in alphabets from $X\cup X^{-1}$ and $h_l\in H_i \setminus 1 $ for some $i\in\Lambda$ for every $l\in \{1,2,\ldots,k\}$. We then get;
 \begin{align*}\label{bounded-word}
     |g_1|+1+|g_2|+1+\cdots +1+|g_{k+1}|=d_{X\cup H}(1,g)\leq C_N.
 \end{align*}
 This implies that $|g_i|< C_N$ and $k<C_N$. Then the conclusion follows.\end{proof}

We record the following immediate corollary of Proposition \ref{prop2.2}.
\begin{corollary}
    Let $G$ be a hyperbolic group relative to a finite collection of residually finite subgroups $\{H_i\}_{i\in\Lambda}$. Let $N:=\llangle \cup_{i\in\Lambda}H_i\rrangle$ be the normal closure of $\cup_i H_i$ in $G$. Then the quotient group $G/N$ is bi-exact. 
\end{corollary}

 To properly introduce our next results, we introduce some terminology on peripheral filling of relative hyperbolic groups. Let $G$ be a hyperbolic group relative to a finite collection of subgroups $\mathcal P=\{H_i\}_{i\in I}$. A Dehn filling of $(G,\mathcal P)$ is the quotient $G(\mathcal N)= G/\llangle \mathcal N\rrangle$ where $\mathcal N = \bigcup_{i\in I} N_i$ and  $N_i\leqslant H_i$ is a finite index subgroup for all $i\in I$. Following \cite{Sun} we say that say that a Dehn filling $G(\mathcal N)$ of $(G,\mathcal P)$  satisfies the Cohen-Lyndon property if the normal closure can be decomposed as free product $\llangle \mathcal N\rrangle =\ast_{i\in I, t\in T_i} N_i^t$, for some left transversals $T_i \in LT(G, \llangle \mathcal N \rrangle H_i)$ where $i\in I$. Therefore, for such Dehn fillings we have the following short exact sequence: 
 \begin{equation}\label{ses2}1\rightarrow \ast_{i\in I, t\in T_i} N_i^t =\llangle \cup_i N_i\rrangle \hookrightarrow G\twoheadrightarrow G(\mathcal N)\rightarrow 1. \end{equation} 
 
Deep results in geometric group theory \cite{Osi06,DGO17,Sun} show that sufficiently deep Dehn fillings of any relative hyperbolic group satisfies the Cohen-Lyndon property and also $G(\mathcal N)$ is nonelementary hyperbolic. In particular, this is the case if each $H_i$ is ressidually finite and we pick $N_i\lhd H_i$ a normal subgroup of sufficiently large finite index.

 \begin{corollary}\label{array2}
   Let $G$ be a hyperbolic group relative to a finite collection of ressidualy finite subgroups $\{H_i\}_{i\in I}$. Let $N_i\lhd H_i$ be finite index subgroups for all $i\in I$ and let  $G(\mathcal N)$ be the corresponding Dehn filling.

   \noindent Then one can find a representation $\sigma: G \rightarrow\mathcal U(\mathcal H)$ and an array $\textbf{q}:G\rightarrow \mathcal H$ satisfying: 
   \begin{enumerate}
       \item $\sigma$ is weakly contained in $\lambda_G$, and
       \item For every $c>0$,  there is a finite set $F_c\subset G(\mathcal N)$ so that $B^{\textbf{q}}_c\subseteq \llangle \mathcal N\rrangle  F_c$. In particular, $G$ is biexact relative to $\llangle \mathcal N\rrangle$.  
   \end{enumerate}\end{corollary}

\begin{proof}Consider the array $\textbf{q} :G\rightarrow \mathcal H$ given by Proposition \ref{prop2.2}. Then by Proposition \ref{prop2.2} one can find finite sets $R,F\subset G$ such that $B^{\textbf{q}}_c\subseteq H^{a_1}_1
H^{a_2}_2 \cdots H^{a_k}_k F$ for some $a_i\in R$.

Since $N_i\lhd H_i$ has finite index for all $i$ then conjugating successively the groups $N_i$ and enlarging $R$ and $F$ if necessary other finite sets we further get that  $B^{\textbf{q}}_c\subseteq N^{b_1}_1
N^{b_2}_2 \cdots N^{b_k}_k F$ for some $b_i\in R$. Thus $B^{\textbf{q}}_c\subset \llangle \mathcal N \rrangle F$, as desired. \end{proof}
 
 \begin{corollary}\label{array3}
   Let $G$ be a hyperbolic group relative to a finite collection of ressidualy finite subgroups $\{H_i\}_{i\in I}$. Let $N_i\lhd H_i$ be finite index normal subgroup for all $1\leq i\leq n$ such that we have the following Dehn filling with Cohen-Lyndon property short exact sequence
   \begin{equation}\label{ses1}1\rightarrow \ast_{i\in I, t\in T_i} N_i^t =\llangle \mathcal N\rrangle \hookrightarrow G\rightarrow G(\mathcal N)\rightarrow 1,\end{equation}
   with $ G(\mathcal N)$  non-elementary hyperbolic and left transversals $T_i \in LT(G, \llangle \mathcal N \rrangle H_i)$.

   \noindent Then one can find a representation $\sigma: G \rightarrow\mathcal U(\mathcal H)$ and an array $\textbf{q}:G\rightarrow \mathcal H$ satisfying: 
   \begin{enumerate}
       \item $\sigma$ is weakly contained in $\lambda_G$.
       \item For every $c>0$, there exist $d_c>0$ and a finite subset $K_c\subset G(\mathcal N)$  such that whenever $g=nh\in G$ with $n\in \ast_{t\in T_i,i\in I} N_i^t$, $h\in G(\mathcal N) $ satisfies $\|\textbf{q}(g)\|<c$ then we must have $\ell_F(n) \leq d_c$ and $h\in K_c$, where $\ell_F(n):= min \{ \sum_i|F_i|\ : n\in \ast_{t\in F_i,i\in I} N_i^t,\ F_i\subseteq T_i\}$.\\
       In particular, for every $c>0$ there are finite sets $T^i_c\subset T_i$, $K_c\subset G(\mathcal N)$ so that $$B^{\textbf{q}}_c\subseteq (\ast_{i\in I, t\in T^i_c} N_i^t) K_c.$$
   \end{enumerate}\end{corollary}

\begin{proof} Consider the array $\textbf{q} :G\rightarrow \mathcal H$ given by Proposition \ref{prop2.2}. Then by the proof of Corollary \ref{array2} one can find finite sets $R,F\subset G$ such that  $B^{\textbf{q}}_c\subseteq N^{b_1}_1
N^{b_2}_2 \cdots N^{b_k}_k F$ for some $b_i\in R$. Using the short exact sequence \eqref{ses1} and the natural action of $G(\mathcal N)$ on the left transversals $T_i$ in the normal closure $\llangle \mathcal N \rrangle=\ast_{i\in I, t\in T_i} N_i^t$ we get the desired conclusion. \end{proof}


Combining  Corollary \ref{array3} with \cite[Corollary 3]{Oza04}, and \cite[Proposition 2.7]{PV12} (see also \cite[Proposition 2.3]{CS11}) we obtain the following consequence to relative bi-exactness:  

\begin{corollary}Let $G$ be a hyperbolic group relative to a finite collection of exact, residually finite subgroups $\{H_i\}_{i\in I}$. Let $N_i\lhd H_i$ be a normal subgroup of sufficiently large finite index such that we have Dehn filling with Cohen-Lyndon property short exact sequence 
\begin{equation}\label{ses3}1\rightarrow \ast_{i\in I, t\in T_i} N_i^t =\llangle \mathcal N\rrangle \hookrightarrow G\twoheadrightarrow G(\mathcal N)\rightarrow 1,\end{equation}
   with $ G(\mathcal N)$ non-elementary hyperbolic and left transversals $T_i \in LT(G, \llangle \mathcal N \rrangle H_i)$.

 \noindent Then $G$ is biexact relative to the subgroups $\{ \ast_{i\in I, t\in F_i} N_i^t \,:\, F_i\subset T_i \text{ finite}\}$.

\end{corollary}

We end this section with the following more general result on relative biexactness for products of relative hyperbolic groups.

\begin{proposition}\label{Prop2.7 new} For every $1 \leq j\leq n,$ let $G_j$ be a group that is hyperbolic relative to a collection of exact, residually finite subgroups $\{H_k^j\}_{k\in I_j}.$ For every $j,k$ assume there exists a finite index normal subgroup $N^j_k\lhd H_k^j$ that yields a nontrivial Dehn filling of $G_j$. Denote by $ \mathcal G_j= \{ *_{k\in I_j, t\in F} (N_{k}^j)^t \; : \; F \subset T_j \; \text{finite} \; \}. $ Consider $G = G_1 \times G_2 \times \dots \times G_n.$ 

Then $G$ is biexact relative to $\{ G_{\hat{j}} \times L \; : \; L \in \mathcal{G}_j, 1\leq j\leq n \}.$ Here we have denoted by $G_{\hat{j}}$ the subgroup of elements $g\in G$ whose $j^{th}$ component is trivial.
    
\end{proposition}

\begin{proof} Arguing as before, we only need to show the existence of a weakly$-\ell^2$ array $\mathbf{q}: G \rightarrow \mathcal O (\mathcal H)$ that is proper with respect to the given family.  Towards this let $\mathbf{q}_i: G_i \rightarrow \mathcal O (\mathcal H_i)$ be the weakly$-\ell^2$ array constructed in Corollary \ref{array3}. One can define an array $\hat{\mathbf{q}}_i: G \rightarrow \mathcal O (\mathcal H_i)$ by letting $\hat{\mathbf{q}}_i(g_1, g_2, \dots, g_n) := \mathbf{q}_i(g_i).$ Using  \cite[Proposition 1.9]{CS11} and \cite[Example 2.7 B)]{CSU11} the tensor product array $\mathbf{q}= \hat{\mathbf{q}}_1 \otimes \hat{\mathbf{q}}_2 \otimes \dots  \otimes \hat{\mathbf{q}}_n : G \rightarrow \mathcal O (\otimes_{i=1}^n \mathcal{H}_i)$ does the job.
\end{proof}

\subsection{Examples of relative hyperbolic groups with specific properties}\label{relhypex} In this subsection, we use the Rips construction due to Belegradek-Osin \cite{BO06} and other small cancellation techniques \cite{Osi06} to build families of relative hyperbolic groups where we can control the normal closure of their peripheral structure and their finite index subgroups. These will serve as examples for our primeness results and computation of factor invariants in the subsequent sections.

\begin{theorem}\label{rips1} 
Let $K$ be any non-elementary hyperbolic group, $Q$ be any finitely presented group, and $H$ be any finitely generated group. Then, one can find a group $H<G $, which is hyperbolic relative to $H$. Moreover, there exists an intermediate normal subgroup $H<N\triangleleft G$ that is a quotient of $K$ and satisfies $G/N \cong Q$. In particular, when $Q$ is infinite we have  $[G:\llangle H \rrangle^G]=\infty.$ In particular, when $K$ and $Q$ have property $(T)$ it follows that $G$ has property $(T)$ as well. 
\end{theorem}

\begin{proof}
    Fix $Q=\langle X | R \rangle$ a finite presentation and let $F(X)$ be a the free group generated by $X$. Let $\phi : F(X) \twoheadrightarrow Q$ be the canonical epimorphism and assume that $\llangle R \rrangle^{F(X)}=ker(\phi).$ Assume that $Z$ and $Y$ are (finite) set of generators for $H$ and $K$, respectively. Let $G_0 = H * F(X) * K$ and notice it is hyperbolic relative to $H$. Next we consider the finite set $$T = \{ xyx^{-1},x^{-1}yx,z \, :\,x \in X \cup Z, y \in Y, z \in Z \} \cup R.$$ Note that $K<G_0$ is a suitable subgroup in the sense of \cite{Osi06}. Thus by \cite[Theorem 2.4]{Osi06} there exists $G$ a quotient of $G_0,$ and words $w_t \in K$, where $t\in T$, such that $\eta : G_0 \rightarrow G = G_0/\llangle tw_t \,|\, t \in T \rrangle^{G
_0}$ with $\eta |_H $ is injective and $G=\eta(G_0)$ is hyperbolic relative to $\eta(H).$ From the definition of T it follows that $N=\eta(K) \triangleleft G$ is a normal subgroup so that $\eta (H) < N.$ Moreover, we can see that $G/N \cong G_0 / \eta^{-1}(\eta (K)).$ \\
However, from the definition, one can see that $$\eta^{-1}(\eta(K)) \geqslant \llangle Z, R, Y \rrangle^{G_0} = ( *_{S \in F(X)}(H*K)^S ) \rtimes \llangle R \rrangle^{F(X)}.$$
Moreover, 
$$\eta^{-1}(\eta(K)) = \llangle tw_t | t \in T \rrangle^{G_0}K \leqslant \llangle Z, R, Y \rrangle^{G_0} $$
Thus, 
$$\eta^{-1}(\eta(K)) = \llangle Z, R, Y \rrangle^{G_0} =( *_{S \in F(X)}(H*K)^S ) \rtimes \llangle R \rrangle^{F(X)}. $$\\
Since $G_0 = (*_{S \in F(X)}(H*K)^S) \rtimes F(X)$ we conclude that
$$G/N \cong G_0 / \eta^{-1}(\eta(K)) \cong F(X) / \llangle R \rrangle^{F(X) }\cong Q.$$
As property (T) passes to quotients and extensions, it follows that when $K$, $Q$ have property (T), then $G$ has property (T) as well.\end{proof}

\vskip 0.03in
For further use, we briefly recall next the $C'(1/6)$-small cancellation theory over a free product of $n$ groups $F = A_1* \dots * A_n$, where each $A_i$ is referred to as a free factor. For every non-trivial element $w \in F$ can be uniquely expressed as a product $w=h_1h_2 \cdots h_k$ called the normal form, where $h_j \in \bigcup_{i=1}^n A_i$ for each $j \in \{ 1, \dots, k \}$, and no consecutive $h_j, h_{j+1}$ belong to the same free factor. Then, the length of the free product is given by $\lvert w \rvert =k$. The normal form of the $w$ is called weakly cyclically reduced if $\lvert w \rvert \leq 1$ or $h_1 \neq h_k^{-1}$. For every $u, v \in F$ with $u = h_1 h_2\cdots h_k$ and $v= g_1 g_2\cdots g_l$, the product $uv$ is referred to as weakly reduced if $h_k \neq g_1^{-1}$. Otherwise, $h_k$ and $g_1$ cancel in the product.

Let $R\subset F$ be a subset of $F$ where each element is represented by a weakly cyclically reduced normal form. Denote the normal closure of $R$ in $F$ by $\llangle R \rrangle$. Let $G$ be the corresponding quotient group
\begin{equation}
G= { F/\llangle R \rrangle.}
\end{equation}

$R$ is symmetrized if every weakly cyclically reduced $r\in R$, conjugate of $r$ and $r^{-1}$ are still in $R$.  A word $b \in F$ is called \emph{a piece} if there exist distinct relators $r_1, r_2 \in R$ in the weakly reduced form $r_1 = bu_1$ and $r_2 = bu_2$. Moreover, we say that $R$ satisfies the $C'(1/6)$ condition over $F$ if $R$ is symmetrized and if each piece $b\in F$ and every relator $r \in R$ with weakly reduced product $r=bu$ implies that $\lvert b \rvert < \frac{1}{6}\lvert r \rvert.$

In this case, we say that the quotient group $G$ is a $ C' (1/6)$–group over the free product $F$.  Using \cite[Corollary V.9.4]{LS77}  we further have that the canonical projection $\eta: F \rightarrow G$ is injective on each factor $A_i$, $1\leq i\leq n$. Moreover, from \cite{Pan99} it follows that $G$ is hyperbolic relative to the collection of the images $\{\eta(A_i)\}_i$,  $1\leq k\leq n$.

\begin{proposition}\label{smallcancoverf} Let $F= A_1\ast \cdots \ast A_n \ast \mathbb F_k$ with $k\geq 1$. Let $\pi: F \rightarrow \mathbb F_k$ be the group homomorphism defined as for every alternating word $w\in \mathbb F_k$, $\pi(w)$ is the word obtained from $w$ after eliminating all its letters from $A_1,\ldots, A_n$. Let $R\subset F$ be any subset satisfying the $C'(1/6)$ condition over $ F$ such that $[\mathbb F_k: \llangle \pi(R)\rrangle^{\mathbb F_k}]=\infty$. Then in the quotient group $G=F/\llangle R\rrangle$ we have that $[G: \llangle \cup_i \eta(A_i)\rrangle^G ]=\infty$. 
\end{proposition}
    
\begin{proof} Follows from definitions. \end{proof}

\vskip 0.08in
We end this section with the following elementary result on commensurable groups which is needed in Section \ref{onesidedpropt}. Two groups $G,H$ are called \emph{commensurable} if they admit finite index subgroups $G_0\leqslant G$ and $H_0\leqslant H$ that are isomorphic, $G_0\cong H_0$. 

\begin{proposition}\label{contnoncommens} Let $\{G_i\}_{i\in I}$ be a continuum family of pairwise non-isomorphic finitely generated groups. Then one can find a continuum subcollection $\{G_j\}_{j\in J\subseteq I}$ consisting of pairwise non-commensurable groups.
    
\end{proposition}

\section{Preliminaries of von Neumann algebras}

\subsection{Deformations associated with arrays}
Let $G$ be a countable, discrete group, and let $\pi:G \rightarrow \mathcal O(\mathcal H)$ be an orthogonal representation. Using the Gaussian construction, \cite{Sin10} $\pi$ gives rise to a nonatomic standard probability measure space $(X, \mu)$ such that $L^{\infty}(X, \mu)$ is generated by a family of unitaries $\{\omega(\xi)\,|\, \xi \in \mathcal H \}$  satisfying:\begin{enumerate}\item [a)] $\omega(0)=1$, $\omega(\xi_1 + \xi_2)=\omega(\xi_1)\omega(\xi_2)$, and $\omega(\xi)^*= \omega(-\xi)$ for all $\xi, \xi_1, \xi_2 \in \mathcal H$.
\item[b)] $\tau(\omega(\xi))= e^{-\|\xi \|^2}$, where $\tau$ is the trace on $L^{\infty}(X, \mu)$ given by integration.\end{enumerate}
Furthermore, there is a pmp action of $G \overset{\hat{\pi}}\ca (X, \mu)$ such that the corresponding $G \overset{\hat{\pi}}\ca L^{\infty}(X, \mu)$ satisfies $\hat{\pi}_g(\omega(\xi))= \omega(\pi_g(\xi))$ for all $g \in G$, and $\xi \in \mathcal H$.
The action $G \overset{\hat{\pi}}\ca (X, \mu)$ is called the \textit{Gaussian action} associated with $\pi$. Next let $G \ca^\rho \mathcal (N, \tau)$ be any $\tau$-preserving action on a von Neumann algebra $\mathcal N$   The corresponding cross-product von Neumann algebra $\tilde \M= (L^{\infty}(X, \mu)\otimes \mathcal N) \rtimes_{\hat\pi\otimes \rho} G$ is called the  \textit{Gaussian dilation} of $\M= \mathcal N \rtimes_\rho G$.  
\\
Let $\mathbf{q}: G \rightarrow \mathcal H$ be an array associated with $\pi$. As in \cite{Sin10,CS11,CSU11}, for each $t \in \mathbb R$, let $V^{\mathbf{q}}_t \in \mathscr U(L^2(X,\mu) \otimes L^2(\mathcal N) \otimes \ell^2(G))$ be defined by 
$$V^{\mathbf{q}}_t(\xi \otimes \delta_g)= \omega(t\mathbf{q}(g))\xi \otimes n\otimes \delta_g \text{ for all } \xi \in L^2(X,\mu), n \in L^2(\mathcal N), \text{ and } g \in G.$$
This path of unitaries is called \textit{Gaussian deformation} of $\mathcal M$ associated with $\mathbf{q}$. For further use,  we collect together some of its basic properties which were  established in the prior works \cite{CS11,CSU11,CSU16,CKP14}.

\begin{proposition}\label{lm1}Under the previous notations, the following properties hold:
  
\begin{enumerate}
    \item (\emph{Popa's transversality property}) For every $x\in \mathcal M$ we have $$\|V_t^{\mathbf{q}}(x)-x\|^2_2\leq 2\|e_{\mathcal M}^{\perp}\circ V_t^{\mathbf{q}}(x) \|^2_2.$$    
 \item (\emph{Assymptotic bimodularity}) For every $x,y\in \mathcal N \rtimes_{\rho, r} G$ we have 
    $$\lim_{t \rightarrow 0} \left ( \sup_{\|\xi\|_2\leq 1}\|xV_t^{\mathbf{q}}(\xi) y-V^{\mathbf{q}}_t(x\xi y)\|_2 \right )=0.$$
    \item For any $S \subseteq (\mathcal M)_1$ the following  infinitesimal properties are equivalent: 
    \vskip 0.02in
    \begin{enumerate}
    \item $\lim_{t\rightarrow 0}\left (\sup_{x\in S} \lVert V_t^{\mathbf{q}}(x) - x \rVert_2\right )= 0$, 
    \item $\lim_{t\rightarrow 0}\left(\sup_{x\in S} \lVert e_{\mathcal M}^{\perp}\circ V_t^{\mathbf{q}}(x) \rVert_2 \right) =0$,  
    \item $\lim_{c \rightarrow \infty}\left(\sup_{x\in S} \lVert \mathbb P_{B_c}(x) -x \rVert_2 \right)= 0$.
    \end{enumerate}
\vskip 0.03in
In (c) we denoted by $B_c=\{ g\in G \,|\, \|q(g)\|\leq c\}$.
    
\vskip 0.02in
    \item If one of (a)-(c) above holds then for every finite subset $F\subset \mathcal M$ we have 

     \begin{equation}\label{bimod2}\lim_{t\rightarrow 0}\left(\sup_{x\in S, a,b\in F} \lVert e_{\mathcal M}^{\perp}\circ V_t^{\mathbf{q}}(a x b) \rVert_2 \right) =0.\end{equation}    
\end{enumerate}

\end{proposition}

\begin{proof} Part (1) is \cite[Lemma 2.8]{CS11}. Part (2) is \cite[Proposition 1.10]{CSU16}. The equivalence (3a) $\Leftrightarrow$ (3b) follows from part (1). The equivalence (3b) $\Leftrightarrow$ (3c) follows from arguments used in \cite{CS11} but we will include some details for readers' convenience. 
Assuming (3b), fix $\frac{1}{2} > \varepsilon > 0$ and $t_\varepsilon>0$ so that for all $0\leq |t|<t_\varepsilon$ and $x\in \mathcal S$ we have
    \begin{equation}\label{thm2eq'}
        \lVert V^{\mathbf{q}}_t(x)-x\rVert _2 \leq \varepsilon ^2.
    \end{equation}
    Let $c_\varepsilon=\frac{\ln (\varepsilon ^{-1})^\frac{1}{2}}{t_\varepsilon}$. For $\mathcal{S}\ni x = \sum x_g u_g$ we have $V^{\mathbf{q}}_t(x)=\sum_g (\omega(t\mathbf{q}(g))\otimes x_g) u_g$ and using \eqref{thm2eq'}, we have 
    \begin{equation*}
        \begin{split}
            \varepsilon ^4 &\geq \lVert V^{\mathbf{q}}_t(x)-x\rVert _2 ^2 = \sum_g \lVert \omega(t\mathbf{q}(g))-1\rVert _2^2 \;\| x_g\|_2 ^2 
            = \sum_g 2(1-e^{-t^2\lVert \mathbf{q}(g)\rVert ^2 }) \| x_g\|^2_2 \\ \geq &\sum _{g\in G \setminus B_{c_\varepsilon}}2(1-e^{-t^2k^2})\| x_g\|_2^2 
            = 2(1-\varepsilon)\sum_{g\in G \setminus B_{c_\varepsilon}}\|x_g \|_2^2 =2(1-\varepsilon)\lVert x - \mathbb P_{B_{c_\varepsilon}}(x)\rVert_2^2. 
        \end{split}
    \end{equation*}

    \noindent Hence $\lVert x - \mathbb P_{B_{c}}(x) \rVert_2^2 \leq \frac{\varepsilon^4}{2(1-\varepsilon)}<\varepsilon^2  \text{ for all }x \in \mathcal{S}$ and $c>c_\varepsilon$, which yields (3c).

\noindent To see the converse, fix $1>\varepsilon>0$ and $c_\varepsilon>0$ so that for every $c>c_\varepsilon$ and $x\in \mathcal S$ we have \begin{equation}\label{ballineq}
    \|\mathbb P_{B_c}(x)-x\|_2\leq \frac{\varepsilon}{2}.
\end{equation}  

\noindent Letting $t_\varepsilon= \frac{\ln((1-\frac{\varepsilon^2}{4})^{-1})^{\frac{1}{2}}}{c_\varepsilon}>0$ and using similar computations and conjunction with \eqref{ballineq}, we see for all $0\leq |t|<t_\varepsilon$ we have  
    \begin{equation*}
        \begin{split}
             &\lVert V^{\mathbf{q}}_t(x)-x\rVert _2 ^2 = \sum_g \lVert \omega(t\mathbf{q}(g))-1\rVert _2^2 \;\| x_g\|_2^2 
            = \sum_g 2(1-e^{-t^2\lVert \mathbf{q}(g)\rVert ^2 }) \| x_g\|_2 ^2 \\ \leq  &\sum _{g\in G \setminus B_{c_\varepsilon}}2\| x_g\|_2^2 
            + \sum _{g\in B_{c_\varepsilon}}2(1-e^{-t_\varepsilon^2c_\varepsilon^2})\| x_g\|_2^2 
            \leq  2\lVert x - \mathbb P_{B_{c_\varepsilon}}(x)\rVert_2^2+ \frac{\varepsilon^2}{2}\leq \varepsilon^2, 
        \end{split}
    \end{equation*}
    showing (3b).

(4) As $F$ is finite, it suffices to show \eqref{bimod2} only for a single fixed pair $a,b\in (\mathcal M)_1$. Fix $\varepsilon>0$. By Kaplanski density theorem one can find $a_\varepsilon, b_\varepsilon \in (\mathcal N \rtimes_{\rho, r}(G))_1$ satisfying $\|a-a_\varepsilon\|_2, \|b-b_\varepsilon\|_2\leq \varepsilon$. Using these estimates in conjunction with triangle inequality and $\mathcal M$-bimodularity of $e^\perp _\mathcal M$ we see for every $x\in \mathcal S$ and $t\in \mathbb R$  we have 

\begin{equation*}\begin{split}
     \| e^\perp_{\mathcal M}\circ V_t^{\mathbf{q}}(axb)\|_2 & \leq 2\varepsilon+ \| e^\perp_{\mathcal M}\circ V_t^{\mathbf{q}}(a_\varepsilon xb_\varepsilon)\|_2 \\
    &\leq 2\varepsilon+ \| e^\perp_{\mathcal M} (V_t^{\mathbf{q}}(a_\varepsilon xb_\varepsilon) - a_\varepsilon  V_t^{\mathbf{q}}(x) b_\varepsilon )\|_2   + \| a_\varepsilon \left(e^\perp _{\mathcal M} \circ  V_t^{\mathbf{q}}(x)\right)  b_\varepsilon)\|_2\\
    &\leq 2\varepsilon+ \| V_t^{\mathbf{q}}(a_\varepsilon xb_\varepsilon) - a_\varepsilon  V_t^{\mathbf{q}}(x) b_\varepsilon \|_2   +\| e^\perp_{\mathcal M} \circ V_t^{\mathbf{q}}(x)\|_2    \end{split}\end{equation*}
Taking supremum over $x\in \mathcal S$ and then limit superior over $t\rightarrow 0$ in the previous inequality and using property (2) above and the hypothesis we get  $$\limsup_{t\rightarrow 0}\left ( \sup_{x\in S}\|e^\perp_{\mathcal M}\circ V_t^{\mathbf{q}}(axb)\|_2\right )\leq 2\varepsilon.$$
As $\varepsilon>0$ was arbitrary, we get the desired conclusion.\end{proof}

For further use, we also recall the following convergence property of the array deformation, which was established in \cite{CS11} is based on a spectral gap argument.
\begin{theorem}\label{ballconv}\cite{CS11}
    Let $G$ be a group and a trace preserving action $G \curvearrowright (\mathcal N, \tau)$ on an amenable von Neumann algebra $\mathcal N$   whose reduced C$^*$-algebra $\mathcal N\rtimes_{\rho, r} G$ is locally reflexive. Assume there is an array $\mathbf{q}: G\rightarrow \mathcal H_\pi$, where $\pi: G\rightarrow \mathscr U(\mathcal H_\pi)$ is a weakly$-\ell^2$ representation. 
    Let $p\in (\mathcal N\rtimes G)$ be a projection and let $\mathcal{A}\subseteq p(\mathcal N \rtimes G)p$ be a von Neumann algebra such that $\A'\cap p(\mathcal N \rtimes G) p$ has no amenable direct summand. Then, 
    \begin{equation*}
       \lim_{t\rightarrow 0 } \left(\sup_{x\in (\mathcal A)_1}\lVert V^{\mathbf{q}}_{t}(x)-x  \rVert _2 \right ) =0.
    \end{equation*}
 
\end{theorem}


\subsection{Ioana-Peterson-Popa deformation for amalgamated free products}\label{afpdef} Below we briefly recall some of the basic properties of the free length deformation associated with the cross-product von Neumann algebra of a free product group action. This was introduced in \cite{IPP05} and played a key role in many of the classification results in von Neuamnn algebras that emerged over the past two decades. 

Let $G= G_1 \ast \cdots \ast G_n$ and let $G \curvearrowright^\sigma \mathcal N$ be  any $\tau$-preserving action on a tracial von Neumann algebra $(\mathcal N,\tau )$. Denote by $\mathcal M = \mathcal N\rtimes_\sigma G$ the corresponding cross-product von Neumann algebra. Now let $\tilde G = (G_1 \ast \mathbb Z_1)\ast\cdots \ast (G_n\ast \mathbb Z_n)$ where  $\mathbb Z_k$ is a copy of the infinite cyclic group generated by $a_k $. Consider the action $\tilde G \ca^{\tilde \sigma} \mathcal N$ where the generators of all the cyclic groups act trivially on $\mathcal N$ and put $\tilde {\mathcal M} = \mathcal N\rtimes_{\tilde\sigma}\tilde G$. Naturally we have $\mathcal M\subset \tilde{\mathcal M}$. For every $1\leq k\leq n$ let $h_k\in \mathcal M$ be the selfadjoint element with spectrum $[-\pi, \pi]$ such that $u_{a_i}= \exp(i h_i)$. For every $t\in \mathbb R$ consider the unitary $u^t_k= \exp(i th_k)\in \mathcal L(\tilde G)$. Next consider the 1-parameter group of automorphisms of $\tilde {\mathcal M}$ defined as    \begin{equation}
    \alpha_t(x )= u_k^t x  u_k^{-t} \text{, for every }x\in \mathcal N \rtimes_{\tilde \sigma} (G_k\times \mathbb Z_k).
\end{equation}
Then one can easily see that $\|\alpha_t(x)-x\|_2\rightarrow 0$, as $t \rightarrow 0$, for every $x\in \mathcal M$.  For every $t\in \mathbb R$  denote by $\rho_t := (\sin(\pi t)/(\pi t))^2\in (0,1)$. Basic computations show that 

 $$E_{\mathcal M}\circ \alpha_t(au_g) = \rho_t^{ |g|} au_g,  \text{ for every }a\in \mathcal N, g\in G.$$
This formula and other similar computations show for every $x\in L^2(\mathcal M)$ we have

\begin{equation}\label{freeprodidentities}\begin{split}
    & \|E_{\mathcal M}\circ \alpha_t(x) \|^2_2 = \sum_{k\in \mathbb N} \rho_t^{ 2 k}\|\mathbb P_{S_k} (x)\|_2^2,\\
    & \| x- \alpha_t(x)\|^2_2=\sum_{k\in \mathbb N}2(1-\rho_t^{k})\|\mathbb P_{S_k}(x)\|_2^2, \text{ and}\\
    &\|\alpha_t(x)-E_{\mathcal M}\circ\alpha_t(x)\|_2^2=\sum_{k\in \mathbb N}(1-\rho_t^{2k})\|\mathbb P_{S_k}(x)\|_2^2.
\end{split}\end{equation}
Above we denoted by $S_k=\{ g\in G \,:\,|g|=k\}$. 

\begin{corollary}\label{convfreedef1} For every $m\in \mathbb N$ we have  
\begin{equation*}\lim_{t\rightarrow 0}\left ( \sup_{\|x\|_2\leq 1}\| \mathbb P_{B_m}(x)- \alpha_t(\mathbb P_{B_m}(x))\|_2\right )=0.
\end{equation*} 
Here we denoted by $B_m=\{ g\in G \,|\, |g|\leq m\}$.
\end{corollary}

\begin{proof} As $0\leq \rho_t < 1$, using the second identity in \eqref{freeprodidentities}, for every $x\in L^2(\mathcal M)$ we have 
    \begin{equation*}
\|\mathbb P_{B_m}(x)- \alpha_t(\mathbb P_{B_m}(x))\|_2^2 =\sum_k 2(1-\rho^k_t)\|\mathbb P_{S_k\cap B_m}(x)\|_2^2 \leq 2(1-\rho_t^m)\sum_{k\leq m}\|\mathbb P_{S_k}(x)\|_2^2\leq 2(1-\rho_t^m)\|x\|_2^2.        
    \end{equation*}
    Since $\rho_t\rightarrow 1$ as $t \rightarrow 0$, the previous inequality yields the desired conclusion.
\end{proof}

\begin{theorem}\label{innerfreedef}\cite[Theorem 4.3]{IPP05}(see also \cite[Theorem 4.2]{CH08})
Under the previous assumptions let $p\in \mathcal M$ be a projection and let $\mathcal{A}\subseteq p\mathcal M p$ be a von Neumann subalgebra. Suppose there exists $v \in \tilde {\mathcal M}$ such that  $v \alpha_t(x)=xv$ for all $x\in  \mathcal A$. Then one can find $1\leq j\leq n$ such that $\mathcal A \prec \mathcal N\rtimes G_j$. 
\end{theorem}


\subsection{Popa's intertwining techniques}\label{popaint}
In \cite{Pop03} introduced the following powerful criterion for the existence of intertwining between von Neumann algebras.

\begin{theorem}\emph{\cite{Pop03}} \label{corner} Let $( \mathcal{M},\tau)$ be a  tracial von Neumann algebra and let $ \mathcal{P},  \mathcal{Q}\subseteq  \mathcal{M}$ be (not necessarily unital) von Neumann subalgebras. 
	Then the following are equivalent:
	\begin{enumerate}
		\item There exist projections $ p\in    \mathcal{P}, q\in    \mathcal{Q}$, a $\ast$-homomorphism $\Theta:p  \mathcal{P} p\rightarrow q \mathcal{Q} q$  and a partial isometry $0\neq v\in  \mathcal{M} $ such that $v^*v\leq p$, $vv^*\leq q$ and $\Theta(x)v=vx$, for all $x\in p  \mathcal{P} p$.
		\item For any group $ \mathcal G \subset \mathscr U( \mathcal{P})$ such that $ \mathcal G''=  \mathcal{P}$ there is no net $(u_n)_n\subset \mathcal G$ satisfying $\|E_{  \mathcal{Q}}(xu_ny)\|_2\rightarrow 0$, for all $x,y\in   \mathcal{M}$.
		\item There exist finitely many $x_i, y_i \in  \mathcal{M}$ and $C>0$ such that  $\sum_i \| E_{\mathcal{Q}}(x_i u y_i) \|^2_2 \geq C$ for all $u\in \mathscr U (\mathcal{P})$.
	\end{enumerate}
\end{theorem} 
\vskip 0.02in
\noindent If one of the equivalent conditions from Theorem \ref{corner} holds, one says \emph{a corner of $ \mathcal{P}$ embeds into $ \mathcal{Q}$ inside $ \mathcal{M}$}, and writes $ \mathcal{P}\prec_{ \mathcal{M}} \mathcal{Q}$. If we moreover have that $ \mathcal{P} p'\prec_{ \mathcal{M}} \mathcal{Q}$, for any projection  $0\neq p'\in  \mathcal{P}'\cap 1_{ \mathcal{P}}  \mathcal{M} 1_{ \mathcal{P}}$, then one writes $ \mathcal{P}\prec_{ \mathcal{M}}^{\rm s} \mathcal{Q}$.
\vskip 0.05in

It is not difficult to see the intertwining relation is not transitive in general. However, transitivity can be insured if slightly stronger assumptions are imposed on the second intertwining, \cite{IPP05, Vae08}.

\begin{theorem}\label{transint} Let $\mathcal P\subseteq \mathcal M$ and $\mathcal Q \subseteq \mathcal R \subseteq \mathcal M$ be von Neumann algebras. Let  $p\in \mathcal P $, $
r\in \mathcal R$ be projections, $0\neq v\in  \mathcal{M} $ be a partial isometry and $\Theta:p  \mathcal{P} p\rightarrow r \mathcal{R} r$ be a $\ast$-homomorphism such that $v^*v\leq p$, $vv^*\leq r$, $s(E_{\mathcal R} (vv^*))= r$, and $\Theta(x)v=vx$, for all $x\in p  \mathcal{P} p$. If $\Theta(p\mathcal P p)\prec _\mathcal R \mathcal Q$ then $\mathcal P \prec_\mathcal M \mathcal Q$.
\end{theorem}
\begin{proof} Let $\mathcal S:=\Theta( p\mathcal P p)$.  Thus one can find projections $e \in \mathcal S, q \in \mathcal Q$, a partial isometry $0\neq y \in e\mathcal R q$ and a $\ast$-isomorphism onto its image $\Phi: e\mathcal Se \rightarrow q\mathcal Q q$ so that 
\begin{equation}\label{eq6}
   \Phi (x)y = yx \ \text{for all} \; x \in e\mathcal S e.
\end{equation}
Let $f \in p\mathcal P p$ be a projection with $\Theta (f) =e$. Using \eqref{eq6} and the hypothesis intertwining relation, for all $ x \in f\mathcal P f$, we have \begin{equation}\label{intertwining10} y vx = y\Theta(x)v = y\Theta(x)v = \Phi(\Theta(x) )yv.\end{equation} Next, we argue that $yv \neq 0$. Assuming otherwise, we have $yvv^* =0$. Applying the expectation and using $\mathcal Q \subseteq \mathcal R$ we get $ 0 = E_{\mathcal R}(yvv^*)=yE_{\mathcal R}(vv^*)$. Thus  $0=y s (E_{\mathcal R}(vv^*))=yr=y$, a contradiction. If $w$ is the polar decomposition $yv$ then \eqref{intertwining10} implies that $w x = \Phi\circ \Theta(x)w, \text{ for all} \; x \in f\mathcal P f$ and hence $\mathcal{P} \preceq_{\mathcal M} \mathcal Q$.\end{proof}

We end this subsection with the following basic intertwining result which is the crossed-product von Neumann algebraic counterpart of \cite[Lemma 2.2]{CI17}. Since the proof is essentially the same with the one in \cite{CI17} we will omit it. 
\begin{lemma}[Lemma 2.2 in \cite{CI17}]\label{intsubgroups} Let $H,K<G$ be groups, let $G \curvearrowright \mathcal N$ be a trace preserving action on a tracial von Neumann algebras $\mathcal N$. Denote by $\mathcal M = \mathcal N\rtimes G$ the corresponding crossed product von Neumann algebra. If $\mathcal L(H)\prec_{\mathcal M}\mathcal N \rtimes K$ then one can find $g\in G$ such that $[H: H\cap gKg^{-1}]<\infty$.    
\end{lemma}

\subsection{Quasinormalizers}
 Following \cite[Definition 4.8]{Pop99}, for a given inclusion $\mathcal A \subseteq \M$ of finite von Neumann algebra we define the quasi-normalizer  $\mathcal {QN}_{\mathcal M}(\mathcal A)$ as the set of all elements $x\in \M$ for which there exist $x_1,\ldots,x_n\in \M$ such that $\mathcal A x \subseteq \sum_i x_i \mathcal A$ and $x\mathcal A \subseteq \sum_i \mathcal A x_i$. Also the one-sided quasi-normalizer $\mathcal {QN}^{1}_{\mathcal M}(\mathcal A)$ is defined as the set of all elements $x\in \M$ for which there exist $x_1,\ldots,x_n\in \M$ such that $\mathcal A x\subseteq \sum_i x_i \mathcal A$, \cite{FGS10}.

\noindent We now record some formulas on compressions of von Neumann algebras generated by quasi-normalizers.

\begin{proposition}\label{quasinormcompress}\emph{\cite{Pop03,FGS10}}\label{QN1}
Let $\mathcal A\subset \M$ be tracial von Neumann algebras. For any projections $p\in\mathcal A$, $p'\in \mathcal A '\cap \mathcal M$ the following hold:
\begin{enumerate}
    \item ${ \mathcal{QN}^{1}_{p\M p}}(p\mathcal A p)''=p { \mathcal{QN}^{1}_{\mathcal M}}(\mathcal A)''p$. 
    
    \item ${ \mathcal{QN}_{pp'\M pp'}}( p\A p p')''=pp'{ \mathcal{QN}_{\M}}(\mathcal A)''pp'$
\end{enumerate}

\end{proposition}

\begin{definition}
Related to this we define inductively a tower of one-sided normalizers von Neumann algebras for all ordinals $\beta,$ as follows: Let $\mathcal{A}^0 = \mathcal{A}.$ Now if $\beta$ is a successor ordinal then let $\mathcal{A}^{\beta}\subseteq \M $ be the von Neumann subalgebra generated by $\mathcal{QN}_{\mathcal M}^{1}(\mathcal{A}^{\beta - 1})$. Notice that $\mathcal{A}^{\beta -1} \subseteq \mathcal{A}^{\beta}.$ If $\beta$ is a limit ordinal then let $\mathcal{A}^{\beta} = \overline{\bigcup_{\gamma < \beta}\mathcal{A}^{\gamma}}^{\rm SOT}.$ Now let $\alpha$ be the first ordinal when the chain $(\mathcal{A}^{\beta})_{\beta}$ stabilizes, i.e. $\mathcal{A}^{\alpha}=\mathcal{A}^{\alpha +1}.$ Then denote by $\overline{\mathcal {QN}}_{\M}^{1}(\mathcal{A})= \mathcal{A}^{\alpha}.$
\end{definition}


For further use we record the following result on control of quasinormalizers which follows from various prior works on (relative) mixing bimodules \cite{IPP05}, \cite[Appendix A]{Bou14}. For the sake of completeness, we inlcude a brief argument based on the proof of \cite[Theorem 1.2.1]{IPP05}.

\begin{theorem}\label{quasinormalizercontrol}
    Let $H<G$ be groups such that $H$ is almost malnormal in $G$.  Assume $G \curvearrowright \mathcal N$ is a trace-preserving action and let $\mathcal M=\mathcal N \rtimes G$. Let $p\in \mathcal N \rtimes H$ be a projection and let $\mathcal Q \subseteq p (\mathcal N\rtimes H) p$ be a von Neumann subalgebra such that $\mathcal Q \nprec \mathcal N$. \\ Then $\overline{\mathcal {QN}}_{p\mathcal M p}^1(\mathcal Q)'' \subseteq p (\mathcal N \rtimes H )p$. In particular, $ \mathcal Q'\cap p\mathcal M p, \mathcal N_{p\mathcal M p}(\mathcal Q)''\subseteq p(\mathcal N \rtimes H)p$.
    \end{theorem}

\begin{proof} We start by noticing we only need to prove that $\mathcal {QN}_{p\mathcal M p}^1(\mathcal Q) \subseteq p (\mathcal N \rtimes H )p$. We show this by closely following the proof of \cite[Theorem 1.2.1]{IPP05}.

Towards this, fix  $\xi \in \mathcal{QN}^1_{p\mathcal M p}(Q)$ for which there exist $\xi_1,\ldots,\xi_n \in p\mathcal M p$ so that  $\mathcal Q \xi \subseteq \sum_i \xi_i (\mathcal N \rtimes H)$ and let $q \in \mathbb{B}(L^2(\M))$ be the orthogonal projection onto $\overline{\Q \xi (\Nn \rtimes H)}^{\|\cdot\|_2}$, then $q \in \Q' \cap p \langle \M, e_{\Nn \rtimes H} \rangle p$ and $\mathrm{Tr} (q) < \infty$.  To prove $\xi \in L^2(\Nn \rtimes H)$ it will suffice to prove $q\leq e_{\Nn \rtimes H}$.  Moreover, following the proof of \cite[Theorem 1.2.1]{IPP05}, it will be sufficient to prove the following claim:

\begin{claim} For every $x_1, \ldots, x_m \in \M \ominus (\Nn \rtimes H)$ and every $\varepsilon >0$, there is $u\in \sU(\Q)$ such that $\lVert E_{\Nn\rtimes H}(x_i u x_j^*) \rVert_2 < \varepsilon$ for any $1\leq i,j \leq m$.
\end{claim}
Since each $x_i$ can be approximated in $\lVert \cdot \rVert_2$ by some $\tilde{x}_i \in \mathrm{span}_\Nn \{u_g : g\in G\setminus H\}$ with $\lVert \tilde{x}_i \rVert_\infty \leq 2\lVert x_i \rVert_\infty $ we can assume from the beginning that $x_i \in \mathrm{span}_\Nn \{u_g : g\in G\setminus H\}$.  Let $F = \cup_{i=1}^m \{ g, g^{-1} \in G: E_\Nn(x_iu_g^*) \neq 0 \}$, so $F=F^{-1} \subset G\setminus H$ is finite and $\mathrm{span} \{\Nn u_g : g\in F\}$ contains all $x_i,x_i^*$.  

Since $H$ is almost malnormal in $G$, for each $g_1,g_2 \in F$, the subset $H_{g_1,g_2} := \{ h \in H : g_1 h g_2^{-1} \in H \} = H \cap g_1^{-1} H g_2 $ is finite and so is $H_0 = \cup_{g_1,g_2 \in F} H_{g_1,g_2}$.  Let $\sK = \overline{\mathrm{span}}^{\lVert \cdot \rVert_2} \{ \Nn u_h : h \in H_0\}$ which is a finitely generated left $\Nn$-module, and let $\mathbb P_\sK : L^2(\mathcal M) \rightarrow\sK$ be the orthogonal projection onto $\sK$. For any $g_1,g_2 \in F$ and $u = \sum_{h\in H} c_h u_h \in \Nn \rtimes H$ we have $$E_{\Nn \rtimes H} (u_{g_1} u u_{g_2}^*) = E_{\Nn \rtimes H} (u_{g_1} \sum_{h\in H_{g_1,g_2}} c_h u_h u_{g_2}^*) = E_{\Nn \rtimes H} (u_{g_1} \mathbb P_\sK (u) u_{g_2}^*) $$
and so $E_{\Nn \rtimes H} (x_i u x_j^*) = E_{\Nn \rtimes H} (x_i P_\sK (u) x_j^*)$ for all $1\leq i,j \leq m$.  Since $\Q \nprec \Nn$, by Theorem \ref{corner} we can find a net $(u_n)_n \subset \sU(\Q)$ such that $\lVert \mathbb P_{\sK} (u_n) \rVert_2 \rightarrow 0$. Thus using the previous formula we obtain $$\lVert E_{\Nn\rtimes H} (x_i u_n x_j^*) \rVert_2 = \lVert E_{\Nn\rtimes H} (x_i P_\sK (u_n) x_j^*) \rVert_2 \leq \max_{i} \lVert x_i \rVert^2_\infty \lVert P_\sK (u_n) \rVert_2 \rightarrow 0.$$
This yields the desired claim.
\end{proof}

\begin{proposition}\label{nonintrhg} Let $G$ be hyperbolic relative to $\{H_1,\ldots, H_n\}$ and let $G \ca \mathcal A$ be trace preserving action on a tracial von Neumann algebra $\mathcal N$ and let $\mathcal M =\mathcal N \rtimes G$ be the corresponding crossed product von Neumann algebra.  Then for any $i\neq j$, any nonzero projection $p\in \mathcal N\rtimes H_i$ and any von Neumann  subalgebra $\mathcal B \subseteq p(\mathcal N\rtimes H_i)p$ with $\mathcal B \nprec_{\mathcal N \rtimes H_i} \mathcal N$ we must have that $\mathcal B\nprec_{\mathcal M} \mathcal N \rtimes H_j$.
    
\end{proposition}

\begin{proof} Assume by contradiction that $\mathcal B\prec_{\mathcal M} \mathcal A \rtimes H_j$. By Popa's intertwining technique one can find $x_1,\ldots, x_n, y_1,\ldots, y_n\in \mathcal M$ and $C>0$ such that for all $b\in \mathscr U(\mathcal B)$ we have \begin{equation}\label{int1}
    \sum^n_{k=1} \|E_{\mathcal N \rtimes H_j}(x_i by_i)\|_2^2\geq C.
\end{equation}

Using $\|\cdot\|_2$-approximations of $x_i$'s and $y_i$'s and increasing $n$ and decreasing $C$, if necessary, we can assssume without any loss of generality that $x_k=u_{g_k}$ and $y_k=u_{h_k}$ for $g_1,\ldots, g_n,h_1,\ldots, h_n \in G$. Since $\mathcal B \subseteq \mathcal N\rtimes H_i$ relation \eqref{int1} further implies that $\sum^n_{k=1} \|\mathbb P_{ H_i \cap g^{-1}_kH_jh^{-1}_k}(b)\|_2^2\geq C$, for all $b\in \mathscr U(\mathcal B)$. However, \cite[Proposition 4.33]{DGO17} implies that every $H_i \cap g^{-1}_kH_jh^{-1}_k$ is finite and therefore, by the Popa's intertwining technique, we further get $\mathcal B \prec_{\mathcal N \rtimes H_i} \mathcal N$, a contradiction. \end{proof}

    

\section{Structural results for von Neumann algebras associated with  relative hyperbolic groups}

\subsection{Relative solidity results}
The first part of the section is devoted to the proofs of the main relative solidity results, Theorems \ref{relsol1}-\ref{relsol3}. Towards this, we first prove the following more general result.
\vskip 0.05in

\begin{theorem}\label{mrelsol1}
    For every $1 \leq j \leq n,$ let $G_j$ be a group that is hyperbolic relative to a family of exact residually finite subgroups  $\{ H_k^j\}_{k\in I_j}.$ Denote by $G = G_1 \times \dots \times G_n$ and let $G \curvearrowright \mathcal N$ be an action on an tracial amenable von Neumann algebra. Denote by $\mathcal M = \mathcal N\rtimes G$ the cross-product von Neumann algebra. Fix $p \in \mathcal M$ a nonzero projection and let $\mathcal A \subseteq p \mathcal M p$ be a von Neumann subalgebra such that the commutant $\mathcal A^{\prime} \cap p\mathcal M p$ has no amenable direct summand. 
    
    \noindent Then one can find $1 \leq j \leq n, s \in I_j$ such that $\mathcal A \preceq_{\mathcal M} \mathcal N \rtimes (G_{\hat{j}} \times H_s^j).$
\end{theorem}
    \begin{proof}
Let $\mathbf{q}$ be a weakly$-\ell^2$ array as in Proposition \ref{Prop2.7 new} and let $(V_t^{\mathbf{q}})_{t\in \mathbb R}$ be the corresponding deformation. Since $\mathcal A^{\prime} \cap p\mathcal M p$ has no amenable direct summand, using Proposition \ref{lm1} and Theorem \ref{ballconv} we have that:
\begin{equation}\label{thm 4.1 eq1} 
    \lim_{t\rightarrow 0} \left( \sup_{x\in (\mathcal A)_1} \lVert e_{\mathcal M}^\perp \circ  V_t^{\mathbf{q}} (x) \rVert_2\right) =0,
\end{equation}
and, equivalently
$$\lim_{m \rightarrow \infty} \left( \sup_{x\in (\mathcal A)_1} \lVert x - \mathbb P_{B_m}(x) \rVert_2 \right) = 0.$$
Fix $\varepsilon > 0.$ Then there exists $m_{\varepsilon} \in \mathbb N$ such that for all $m \geq m_{\varepsilon}$ we have 
\begin{equation}\label{thm 4.1 eq 2}
    \lVert x - \mathbb P_{B_m} (x) \rVert_2 \leq \varepsilon, \; \text{for every } \; x \in (\mathcal A)_1.
\end{equation}

\noindent This together with Corollary \ref{array3} and Proposition \ref{Prop2.7 new} imply that there exist a maximal set $F\subset \{1,...,n\}$ and finite subsets $A_j\subset  I_j$ for $j\in F$ and $F^j_s\subset T_j$ for $s\in A_j$ which are of minimal cardinality satisfying 
\begin{equation}\label{thm 4.1 eq3}
 \mathcal A \preceq \Nn \rtimes \left( G_{\hat{F}} \times \varprod_{j\in F}\left( *_{s\in A_j, t\in F^j_s }(N_s^j)^t\right) \right) =: \mathcal Q.   
\end{equation}

Observe that if there is $j$ so that $\lvert A_j \rvert =1$ and $\lvert F^j_s \rvert =1$ for $s\in A_j,$ then the desired conclusion already follows.

\vskip 0.03in
In the remaining part, we assume that for all $j$ we have $\sum_{s\in A_j} \lvert F^j_s \rvert \geq 2,$ and we will show this leads to a contradiction.

By \eqref{thm 4.1 eq3} there exist projections $a\in \mathcal A, q^{\prime} \in \mathcal{Q}$ and a $*-$isomorphism onto its image $\Theta: a\mathcal A a \rightarrow \mathcal P := \Theta (a\mathcal A a) \subseteq q^{\prime} \mathcal Q q^{\prime}$ such that \begin{equation}\Theta(x) v = vx\text{, for all }x\in a\mathcal A a.\end{equation} 

This intertwining relation entails that $v^*v \leq a, vv^* \leq q^{\prime
}$ and $vv^* \in \mathcal P'\cap (q' \mathcal M q')$.

\noindent Moreover, it also imply that $e_{\mathcal M}^{\perp} \circ V_t^{\mathbf{q}} (\Theta(x)vv^*) = e_{\mathcal M}^{\perp} \circ V_t^{\mathbf{q}} (vxv^*),$ for all $x\in a\mathcal A a,$ and using part 4) in Proposition \ref{lm1} and \eqref{thm 4.1 eq1} we get  
$$\lim_{t\rightarrow 0}\left (\sup_{y\in (\mathcal P)_1}\lVert e_{\mathcal M}^{\perp} \circ V_t^{\mathbf{q}} (yvv^*)\rVert_2\right ) = 0.$$
Using the asymptotic bimodularity once again and basic approximations, for every $\varepsilon > 0,$ there are $t_\varepsilon >0$, a finite set of left cosets $K_{\varepsilon}$ of  $G$ over ${G_{\hat{F}} \times \varprod_{j\in F}\left( 
*_{s\in A, t\in F_s}(N_s^j)^t \right)}$ and $x_h \in \mathcal Q$ with $q'x_h=x_h$, for every $h\in K_{\varepsilon}$ such that for every $t<t_\varepsilon$ we have 
\begin{equation*}
    \begin{split}
        &\lVert vv^* - \sum_{h\in K_{\varepsilon}}x_h h \rVert_2  \leq \varepsilon, \text{ and}  \\
        & \sup_{y\in (\mathcal P)_1}\left( \sum_{h\in K_{\varepsilon}} \lVert e_{\mathcal{M}}^{\perp} \circ V_t^{\mathbf{q}} (y x_h ) \rVert_2^2 \right)  \leq \varepsilon \sum_{h\in K_\varepsilon} \|x_h\|^2_2.
    \end{split}
\end{equation*}

\noindent Thus shrinking $t_{\varepsilon}>0$ if necessary and using part (3) in Proposition \ref{lm1} one can find $m\in \mathbb N$ such that 
\begin{equation}\label{thm 4.1 eq4}
   \sup_{y\in (\mathcal P)_1} \left( \sum_{h\in K_{\varepsilon}}\lVert yx_h - \mathbb P_{B_m}(yx_h) \rVert_2^2 \right) \leq \varepsilon \sum_{h\in K_\varepsilon} \|x_h\|^2_2.
\end{equation}
Using Corollary \ref{array3}, ${B_m}$ satisfy
$${B_m} = \bigcup_{j\in \hat{F}} \left(K_m^j \times G_{\hat{F}\setminus j} \times \varprod_{i\in F}\left( *_{s\in A_i, t\in F^i_s}(N_s^i)^t \right) \right )\cup \bigcup_{j\in F} \left(  G_{\hat{F}} \times K_m^j \times \varprod_{k\in F\setminus \{j\}}\left( *_{s\in A_k, t\in F^k_s}(N_s^k)^t \right)\right).$$
where \begin{enumerate}\item for every $j\in \hat{F}$ we have  $K_m^j=L_m^j S_m^j\subset G_j$, where $S_m^j$ is a finite set of left cosets of $G_j$ over $\llangle \bigcup_s N_s^j \rrangle^{G_j}$ and  $L_m^j \subset \llangle \bigcup_s N_s^j \rrangle^{G_j} = *_{t \in T_j, s\in I_j}(N_s^j)^t$ is a set of words with uniformly bounded alternating free product length.

\item for all $j \in F$ we have $K_m^j= L_m^j S_m^j\subset \llangle \bigcup_s N_s^j \rrangle^{G_j} =*_{s\in I_j, t \in T_j}(N_s^j)^t$ where  $S_m^j$ is a finite set of cosets representatives of $*_{s\in I_j, t \in T_j}(N_s^j)^t $ over $ *_{s\in A_j, t \in F^j_s}(N_s^j)^t$ and $L_m^j \subset *_{s\in A_j, t \in F^j_s}(N_s^j)^t$ is a set of words with uniformly bounded alternating free product length. 
\end{enumerate}

\noindent To simplify the writing, for every $j\in F$ consider the following notations: \begin{equation}\begin{split}
R_j := *_{s\in A_j, t \in F^j_s}(N_s^j)^t ,\quad  R_F := \varprod_{j\in F}R_j\\
\Tilde{G} = \varprod_{j=1}^n \Tilde{G_j} \text{ where } \Tilde{G_j}= \left\{ \begin{array}{rcl} 
R_j & \mbox{for} & j\in F \\ 
G_j & \mbox{for} & j\in \hat{F}. \end{array}\right.
\end{split}\end{equation}

Therefore, \eqref{thm 4.1 eq4} and Cauchy-Schwarz inequality imply that 
\begin{equation*}
    \begin{split}
        \sum_h \lVert yx_h \rVert_2^2 & \leq \sum_h \lVert \mathbb P_{\bigcup_{j\in \hat{F}}K_m^j \times G_{\hat{F}\setminus j} \times R_F \cup \bigcup_{j\in F} G_{\hat{F}}\times K_m^j \times R_{F\setminus j}}(yx_h) \rVert_2^2 +\varepsilon \sum_{h} \|x_h\|^2_2\\
        & \leq n\sum_{h} \left( \sum_{j\in \hat{F}} \lVert \mathbb P_{K_m^j \times G_{\hat{F}\setminus j}\times R_F}(yx_h) \rVert_2^2 + \sum_{j\in F} \lVert \mathbb P_{G_{\hat{F}}\times K_m^j \times R_{F\setminus j}}(yx_h) \rVert_2^2  \right)+\varepsilon \sum_{h} \|x_h\|^2_2\\
        & = n \sum_{h} \left( \sum_{j\in \hat{F}} \sum_{f\in S_m^j} \lVert \mathbb P_{L_m^j \times G_{\hat{F}\setminus j}\times R_F}(yx_hf^{-1}) \rVert_2^2 + \sum_{j\in F} \sum_{f\in S_m^j} \lVert \mathbb P_{G_{\hat{F}}\times L_m^j \times R_{F\setminus j}}(yx_hf^{-1}) \rVert_2^2  \right)+\\&\quad 
        + \varepsilon \sum_{h} \|x_h\|^2_2.
    \end{split}
\end{equation*}
Therefore, for every $y\in \mathscr U(\mathcal P)$ we have
\begin{equation*}
    \begin{split}
        & \sum_{h} \left( \sum_{j\in \hat{F}} \sum_{f\in S_m^j} \lVert yx_h f^{-1} - \mathbb P_{L_m^j \times G_{\hat{F}\setminus j}\times R_F}(yx_hf^{-1}) \rVert_2^2 + \sum_{j\in F} \sum_{f\in S_m^j} \lVert yx_h f^{-1} - \mathbb P_{G_{\hat{F}}\times L_m^j \times R_{F\setminus j}}(yx_hf^{-1}) \rVert_2^2 \right)\\
        & \leq \left( \sum_{h} \left( \sum_{j=1}^n \sum_{f\in S_m^j} \Vert yx_h f^{-1} \rVert_2^2 \right) - \frac{1}{n} \sum_{h}\lVert yx_h \rVert_2^2 \right) + \varepsilon \left( \sum_h \lVert x_h \rVert_2^2 \right) \\
        & \leq \left( \frac{\sum_{j=1}^n \lvert S_m^j \rvert -\frac{1}{n} + \varepsilon}{\sum_{j=1}^n \lvert S_m^j \rvert} \right) \sum_h \sum_j \sum_{f\in S_m^j}\lVert yx_hf^{-1} \rVert_2^2.
    \end{split}
\end{equation*}
Denote by $C:= \left( \frac{\sum_{j=1}^n \lvert S_m^j \rvert -\frac{1}{n} + \varepsilon}{\sum_{j=1}^n \lvert S_m^j \rvert} \right),$ and notice that $C<1$ for well chosen $\varepsilon.$\\

Consider the finite sets
$$\mathcal T_j = \{ x_h f^{-1} \; : \; f \in S_m^j,\; h\in K_{\varepsilon} \}$$
Thus, the previous inequality implies that for every $y\in \mathscr U(\mathcal P)$, 
\begin{equation}\label{thm 4.1 eq 5}
    \sum_j \sum_{r\in \mathcal T_j} \lVert yr - \mathbb P_{L_m^j \times \Tilde{G_j}} (yr) \rVert_2^2 \leq C\sum_{j, r\in \mathcal T_j} \lVert yr \rVert_2^2.
\end{equation}
On $\mathcal N \rtimes \llangle \bigcup H_s^j \rrangle^{G_j}  =  \mathcal N \rtimes  (*_{t\in \mathcal T_j, s} (H_s^j)^t )$ we consider the free product length deformation $(\alpha_t^j)_{t\in \mathbb R}$ from Subsection \ref{afpdef}. Next, we see that

\begin{equation*}
    \begin{split}
        &\sum_{j, r\in \mathcal T_j} \lVert yr - I \otimes \alpha_t^j (yr) \rVert_2^2  
        \leq \sum_{j, r \in \mathcal T_j} ( \lVert yr - \mathbb P_{L_m^j \times \Tilde{G}_{\hat j}}(yr) - I \otimes \alpha_t^j \left(yr - \mathbb P_{L_m^j \times \Tilde{G}_{\hat j}}(yr) \right) \rVert_2
        +\\
        & + \| \mathbb P_{L_m^j \times \Tilde{G}_{\hat j}}(yr) - I \otimes \alpha_t^j \left( \mathbb P_{L_m^j \times \Tilde{G}_{\hat j}}(yr) \right) \|_2  )^2.
    \end{split}
\end{equation*}
Letting $t$ small enough and using Corollary \ref{convfreedef1} and the second identity in \eqref{freeprodidentities}  we get that the last term above is smaller or equal than 
$$\leq \sum_{j,r\in \mathcal T_j} 2 \lVert yr - \mathbb P_{L_m^j \otimes \Tilde{G_j}}(yr) \rVert_2^2 + \varepsilon.$$
Then, using \eqref{thm 4.1 eq 5} this is further smaller or equal than
$$\leq \sum_{j, r\in \mathcal T_j} 2C \lVert yr \rVert_2^2 + \varepsilon, \text{ for every} \; y\in \mathscr{U}(P).$$
Now consider the von Neumann algebra 
$$\Tilde{\mathcal M} := \oplus_{j,r} \mathcal M  \supset \oplus_{j,r} \mathcal P \supset \text{d}(\mathcal P) = \{ y \oplus \dots \oplus y =: \text{d} (y) \; | \; y \in \mathcal P \}.$$
Also let $\oplus_{j,r} r= :\tilde r$. Then, the previous inequality can be restated as 
\begin{equation}\label{thm 4.1 eq 6}
\begin{split}
   & \lVert \text{d}(y) \tilde r  -  \left( \oplus_{j,r} I \otimes \alpha_t^j \right)  \left( \text{d}(y) \tilde r \right) \rVert_2^2 \leq 2C \lVert \text{d}(y)  \tilde r  \rVert_2^2 + \varepsilon, 
\end{split}\end{equation}
and therefore, we further have 
 \begin{equation}  
 \begin{split}
    &(2-2C)\lVert \text{d}(y)  \tilde r) \rVert_2^2 -\varepsilon \leq 2 Re \tau^{\oplus_{j,r}} \left((\oplus_{j,r} I \otimes \alpha_t^j \left( \text{d}(y) \tilde r\right ) )\left( \text{d}(y) \tilde r \right)^* \right)  \\
    & = 2 Re \tau^{\oplus_{j,r} }\left( \text{d}(y)^* (\oplus_{j,r} I \otimes \alpha_t^j) \left( \text{d}(y) \right) \left( \oplus_{j,r} I \otimes \alpha_t^j \right) (\tilde r) \tilde  r^* \right)
\end{split}
\end{equation}
for all $y\in \mathscr{U}(\mathcal P).$\\
Let $\eta \in \overline{co}^w \{ \text{d}(y)^* (\oplus_{j,r} I \otimes \alpha_t^j)\left( \text{d}(y) \right) \; | \; y \in \mathscr{U}(\mathcal P) \}$ be the unique element of minimal $\lVert \cdot \rVert_2.$ Then inequality \eqref{thm 4.1 eq 6} shows that $\eta \neq 0.$ Moreover, uniqueness of $\eta$ entails that $\text{d}(y)^* \eta (\oplus_{r,j} I \otimes \alpha_t^j \left( \text{d}(y) \right)) = \eta \;\; \text{for every} \; y\in \mathscr{U}(\mathcal P)$ and therefore 
$$\eta (\oplus_{r,j} I \otimes \alpha_t^j )\left( \text{d}(y) \right) = \text{d}(y) \eta \; \; \text{for all} \; y \in \mathcal P.$$
As $\eta \neq 0$ there are $r,j$ such that the component $\eta_{r,j} \neq 0$ and also 
$$\eta_{r,j} I \otimes \alpha_t^j(y) = y \eta_{r,j} \; \; \text{for all} \; y \in \mathcal P.$$
Then Theorem \ref{innerfreedef} further implies there are $j_0 \in J$ and $s_0 \in A_{j_0}$ such that  $\mathcal P \preceq_{\mathcal Q} \mathcal N \rtimes \left( G_{\hat{F}} \times \varprod_{j\in F\setminus\{j_0\}}\left( *_{s\in A_j, t\in F^j_s }(N_s^j)^t\right)  \times N^{j_0}_{s_0}\right) $. 
However, this together with Theorem \ref{transint}  entail that $\mathcal A \preceq \mathcal N \rtimes \left( G_{\hat{F}} \times \varprod_{j\in F\setminus\{j_0\}}\left( *_{s\in A_j, t\in F^j_s }(N_s^j)^t\right)  \times N^{j_0}_{s_0}\right) $, contradicting the minimality of the sets $A_{j_0}$ and $F^{j_0}_{s_0}$ in \eqref{thm 4.1 eq3}. \end{proof}
\vskip 0.05in

We continue with the proof of Theorem \ref{relsol2}.

\begin{theorem}\label{relsol4}
 Let $G$ be a group hyperbolic relative to a family of exact, residually finite subgroups $\{ H_1, \ldots, H_n \}.$  Let $G \curvearrowright \mathcal N$ be an action on a tracial von Neumann algebra and denote by $\mathcal{M}=\mathcal N \rtimes G$ be the corresponding cross-product von Neumann algebra.  Let $p \in \M$ be a nonzero projection and let $\mathcal{A} \subseteq p\mathcal{M}p$ be a diffuse von Neumann subalgebra such that $\mathcal{A}' \cap p\mathcal{M}p$ has no amenable direct summand. Let $\mathcal{Q} = \mathcal{QN}_{p\mathcal{M}p}(\mathcal{A})^{''}.$ Then one can find $1 \leq i \leq n$, projections $r \in  \mathcal{A}, q \in \mathcal{A}' \cap p\mathcal{M}p$ with $ rq \neq 0$ and $u \in \mathscr{U}(\mathcal{M})$ such that $u(rq \overline{\mathcal{QN}}_{p\mathcal{M}p}^{1}(\mathcal{Q})r q)u^* \subseteq \mathcal N \rtimes H_i.$
If we assume in addition that $H_i$ are icc then in the previous containment we can replace $r q$ by its central support in $\overline{\mathcal{QN}}_{p\mathcal{M}p}^{1}(\mathcal{Q}).$
\end{theorem}

\begin{proof} From the previous theorem we have $\mathcal{A} \preceq \mathcal N \rtimes H_i.$ Hence there are nonzero projections $r \in \mathcal{A}, q \in \mathcal N\rtimes  H_i$, a partial isometry $v \in p\mathcal{M}q$ and a $*-$isomorphism onto its image $\Psi : r\mathcal{A}r \rightarrow \Psi (r\mathcal{A}r)=\mathcal{R} \subseteq q (\mathcal N \rtimes  H_i)q$ such that $\Psi (x)v =vx$ for all $x \in r\mathcal{A}r.$  This relation implies that $v^{*}v \in r\mathcal{A}r' \bigcap r\mathcal{M}r = (\mathcal{A}^{'}\bigcap p\mathcal{M}p)r$ and $vv^{*} \in \mathcal{R}^{'} \bigcap q\mathcal{M}q.$\\
The hypothesis and Theorem \ref{transint} imply that $\mathcal R \nprec_{ \mathcal N \rtimes  H_i} \Nn$.  Since $H_i < G$ is almost malnormal, we have  $\mathcal{R}^{'} \bigcap q\mathcal{M}q \subseteq q(\mathcal N \rtimes H_i)q$ and hence $vv^{*} \in q(\mathcal{N} \rtimes H_i)q.$ Thus $vr\mathcal{A}rv^{*} \subseteq \mathcal N \rtimes  H_i.$ 

Letting $u \in \M$ be a unitary with $uv^{*}v=v$ we get $ur\mathcal{A}rv^{*}vu^{*} \subseteq \mathcal N \rtimes H_i.$ As $v^{*}v \in (\mathcal{A}^{'}\bigcap p\mathcal{M}p)r$ there is a projection $q \in \mathcal{A}^{'} \bigcap p\mathcal{M}p$ with $v^{*}v = rq \neq 0.$ Thus $ur\mathcal{A}rq u^{*} \subseteq \mathcal N \rtimes H_i$ and since $H_i < G$ is almost malnormal we get $u\mathcal{QN}_{rq \mathcal{M}rq}(r\mathcal{A}rq)^{''}u^{*} \subseteq \mathcal N \rtimes H_i.$ Using the quasinormalizers formulae \cite{FGS10} we conclude  $urq \mathcal{Q}rqu^{*} \subseteq \mathcal N \rtimes  H_i .$ Repeating this argument we finally obtain $ur q\overline{\mathcal{QN}}_{p\mathcal{M}p}^{1}(\mathcal{Q})r qu^{*} \subseteq \mathcal N \rtimes  H_i.$ The last part follows as in the proof of \cite[Theorem 5.1]{IPP05}. We leave the details to the reader.
\end{proof}

\begin{theorem}\label{commut2}
 Let $G$ be a group hyperbolic relative to a family of exact, residually finite subgroups $\{ H_1, \dots, H_n \}.$  Let $G \curvearrowright \mathcal N$ be an action on a tracial von Neumann algebra and denote by $\mathcal{M}=\mathcal N \rtimes G$ be the corresponding cross-product von Neumann algebra.    Let $p \in \M$ be a nonzero projection and let $\mathcal{A} \subseteq p\mathcal{M}p$ be a von Neumann subalgebra such that $\mathcal A \nprec \mathcal N$.  Denote by $\mathcal{Q}=\mathcal{QN}_{p\mathcal{M}p}(\mathcal{A})^{''}.$ Then one can find orthogonal projections $p_0,  \dots, p_n \in \mathscr{Z}(\overline {\mathcal{QN}}_\M^{1}(\mathcal{Q}))$ with $p_0 + \dots + p_n = p$ that are maximal with the following properties: 
    \begin{enumerate}
        \item $(\mathcal{A}^{'} \bigcap p\mathcal{M}p)p_0$ is amenable.
        \item Any corner of $\overline {\mathcal{QN}}_\M^{1}(\mathcal{Q})p_i$ has a nontrivial subcorner which can be unitarily conjugated into $\mathcal N \rtimes H_i.$

    \end{enumerate}
     Moreover, we can pick $p_i \in \mathscr{Z}(\overline{\mathcal{QN}}_{p\M p}^{1}(\mathcal{Q}))$ for every $1\leq i \leq n.$
     \begin{enumerate}
     \setcounter{enumi}{2}
         \item If in addition $\mathcal N \rtimes  H_i$ are factors then one can find unitaries $u_1, \ldots, u_n \in \mathcal{M}$ such that $u_i \overline{\mathcal{QN}}_{\mathcal{M}}^{1}(\mathcal{Q})p_i u_i^{*} \subseteq \mathcal N \rtimes  H_i$ for every $1\leq i\leq n.$ In particular, when $n=1$, there is a unitary $u\in \mathcal M$  such that $u \overline{\mathcal{QN}}_{\mathcal{M}}^{1}(\mathcal{Q}) u^{*} \subseteq \mathcal N \rtimes  H_1$. 
\end{enumerate}
\end{theorem}

\begin{proof}
  Let $p_0 \in \mathscr{Z}(A^{'} \bigcap p\mathcal{M}p)$ be the maximal projection such that $(\mathcal{A}^{'} \bigcap p\mathcal{M}p)p_0$ is amenable. Let $q = p - p_0 \in \mathscr{Z}(\mathcal{A}^{'} \bigcap p\mathcal{M}p)$ and notice that $(\mathcal{A}^{'} \bigcap p\mathcal{M}p)q_0$ has no amenable direct summand.\\
  Let $\{p_i\}_i \subset \mathscr{Z}(\overline{\mathcal{QN}}_{\mathcal M}^{1}(\mathcal{Q}))$ with $p_i \leq q$, be a family of maximal projections satisfying condition $(2)$ in the conclusion. Next we argue that $p_i \perp p_j$ for every $i \neq j.$ Otherwise let $p_i p_j =: r \neq 0.$ Thus $\overline{\mathcal{QN}}_{\mathcal{M}}^{1}(\mathcal{Q})r \preceq ^{\rm s} \mathcal \Nn \rtimes H_i$ and $\overline{\mathcal{QN}}_{\mathcal{M}}^{1}(\mathcal{Q})r \preceq^{\rm s} \Nn \rtimes H_j.$ Fix $\varepsilon >0.$ By \cite[Lemma 2.5]{Vae10} there exist finite subsets $L_i , R_i , L_j , R_j \subset G$ such that 
  \begin{equation*}
    \begin{split}
    &\lVert \mathbb P_{L_i H_i R_i}(x) - x\rVert_2 \leq \varepsilon , \; \; \text{for all} \; x\in \mathscr{U}(\overline{\mathcal{QN}}_{\mathcal{M}}^{1}(\mathcal{Q})), \text{ and } \\
    & \lVert \mathbb P_{L_j H_j R_j}(x) - x\rVert_2 \leq \varepsilon , \; \; \text{for all} \; x\in \mathscr{U}(\overline{\mathcal{QN}}_{\mathcal{M}}^{1}(\mathcal{Q})).
    \end{split}
\end{equation*}
Using triangle inequality further implies that
\begin{equation}\label{eqt9_1}
    \lVert x - \mathbb P_{L_i H_i R_i \bigcap L_j H_j R_j}(x)\rVert_2 \leq 2\varepsilon , \; \; \text{for all} \; x\in \mathscr{U}(\overline{\mathcal{QN}}_{\mathcal{M}}^{1}(\mathcal{Q})).
\end{equation}
However, since $H_i, H_j$ are almost malnormal in $G$ we get that $L_i H_i R_i \bigcap L_j H_j R_j$ is a finite set. Thus Theorem \ref{corner} and \eqref{eqt9_1} imply that $\overline{\mathcal{QN}}_{\mathcal{M}}^{1}(\mathcal{Q})\prec \mathcal N $, contradicting the hypothesis that $\mathcal A \nprec \mathcal N$. 

Now let $s := p - p_0 - \sum_{i=1}^{n}p_i \in \mathscr{Z}(\mathcal{A}^{'} \bigcap p\mathcal{M}p)$ and assume that $s \neq 0.$ \\
Consider the commuting algebras inside the corner $\mathcal{A}s, (\mathcal{A}s)^{'} \bigcap s \mathcal{M} s \subset s \mathcal M s$ and note that  $(\mathcal{A}s)^{'} \bigcap s \mathcal{M} s$ has no direct direct summand.  Using Theorem \ref{relsol4} one can find an index $j$, projections $e\in \mathcal A ,f\in s(\mathcal A'\cap p\mathcal M p)s$ with $ef\neq 0$ and a unitary $u\in \mathscr U(\mathcal M)$ such that 
 $u ef \overline{\mathcal{QN}}_{s \mathcal{M} s}^{1}(s \mathcal{Q} s)ef u^* \subseteq \Nn \rtimes H_i$ and thus $u ef\overline{\mathcal{QN}}_{p \mathcal{M} p}^{1}(\mathcal{Q})ef u^* \subseteq \Nn \rtimes H_j.$ Letting $z = z(ef) \in \mathscr{Z}(\overline{\mathcal{QN}}_{p\mathcal{M}p}^{1}(\mathcal{Q}))s$, this further implies any corner of $\overline{\mathcal{QN}}_{p\mathcal{M}p}^{1}(\mathcal{Q})z$ has a nontrivial corner that can be unitarily  conjugated onto $\Nn \rtimes H_i.$ But, replacing $p_j$ by $p_j + z$ we get the same conclusion. This contradicts $p_j$ is the largest projection satisfying  $(2).$ Hence $s = 0$. Part $(3)$ follows from Theorem \ref{transint}. 
\end{proof}

\subsection{Examples}\label{Examples} Our previous relative solidity results apply to several remarkable classes of relative hyperbolic groups, which do not seem to be covered by the existing results in the literature. 
\vskip 0.06in

\noindent{1) \emph{ Hyperbolic-by-cyclic groups.}} 
Let $G$ be any torsion-free hyperbolic group (e.g.\ any free group $G=\mathbb F_n$ of finite rank $n$) and let $\alpha \in Aut(G)$ be an infinite order automorphism. This induces and action still denoted by $ \alpha:\mathbb Z=\langle t \rangle \rightarrow Aut(G)$ by letting $\alpha_{t^k} (g)=\alpha^k(g)$ for all $g\in G$, $k\in \mathbb Z$. The corresponding semidirect product is denoted by $G\rtimes_\alpha \langle t\rangle$ or $G \rtimes_\alpha \mathbb Z$ and is called a \emph{hyperbolic-by-cyclic group}. 

Notice that if in addition $G$ is ressidually finite then so is $G \rtimes_\alpha  \mathbb Z$.
\vskip 0.06in
An element $g\in G$ has a polynomial growth under $\alpha$ if there is a polynomial $p \in \mathbb Z [X]$ such that $\lVert \alpha^n(g) \rVert \leq p(n)$ for all $n\in \mathbb N$. Here $\lVert \cdot \rVert$ can be either $\lVert \cdot \rVert_{w}$ or $\lVert \cdot \rVert_T$. Specifically, $\lVert g \rVert_{w}$ is defined as the infimum of $d_w(1, hgh^{-1})$ over $h\in G$ where $d_w$ is a word metric for some chosen generating set, and $\lVert g \rVert_T$ is defined as the infimum of $d_T(v,gv)$ for $v$ ranging over the vertices of a given metric $G$-tree $T$.
Moreover, let $A\leqslant G$ be a subgroup whose conjugacy class is preserved by some positive power of $\alpha$, and let $s$ be the smallest positive integer satisfying $
\alpha^{s}(A) = h^{-1} A h$. Then the \emph{suspension of $A$ by $\alpha$ inside  $G\rtimes_{\alpha} \langle t \rangle$} is the canonical semidirect product subgroup $A \rtimes _{{\rm ad}(h) \circ \alpha^{s}}\langle t^{s}h^{-1} \rangle$. 
\vskip 0.05in 
In \cite[Theorems 1.1-1.2]{DK23}, it was proved that one can find a finite family of malnormal quasiconvex subgroups $\{A_1, \ldots, A_r\}$ of $G$ such that all elements of $A_i$ are polynomially growing under $\alpha$. Conversely, any polynomially growing element under $\alpha$ can be conjugated in one of the $A_i$. Finally, if we pick a maximal family $\{A_1,\ldots,A_r\}$ of maximal subgroups whose elements grow polynomially under $\alpha$  then  $G\rtimes_\alpha \mathbb Z$ is hyperbolic relative to $\{H_i\}$---the suspensions of $A_i$ under $\alpha$.  Combining these results with Theorem \ref{mrelsol1} yields the following. 

\begin{corollary} Under the previous notations assume in addition $G$ is residually finite (e.g.\ when $G$ is a free group). Then $\mathcal L(G\rtimes _\alpha \mathbb Z)$ is solid relative to $\{\mathcal L(H_i)\}_i$.
\end{corollary}

\noindent {2) \emph{(Free product)-by-cyclic groups}.} Let $G$ be a free product group of the form $G=G_1 \ast \dots \ast G_p \ast \mathbb F_k$ where $G_i$ is freely indecomposable and not infinite cyclic for each $ 1\leq i \leq p $, and $\mathbb F_k$ is a free group of rank $k$.

Given a subset $X\subset G$ we will denote by $[X]$ the collection of all conjugates of $X$ in $G$. 
The collection $\mathbb G=\{ [G_1], \dots, [G_p] \}$ is called a free factor system of the group $G$. We denote by $Aut(G,\mathbb G)$ to be the set of all $\Phi \in Aut(G)$ such that for every 
 $ 1\leq i\leq  p $, one can find an element $g_i \in G$ and $1 \leq j \leq p$ such that $\Phi (G_i) = g_i^{-1}G_j g_i$. Note that $Inn(G)\unlhd Aut(G,\mathbb G)$ and this one could  define $Out(G,\mathbb G) = Aut(G,\mathbb G) / Inn(G)$.
\vskip 0.06in

An automorphism $\Phi\in Aut(G,\mathbb G)$  is called \emph{fully irreducible} if the given free factor system $\mathbb G$ is the largest one fixed by $\Phi^n$ for every $n\in \mathbb N$. As $\Phi^n$ preserves $\mathbb G$, for every $1\leq i\leq  p $ one can find a smallest positive integer $n_i$ and $g_i \in G$ such that $\Phi^{n_i}(G_i)= g^{-1}_i G_i g_i$. Thus one can consider the so-called \emph{mapping torus of $G_i$}, i.e.\ the semidirect product $G_i \rtimes_{{{\rm ad}(g_i)\circ \Phi^{n_i}}_{|G_i}}\mathbb Z$ subgroup of $G$.

An automorphism $\Phi \in Aut(G, \mathbb G)$ is called \emph{atoroidal} if for any hyperbolic element $g\in G$ and any $n\in \mathbb N$  we have  $[\Phi^n(g)] \neq [g]$.

\vskip 0.06in

In \cite[Theorem 4.34, Theorem 6.10]{Li18} it was proved that for any finitely generated free product group $G = G_1 \ast \dots \ast G_p \ast \mathbb F_k$ with $\; k \geq 2 \; \text{or} \; p+k \geq 3$, and every fully irreducible, atoroidal automorphism $\Phi\in Aut(G,\mathbb G)$, the corresponding  semidirect product $G\rtimes_{\Phi} \mathbb Z$ is hyperbolic relative to the mapping-torus subgroups of $\mathbb G$, i.e. $\{ G_i \rtimes_{ad_{g_i}\circ\Phi^{n_i}|_{G_i}}\mathbb Z, 1\leq i\leq p \}$. Combining this result with Theorem \ref{mrelsol1} we obtain 

\begin{corollary} Under the previous notations assume in addition that $G_i$ is ressidually finite for every $i$. Then $\mathcal L(G\rtimes_\Phi \mathbb Z)$ is solid relative to $\{ \mathcal L(G_i \rtimes_{ad_{g_i}\circ\Phi^{n_i}|_{G_i}}\mathbb Z), 1\leq i\leq p \}$.    
\end{corollary}

\noindent {3) \emph{Free-by-free groups.}  Let $\mathbb F$ be a free group of finite rank $\geq 3$, and let $\Phi \in Out(\mathbb F)$. Define $\mathcal K =\{ K_1,  \dots, K_p \}$ as a subgroup system, and denote $\mathcal L_{\mathcal K}^{+}(\Phi)$ and $\mathcal L_{\mathcal K}^{-}(\Phi)$ the collection of attracting and repelling laminations of $\Phi$, respectively, whose generic leaves are not carried by $\mathcal K$. As defined in \cite{GG23}, $\mathcal K$ is an admissible subgroup system for any rotationless and exponentially growing $\Phi \in Out(\mathbb F)$ if it satisfies the following conditions:
\begin{enumerate}
    \item $\mathcal K$ is a malnormal subgroup system,\\
    \item $\mathcal L_{\mathcal K}^{+}(\Phi)$ and $ \mathcal L_{\mathcal K}^{-}(\Phi)$ are nonempty,\\
    \item $\Phi (K_s) = K_s$ for every $ 1\leq s \leq p$,\\
    \item Let $V^{+}$ denote the union of attracting neighborhoods of elements of $\mathcal L_{\mathcal K}^{+}(\Phi)$ defined by generic leaf segments of length $\geq 2C$, and $V^{-}$ is defined similarly for $\mathcal L_{\mathcal K}^{-}(\Phi)$. By increasing $C$ if necessary, we have $V^{+} \cap V^{-} = \varnothing $,\\
    \item Every conjugacy class which is not carried by $\mathcal K$ is weakly attracted to some element of $\mathcal L_{\mathcal K}^{+}(\Phi)$.
    
\end{enumerate}

Furthermore, we say that $\mathcal K$ is an admissible subgroup system for a finitely generated subgroup $ < \Phi_1,  \dots, \Phi_k >$ if $\mathcal K$ is an admissible subgroup system for each $\Phi_i \in Out(\mathbb F)$. We will use the term "admissible subgroup system" when the context is clear. \\
Recall that a finite collection $\mathcal K = \{ [K_1],  \dots, [K_p] \}$ of conjugacy classes of nontrivial finite rank subgroups $K_s < \mathbb F$ for $ 1\leq s \leq p $ is called a subgroup system.  \\
Let $\mathbb F$ be a free group of finite rank $\geq 3$, and $\Phi_1, \dots, \Phi_k \in Out(\mathbb F)$ be a collection of rotationless, exponentially growing automorphisms that are independent relative to admissible subgroup system $\mathcal K = \{ [K_1], \dots, [K_p] \}.$ 

Then, following \cite[Lemma 4.7]{GG23} there exists some $M \geq 1$ such that for all $m_i \geq M, Q = <\Phi_1^{m_1}, \dots, \Phi_k^{m_k}>$ is a free group and every element of the group $Q$ is exponentially growing, rotationless and admits an admissible subgroup system. Additionally, the same result implies that $\mathbb F \rtimes \widehat{Q}$ is hyperbolic relative to the collection of subgroups $\{ K_s \rtimes \widehat{Q}_s \}$ for $ 1\leq s \leq p $, where $\widehat{Q}_s$ is a lift of $Q_s$ that preserves $K_s$ and $\widehat{Q}$ is any lift of $Q$. Since these groups are also ressidually finite, then combining these with Theorem \ref{mrelsol1} we get the following:

\begin{corollary} Under the previous notations the algebra $\mathcal L(\mathbb F \rtimes \widehat{Q})$ is solid relative to $\{ \mathcal L(K_s \rtimes \widehat{Q}_s), 1\leq s\leq p \}$.    
\end{corollary}

\vskip 0.03in

\noindent {4) The relative solidity also holds for von Neumann algebras of all ressidually finite groups $G$ satisfying $C'(1/6)$-small cancellation condition over free product $F=A_1\ast \cdots \ast A_n$} as described towards the end of the Subsection \ref{relhypex}.

 \subsection{An application to 
 s-primeness of group von Neumann algebras} In this subsection we derive our main primeness results.

\noindent\emph{Proof of Theorem \ref{primness2}}.  Let $\M=\mathcal L(G)$. Also throughout this proof for finite index subgroups  $N_i \trianglelefteq H_i$ and a finite set $F \subset G$ we will denote by $N_F< G$ the infinite index subgroup generated by $N_F=\underset{1\leq i\leq n, h\in F}{\bigvee}N_i^h$ as in the hypothesis. 
 
    Assume by contradiction there exist a projection $0\neq p\in \M$ and diffuse commuting von Neumann subalgebras $\mathcal A, \mathcal B \subseteq p\mathcal M p$ such that  $[p\mathcal{M}p:\mathcal A \vee 
    \mathcal B]<\infty$. Next we claim that $\A$ and $ \B$ are non-amenable. Towards this, assume by contradiction that $\A$ is amenable. Since the finite conjugacy radical of $G$ is a finite group, after cutting by a projection in the center of $\M$, we can assume in addition that $\A\vee \B \subseteq p\M p$ is a finite index inclusion of II$_1$ factors. Then using the same argument form the proof of \cite[Corollary 7.2]{CdSS17} we get that $p \M p $ is a McDuff factor. Hence, letting $z$ be the central support of $p$, we further have $\M z$ is a McDuff factor. In particular, its central sequence algebra $(\M' \cap \M^\omega )z$ is diffuse. However, since $G$ admits an unbounded quasi-cocycle into its left regular representation, this contradicts \cite[Theorem 3.1]{CSU16}. This yields our claim.  
\vskip 0.05in
    
    By Theorem \ref{ballconv},  for every $ \varepsilon >0$ there exists $ k>0$ such that 
        $\lVert \mathbb P_{B_k} (x) - x  \rVert _2 \leq \varepsilon, \; \; \text{for all} \; x\in \mathscr{U}(\A) \cup \mathscr{U}(\B).$ Using the assumption, there are finite subsets $F, S\subset G$ so that $B_k \subseteq N_F S$. Altogether, these show for every $ \varepsilon >0$ there exist finite subsets $F, S\subset G$ satisfying \begin{equation}\label{conv3}\lVert \mathbb P_{N_F S}(x) - x\rVert_2 \leq \varepsilon, \; \; \text{for all} \; x\in \mathscr{U}(\A) \cup \mathscr{U}(\B).\end{equation} 

\noindent Fix $\sqrt{{\frac{\tau(p)}{2}
}}>\varepsilon >0$. By \eqref{conv3} there are finite sets $F_1, F_2, S_1,S_2 \subset G$ satisfying 

\begin{equation}\label{conv4}
\begin{split}
&\lVert \mathbb P_{N_{F_1} S_1}(x) - x\rVert_2 \leq \frac{\varepsilon}{4} , \; \; \text{for all} \; x\in \mathscr{U}(\A); \text{ and } \\
& \lVert \mathbb P_{N_{F_2} S_2}(y) - y\rVert_2 \leq \frac{\varepsilon}{4\lvert S \rvert} , \; \; \text{for all} \; y\in \mathscr{U}(\B).
\end{split}
\end{equation}

\noindent Using \eqref{conv4}, for every $x\in \mathscr U (\A)$ and $y\in \mathscr U(\B)$ we have that
\begin{equation}\label{joinconv1}
    \begin{split}
        &\lVert \mathbb P_{ N_{F_1}S_1}(x) \mathbb P_{N_{F_2}S_2}(y) - xy \rVert_2  = \lVert \mathbb P_{ N_{F_1}S_1}(x)(\mathbb P_{N_{F_2}S_2}(y)-y)+  \mathbb P_{ N_{F_1}S_1}(x) y - xy \rVert_2\\
        & \leq \lVert \mathbb P_{ N_{F_1}S_1}(x) \rVert_\infty . \lVert \mathbb P_{N_{F_2}S_2}(y)-y \rVert_2 + \lVert \mathbb P_{ N_{F_1}S_1}(x) y - xy \rVert_2 \leq  \lvert S \rvert \frac{\varepsilon}{4\lvert S \rvert} + \frac{\varepsilon}{4}=\frac{\varepsilon}{2}.
    \end{split}
\end{equation}

\noindent Pick $F_3, S_3 \subset G$ be finite subsets satisfying  $(N_{F_1}S_1)( N_{F_2}S_2)\subseteq N_{F_3}S_3$. This combined with \eqref{joinconv1} imply that $\lVert \mathbb P_{N_{F_3} S_3}(xy) - xy \rVert_2 \leq \varepsilon,$ for all $x \in \mathscr{U}(A), y\in \mathscr{U}(B).$ Hence we further have $\frac{\tau(p)}{2}<\tau (p)-\varepsilon^2\leq \|\mathbb P_{N_{F_3} S_3}(z)\|_2^2$, for all $z \in \mathscr{U}(A) \mathscr{U}(B)$. Since $\|\mathbb P_{N_{F_3} S_3}(z)\|_2^2= \sum_{s\in S_3} \|\mathbb E_{\mathcal L(N_{F_3})}(z u_{s^{-1}})\|^2_2$, Theorem \ref{corner} further entails that $\A \vee \B \preceq \mathcal L(N_{F_3})$. Since $[p \M p:\A\vee \B ]<\infty$ we therefore have $\M \preceq \mathcal L(N_{F_3})$. Finally, using \cite[Lemma 2.2]{CI17} this implies that $[G: N_{F_3}] < \infty$, which contradicts the assumption. 
\hfill $\square$

\vskip 0.06in
This result combined with Theorems \ref{rips1}, yields many new examples of s-prime von Neumann algebras. 

\begin{corollary}\label{sprimeexamples}  $\mathcal L(G)$ is s-prime for any $G$ in one of the following classes of relatively hyperbolic groups:
\begin{enumerate}
    \item all Rips construction groups constructed in Theorem \ref{rips1} for $Q$ infinite. 
    \item all hyperbolic-by-cyclic groups  
$H \rtimes_\phi \mathbb Z$, where $H$ is a non-elementary, residually finite, torsion free, hyperbolic group (e.g.\ any nonabelian free group of finite rank) and $\phi $ any automorphism of exponential growth, \cite{Li18}.
\item all free-by-cyclic groups $(G_1\ast \cdots \ast G_p \ast \mathbb F_k)\rtimes_\Phi \mathbb Z$ where $G_i$ are ressidually finite and $\Phi$ is a fully irreducible atoroidal automorphism.  

\item all  $C'(1/6)$-small cancellation groups over free products satisfying the hypothesis of Proposition \ref{smallcancoverf}.

\end{enumerate}
    
\end{corollary}

\subsection {Peripheral absorption of abelian algebras through their normalizers}

\begin{theorem}\label{peripheral absorption} Let $G$ be a group together with a hyperbolically embedded subgroup $H\leqslant G$. Assume $G \curvearrowright \mathcal N$ is trace preserving action on tracial von Neumann algebra $\mathcal N$ and let  $\mathcal M=\mathcal N\rtimes G$. Let $p\in \mathcal N \rtimes H$ be a nonzero projection. Let $\mathcal A \subseteq p \mathcal M p$ be an abelian von Neumann algebra such that $\A\nprec_\M \Nn$ and let $\mathcal G \leqslant \mathcal N_{p\mathcal M p}(\mathcal A)$ be a subgroup whose von Neumann algebra satisfies $\mathcal G''\nprec_{\mathcal M} \mathcal N $.

\noindent If $\mathcal G \subset p (\mathcal N \rtimes H) p$ then $\overline{\mathcal {QN}}_{p\mathcal M p}^{1}(\mathcal A)\subseteq p(\mathcal N \rtimes H)p$. In particular, $\mathcal A \subset \mathcal N_{p\mathcal M p} (\mathcal A)''\subset p(\mathcal N \rtimes H) p$.
    
\end{theorem}

\begin{proof}
First we will show $\A \prec^s \Nn \rtimes H$.  Fix $r \in \Nn_{pMp}(A)'\cap pMp$ a non-zero projection and assume for a contradiction that $Ar \nprec \Nn \rtimes H$.  By Theorem (\ref{corner}), for any $\varepsilon>0$ there is $a \in \sU(A)$ such that $\lVert E_{\Nn \rtimes H}(ar) \rVert_2 < \varepsilon$.  By Kaplansky's density theorem there is $a_\varepsilon \in \mathrm{span} \{ \Nn u_g : g \in G \setminus H\}$ such that \begin{equation}\label{approx1}\lVert ar - a_\varepsilon \rVert_2 \leq \varepsilon\text{ and }\lVert a_\varepsilon \rVert_\infty \leq 2.\end{equation}  Let $K_\varepsilon = \{ g \in G \setminus H : E_\Nn( a_\varepsilon u_g^*) \neq 0 \}$ and note it is a finite set.  Since $r \in \mathcal G' \cap \A'$, for every  $u \in \mathcal G$ we have $uaru^*ar = (uau^*)ar = aruaru^*$ and using inequalities \eqref{approx1} we get
\begin{equation*}
\begin{split}
    \lVert r \rVert^2_2 &= \langle uaru^*ar , aruaru^* \rangle \leq |\langle u (ar - a_\varepsilon) u^*ar , aruaru^* \rangle| + |\langle ua_\varepsilon u^*ar, aruaru^* \rangle| \\
    &\leq \varepsilon + |\langle ua_\varepsilon u^*(ar - a_\varepsilon), aruaru^* \rangle| + |\langle ua_\varepsilon u^*a_\varepsilon, aruaru^* \rangle|\\
    & \leq \varepsilon + 2\varepsilon + |\langle ua_\varepsilon u^*a_\varepsilon, (ar - a_\varepsilon)uaru^* \rangle| + |\langle ua_\varepsilon u^*a_\varepsilon, a_\varepsilon uaru^* \rangle| \\
    &\leq 3\varepsilon + 4\varepsilon + |\langle ua_\varepsilon u^*a_\varepsilon, a_\varepsilon u(ar - a_\varepsilon)u^* \rangle| + |\langle ua_\varepsilon u^*a_\varepsilon, a_\varepsilon ua_\varepsilon u^* \rangle| \\
    & \leq 7\varepsilon + 8\varepsilon + |\langle ua_\varepsilon u^*a_\varepsilon, a_\varepsilon ua_\varepsilon u^* \rangle|\ .
\end{split}
\end{equation*}
Since $K_\varepsilon \subset G\setminus H$ is finite and $H$ is hyperbolically embedded in $G$, then by \cite[Theorem 2.9]{CDS23} there is a finite subset $L_\varepsilon \subset H$ such that 
\begin{equation}\label{free_independence}
K_\varepsilon (H\setminus L_\varepsilon) K_\varepsilon (H\setminus L_\varepsilon) \cap (H\setminus L_\varepsilon) K_\varepsilon (H\setminus L_\varepsilon) K_\varepsilon  = \emptyset.
\end{equation}
Thus the last inner product in the above equation satisfies
\begin{equation*}
\begin{split}
    |\langle ua_\varepsilon u^*a_\varepsilon, a_\varepsilon ua_\varepsilon u^* \rangle| \leq  | &\langle \mathbb P_{L_\varepsilon}(u) a_\varepsilon u^*a_\varepsilon, a_\varepsilon ua_\varepsilon u^* \rangle | \\
    &+ | \langle \mathbb P_{H\setminus L_\varepsilon}(u) a_\varepsilon \mathbb P_{L_\varepsilon}(u^*) a_\varepsilon, a_\varepsilon ua_\varepsilon u^* \rangle | \\
    &+ | \langle \mathbb P_{H\setminus L_\varepsilon}(u) a_\varepsilon \mathbb P_{H\setminus L_\varepsilon}(u^*) a_\varepsilon, a_\varepsilon \mathbb P_{L_\varepsilon}(u) a_\varepsilon u^* \rangle | \\
    &+ | \langle \mathbb P_{H\setminus L_\varepsilon}(u) a_\varepsilon \mathbb P_{H\setminus L_\varepsilon}(u^*) a_\varepsilon, a_\varepsilon \mathbb P_{H \setminus L_\varepsilon}(u) a_\varepsilon \mathbb P_{L_\varepsilon}(u^*) \rangle |\\
    &+ | \langle \mathbb P_{H\setminus L_\varepsilon}(u) a_\varepsilon \mathbb P_{H\setminus L_\varepsilon}(u^*) a_\varepsilon, a_\varepsilon \mathbb  P_{H \setminus L_\varepsilon}(u) a_\varepsilon \mathbb P_{H \setminus L_\varepsilon}(u^*) \rangle |\ .
\end{split}
\end{equation*}

Notice the last term above is zero by (\ref{free_independence}), $\lVert \mathbb P_{H\setminus L_\varepsilon}(u) \rVert_\infty = \lVert E_{\mathcal N \rtimes H}(u) - \mathbb P_{L_\varepsilon}(u) \rVert_\infty \leq 1 + |L_\varepsilon|$, $\lVert \mathbb P_{H\setminus L_\varepsilon}(u^*) \rVert_\infty \leq 1 + |L_\varepsilon|$ and since $\mathcal G'' \nprec \Nn$, then $u\in \mathcal{G}$ can be chosen so that $\lVert \mathbb P_{L_\varepsilon}(u) \rVert_2$ and $\lVert \mathbb P_{L_\varepsilon}(u^*) \rVert_2$ are arbitrarily small.  Hence $u \in \mathcal G$ can be chosen so that the previous inequalities imply $\lVert r \rVert^2_2 \leq 15\varepsilon + |\langle ua_\varepsilon u^*a_\varepsilon, a_\varepsilon ua_\varepsilon u^* \rangle| \leq 16\varepsilon $. Since $\varepsilon>0$ was arbitrary, we get $r=0$, a contradiction.
\vskip 0.05in

In particular, using \cite[Lemma 2.4(2)]{DHI19},  the above argument implies that $A\prec^s \Nn \rtimes H$.  We will now show for any projection $0\neq e \in \A' \cap p\M p$ there exists a projection $0\neq r \in e(\A' \cap p\M p)e$ such that $\A r \subset \Nn \rtimes H$.  

Towards this, fix a projection $0\neq e \in \A' \cap p\M p$. Thus $\A e \prec \Nn \rtimes H$ and there are non-zero projections $a \in\A e$, $q \in \Nn \rtimes H$, a partial isometry $v \in q \M a \setminus \{0\}$ and a unital normal $*$-embedding $\Theta : \A a \rightarrow q (\Nn \rtimes H) q$ such that $\Theta(x)v =vx$ for all $x \in \A a$.  Let $\Q := \Theta(\A a)$, and since $\A \nprec \Nn$ then we have that $\Q \nprec_{\Nn\rtimes H} \Nn$.   Since $H$ is almost malnormal in $G$, by Theorem \ref{quasinormalizercontrol} this further implies  that $vv^* \in \Q' \cap q \M q \subset \Nn \rtimes H$ and $v\A v^* = \Q vv^* \subset \Nn \rtimes H$.  Let $r = v^*v$ and $w \in \sU(\M)$ be such that $v = wr$.  We have that $w\A r w^* = \Q vv^* \subset \Nn \rtimes H$ and by Thorem \ref{quasinormalizercontrol} again we have $w \Q\Nn_{r\mathcal M r}(\A r)'' w^* \subset \Nn \rtimes H$. Using the compression formula from Proposition \ref{quasinormcompress} we get \begin{equation}\label{normalizercontainment}wr \Nn_{p\M p} (\A)'' rw^*  \subset wr\Q\Nn_{p\M p} (\A )'' rw^*= w\Q\Nn_{r\M r} (\A r)'' rw^*\subset \Nn \rtimes H .\end{equation}

Let $z$ be the central support of $r$ in $\Nn_{p\M p}(\A)''$.  Let $\varepsilon>0$, then there is a central projection $z' = \sum_{i=1}^m x_i^* x_i \in \sZ(\Nn_{p\M p}(\A)'')$ such that $x_i \in r \Nn_{p\M p}(\A)''$ is a partial isometry and $\tau(z -z') < \varepsilon$.  Therefore, using \eqref{quasinormalizercontrol}, we obtain $$ wz'r \Nn_{p\M p}(A)'' = wr \Nn_{p\M p}(\A)'' z' = \sum_{i=1}^m (wr \Nn_{p\M p}(\A)'' x_i^* w^*) wx_i \subset \sum_{i=1}^m (\Nn \rtimes H) w x_i.$$
 Therefore $wz'r \A \subset \sum_{i=1}^m (\Nn \rtimes H) wx_i $, which by Theorem \ref{quasinormalizercontrol} implies $wz'r \in \Nn \rtimes H$.  Since $\lVert wr - wz'r \rVert_2^2 \leq \lVert z - z' \rVert_2^2 <\varepsilon $, we get that  $wr \in \Nn \rtimes H$ and hence $r \Nn_{p\M p}(\A)'' r \subset \Nn \rtimes H$.   Thus, we have shown for any projection $0\neq e \in \A' \cap p\M p$, there is projection $0\neq r \in e(\A' \cap p\M p)e$ such that $\A r \subset \Nn \rtimes H$. A basic maximality argument implies that  $\A \subset \Nn \rtimes H$, which by Theorem \ref{quasinormalizercontrol} yields the desired conclusion.
\end{proof}




\section{ \texorpdfstring{W$^*$}--rigidity and one-sided fundamental semigroup for \texorpdfstring{II$_1$}- factors associated with a class of relative hyperbolic groups}\label{onesidedpropt}



\subsection{Factors arising from property (T) relative hyperbolic groups} This subsection is mainly devoted to the computation of the one-sided fundamental semigroup of II$_1$ factors arising from certain relative hyperbolic groups, many with property (T). To derive these results, we first need a few W$^*$-rigidity results for $\ast$-embeddings between factors of wreath product groups that will be used in the ``peripheral analysis'' of our relative hyperbolic groups.

\vskip 0.07in

The first result is a generalization of a fundamental theorem by Popa on the W$^*$-rigidity of wreath-product von Neumann algebras \cite[Theorem 0.1]{Pop03}. While the core of our proof is based on the work of Popa \cite{Pop03} and Ioana-Popa-Vaes \cite[Theorem 6.1]{IPV10}, we present an alternative proof using the precise form of \cite[Theorem 4.1]{CIOS22a}, which in turn heavily relies on the two aforementioned results. Our approach follows an argument very similar to that of \cite[Theorem 4.1]{CIOS24}. We adopt the notation from this result and encourage the reader to consult it in advance, as our proof focuses on different aspects of the technique.

To introduce our statement we need a definition. Two II$_1$ factors $\mathcal M$ and $\mathcal N$ are called \emph{virtually isomorphic} if there is $t>0$ and a $\ast$-embedding  $\Theta:\mathcal M\rightarrow \mathcal N^t$ whose image $\Theta(\mathcal M )\subseteq \mathcal N^t$ has finite index; such  $\Theta$ is called a \emph{virtual $\ast$-isomorphism}.





\begin{theorem}\label{embedding2} 

\noindent For $i=1,2$, let $G_i=A_i\wr B_i$ with $A_i$ any infinite abelian group, $B_1$ admitting a subgroup $T = T_1\times T_2 \leqslant B_1$ for some infinite icc property (T) groups $T_1,T_2$, and $B_2 = B_2^1\times B_2^2$ for some icc non-amenable, biexact, weakly amenable groups $B_2^1, B_2^2$.  Let $t>0$ and assume $\Theta :\mathcal L(G_1) \rightarrow \mathcal L(G_2)^t$ is any $\ast$-embedding.
\vskip 0.03in
\noindent  Then $t\in\mathbb N$ and there are $t_1,\ldots, t_q\in\mathbb N$ with $t_1+\cdots+t_q=t$, for some $q\in\mathbb N$, a finite index subgroup $K<B_1$, an injective homomorphism $\delta_i:K\rightarrow B_2$, and a unitary representation $\rho_i:K\rightarrow\sU_{t_i}(\mathbb C)$, for every $1\leq i\leq q$, and a unitary $w\in \mathcal L(G_2)^t=\mathcal L(G_2)\overline{\otimes}\mathbb M_t(\mathbb C)$ such that 
$$\text{$w\Theta(u_g)w^*=\emph{diag}(v_{\delta_1(g)}\otimes\rho_1(g),
\ldots, v_{\delta_q(g)}\otimes\rho_q(g))
$, for every $g\in K$.}$$
where $(u_g)_{g\in G_1}$ and $(u_h)_{h\in G_2}$ are the canonical generating unitaries of $\Ll (G_1)$ and $\Ll (G_2)$ respectively.

\noindent In addition, we have the following (finite index) inclusion $$w \Theta (\mathcal L(A_1^{(B_1)}))w^*\subseteq \mathcal L(A_2^{(B_2)})\otimes \mathbb D_t(\mathbb C).$$

\noindent Moreover, if $\Theta$ is a virtual $\ast$-isomorphism then the inclusion $\delta_i(K)<B_2$ has finite index for all $i$. In particular, $B_1$ is commensurable with $B_2$.    

\end{theorem}

\begin{proof}

First, we fix some notation.  Let $\M=\Ll(G_2)$ and $\Q = \Ll(A_2^{(B_2)})$ be a Cartan subalgebra of $\M$. Let $n$ be the smallest integer such that $t\leq n$, $\tM = \M \otimes \mathbb{M}_n(\mathbb{C})$ and $\tQ = \Q \otimes \mathbb{D}_n(\mathbb{C})$, where $\mathbb{D}_n(\mathbb{C}) \subseteq \mathbb{M}_n(\mathbb{C})$ is the algebra of diagonal matrices. Notice that $\tQ \subseteq \tM$ is a Cartan inclusion.  Let $p\in \tQ$ be a projection such that $(\tau\otimes \mathrm{Tr})(p) = t$, and identify $p\tM p = \Ll(G_2)^t$.  Taking the Pontryagin dual of the abelian group $A_2$ we obtain a probability space $(Y,\nu)$  and a Bernoulli action $B_2 \ca^\alpha (X,\mu)=(Y^{\otimes B_2}, \nu^{\otimes B_2})$ inducing  the identifcation $\Q = L^\infty(X,\mu)$, $\M = L^\infty(X,\mu) \rtimes_\alpha B_2$, $\tQ = L^\infty (\widetilde{X},\widetilde{\mu})$ and $\tM = L^\infty(\widetilde{X},\widetilde{\mu})\rtimes_{\widetilde{\alpha}} \widetilde{B_2}$, where $\widetilde{B_2} = B_2 \times \mathbb{Z}/n\mathbb{Z}$, $(\widetilde{X},\widetilde{\mu})=(X \times \mathbb{Z}/n\mathbb{Z}, \mu \times c)$, where $c$ the counting measure, $\mathbb{Z}/n\mathbb{Z} \ca \mathbb{Z}/n\mathbb{Z}$ by left multiplication and $\widetilde{\alpha}$ is the product action.  Fix a Borel subspace $X_0 \subset \widetilde{X}$ such that $1_{X_0} = p \in L^\infty (\widetilde{X})$.

We now proceed to show that the embedding $\Theta$ produces a p.m.p. action of $B_1$ on $X_0$.  We will use a few more definitions for notational simplicity.  Let $\A = \Theta(\Ll(A_1^{(B_1)}))$, $\B = \Theta(\Ll(B_1))$ and $\mathcal{D}_j = \Theta(\Ll(T_j)) $ for $j=1,2$, so $\A,\B,\mathcal{D}_1,\mathcal{D}_2 \subseteq p\tM p$.  Notice that $\tM \cong (\Q^n \rtimes B_2^1) \rtimes B_2^2 $ (where $\Q^n = \Q \otimes \mathbb M_n(\mathbb C)$ and $B_2$ acts trivially on $\mathbb M_n(\mathbb C)$).  Thus, by \cite[Theorem 1.4]{PV12} and the presence of the two commuting icc property (T) groups $T_1^1, T_1^2$ inside the normalizer subgroup $\sN_{G_1}(A_1)$, we have $\A \prec^s \Q^n \rtimes B_2^1$.

Indeed, since $\A$ is amenable  and $B_2^2$ is weakly amenable and bi-exact, \cite[Theorem 1.4]{PV12} implies that for any non-zero projection $z \in \sZ(\sN_{p\Tilde{\M}p}(\A)'')$ we have either :
\begin{enumerate}
    \item $\A z \prec_{\tM} \Q^n \rtimes B_2^1$, or
    \item the normalizer $\sN_{z\tM z}(\A z)''$ remains amenable relative to $\Q^n \rtimes B_2^1$.
\end{enumerate}
Assume, for a contradiction, that case (2) holds.  Notice that $T
\leqslant B_1 $ is an icc property (T) subgroup of $\sN_{G_1}(A_1^{(B_1)})$.  Then $\Theta(\Ll (T)) \subset \sN_{z\tM z}(\A z)''$ and $\Theta(\Ll (T))$ is amenable relative to $\Q^n \rtimes B_2^1$.  But, together with the property (T) of $\Theta(\Ll (T))$, this would imply $\Theta(\Ll (T)) \prec\Q^n \rtimes B_2^1$ and so $\Theta(\Ll(T)) \prec (\Q\otimes \mathbb{D}_n(\mathbb C)) \rtimes B_2^1$.  The last intertwining, however, contradicts the fact that $(\Q\otimes \mathbb{D}_n(\mathbb C)) \rtimes B_2^1$ is semisolid, a property implied by \cite[Theorem 4.6]{Oza05}.

Since $z \in \sP(\sZ(\sN_{p\tM p}(\A)''))$ was arbitrary, we have obtained $\A \prec^{\rm s} \Q^n \rtimes B_2^1$, and the same argument for the action of $B_2^1$ implies $\A \prec^{\rm s} \Q^n \rtimes B_2^2$.  Since the subalgebras $\Q^n \rtimes B_2^1, \Q^n \rtimes B_2^2 \subset \tM$ form a commuting square, we obtain $\A \prec^{\rm s} \Q^n$ by \cite[Lemma 2.8(2)]{DHI19}.  Let $\C := \A' \cap p\tM p$, then $\C$ is amenable by \cite[Corollary 4.7]{CIOS22a}.  Notice that $\theta(\Ll (T)) \subset \sN_{p\tM p}(\A)'' \subset \sN_{p\tM p}(\C)''$.  The same argument presented before to prove $\A \prec^{\rm s} \Q^n$ can be repeated for $\C$ , implying that $\C \prec^{\rm s} \Q^n$.

Hence, we have obtained $\C\prec^{\rm s} \tQ$, so \cite[Lemma 3.7]{CIOS22a} provides $u\in \sU(p\tM p)$ such that $\A \subset u(\tQ p)u^* \subset \C$.  Therefore, after replacing $\Theta$ by $\mathrm{ad}(u) \circ \Theta$ we may assume $\A \subset \tQ p \subset \C$.  Notice that for any $g\in B_1$ we have $\Theta(u_g) \in \sN_{p\tM p}(\A) \subset \sN_{p\tM p}(\C)$, so $\sigma_g = \mathrm{ad}({\Theta(u_g)})$ gives a trace preserving action $B_1 \curvearrowright \C$ that leaves $\A$ invariant and is free when restricted to $\A$, as it is conjugate to the Bernoulli action coming from the wreath product $A_1\wr B_1$.  Hence, from \cite[Lemma 3.8]{CIOS22a} we obtain an action $B_1\curvearrowright^{\beta}\C$ such that for all $g\in B_1$ we have $\beta_g = \sigma_g \circ \mathrm{ad}(w_g)$ for some $w_g \in \sU(\C)$, $\tQ p$ is $\beta$ invariant and the restriction of $\beta$ to $\tQ p$ is free.  Notice that the restriction of $\beta$ to $\tQ p$ induces a free p.m.p. action $B_1\curvearrowright X_0$.  Moreover, each automorphism $\beta_g$ is given by conjugation by unitaries in $\sN_{p\tM p} (\tQ p)$ so by \cite[Proposition 2.9(3)]{FM77} we have $$ \beta(B_1) \cdot x \subset \Tilde{\alpha}(B_2 \times \mathbb Z / n\mathbb Z) \cdot x \text{\ \ for a.e.\ }x \in X_0.$$

Since $B_1$ contains a weakly normal property (T) subgroup and $\alpha$ is Bernoulli, applying \cite[Lemma 3.10]{CIOS24}, we find a partition $X_0=\sqcup_{i=1}^lX_i$ into non-null measurable sets, for some  $l\in\mathbb N\cup\{\infty\}$,  and a finite index subgroup $ S_i< B_1$ such that $X_i$ is $\beta( S_i)$-invariant and
 the restriction of $\beta_{| S_i}$ to $X_i$ is weakly mixing, for all $i$.

The goal is to apply \cite[Theorem 4.1]{CIOS22a}. Notice that all conditions in this theorem are satisfied as $T<B_1$ has property (T) and $B_2 \ca X$ is Bernoulli action. Moreover since  $B_2 \ca X$ is Bernoulli, the stabilizers $Stab_{B_2}(b)=1$ for all $b\in B_2$. Since $S_i\leqslant B_1$ has finite index  for every $b \in B_2$, every infinite sequence $(g_m)_m\subset S_i$ satisfies that for all $s,t \in B_2 \times \mathbb Z/n\mathbb Z$ we have $\lim_{m\rightarrow\infty}\tilde \mu(\{ x\in X_0 \,:\, g_m\cdot x\in s( Stab_{B_2}(b)\times \mathbb Z/n\mathbb Z) t\cdot x\})=0$.
 Also, since $B_2$ is a product of non-amenable biexact groups, for every nontrivial $b\in B_2$ the centralizer $C_{B_2}(b)$ is either amenable or of the form $A\times B^2_2$ or $B^1_2\times A$ where $A$ is an amenable group. However in all these cases, since $T=T_1\times T_2<B_1$ is an icc property (T) group we have that $\mathcal L(T)\nprec_{\tilde M} \mathcal L(C_{B_2}(b))\otimes \mathbb M_n(\mathbb C)$. As $S_i\leqslant B_1$ has finite index we further have that $\mathcal L(S_i)\nprec_{\tilde M} \mathcal L(C_{B_2}(b))\otimes \mathbb M_n(\mathbb C)$. Altogether, these imply that the conclusion of \cite[Theorem 4.1]{CIOS22a} holds true.

Therefore one can find  a group monomorphism $\delta_i:S_i\rightarrow B_2$ and $\theta_i\in [\sR((B_2\times \mathbb Z/n\mathbb Z)\curvearrowright\widetilde X)]$ such that
 $\theta_i(X_i)= X\times\{i\}\equiv X$ and
$\theta_i\circ\beta(g)_{X_i}=\alpha(\delta_i(g))_{X_i}$ for all $g\in S_i$.

In particular,  $  \widetilde\mu(X_i)=1$. Thus,  $t=\widetilde\mu(X_0)=\sum_{i=1}^l\widetilde\mu(X_i)=l\in\mathbb N$. Since $n$ is the smallest integer with $n\geq t$, we get that
  $n=t=l$ and $p=1_{\widetilde{\mathcal M}}$.

For $1\leq i\leq n$, let $p_i=\textbf {1}_{X_i}\in\widetilde\Q p$ and $u_i\in\sN_{\widetilde{\mathcal M}}(\widetilde{\mathcal Q})$ such that $u_iau_i^*=a\circ\varphi_i^{-1}$, for every $a\in\widetilde{\mathcal Q}$. Then $u_ip_iu_i^*=1\otimes e_i$ and since $\beta_g={\rm ad}(\Theta(u_{g})w_g)$ we find $(\zeta_{i,g})_{g\in S_i}\subset\sU({\Q})$ with
\begin{equation}\label{th}
u_i\Theta(u_{g})w_g p_iu_i^*=\zeta_{i,g}v_{\delta_i(g)}\otimes e_i, \text{ for every }g\in S_i.
\end{equation}
 
Replacing $ S_i$ by $ S=\cap_{i=1}^n S_i$ we may assume that $ S_i= S$, for all $1\leq i\leq n$. Let $\M_0=\mathcal L(A^{(B_1)}_1\rtimes S)$. To prove the conclusion, it suffices to find a projection $p_0\in\theta(\M_0)'\cap\widetilde{\mathcal M} $ with $(\tau\otimes\text{Tr})(p_0)=k\in\{1,\ldots,n\}$, homomorphisms $\delta:S\rightarrow B_2$, $\rho:S\rightarrow\sU_k(\mathbb C)$, $w\in\sU(\widetilde {\mathcal M})$ such that $\delta$ is injective, $wp_0w^*=1\otimes (\sum_{i=1}^ke_i)$ and $w\theta(u_g)p_0w^*=v_{\delta(g)}\otimes\rho(g)$, for all $g\in S$.  Indeed, once we have this assertion, the conclusion will follow by a maximality argument.

Next, using the same arguments as in the proof of \cite[Claim 4.2]{CIOS24}  one can find  $1\leq k\leq n$ and a homomorphism $\delta: S\rightarrow B_2$ such that, after renumbering, we have $p_0=\sum_{i=1}^kp_i\in\theta(\M_0)'\cap \widetilde{\mathcal M}$ and can take $\delta_i=\delta$, for every $1\leq i\leq k$.

Let $u=\sum_{i=1}^ku_ip_i$ and $e=\sum_{i=1}^ke_i$. Then $u$ is a partial isometry, $uu^*=1\otimes e, u^*u=p_0$ and $u\widetilde{\mathcal Q} p_0u^*=\widetilde{\mathcal Q}(1\otimes e)$. If we let $\zeta_h=\sum_{i=1}^k\zeta_{i,h}\otimes e_i\in\sU(\widetilde{\mathcal Q}(1\otimes e))$, then \eqref{th} gives that
\begin{equation}\label{thetap_0}\text{$u\Theta(u_{g})w_gp_0u^*=\zeta_g(v_{\delta(g)}\otimes e)$, for every $g\in S$.}\end{equation}

Identify $\mathcal M_1:=(1\otimes e)\widetilde{\mathcal M}(1\otimes e)$  and $\Q_1:=\widetilde{\mathcal Q}(1\otimes e)$ with $\mathcal M\overline{\otimes}\mathbb M_k(\mathbb C)$ and $\mathcal Q\overline{\otimes}\mathbb D_k(\mathbb C)$, respectively.
Consider the unital $*$-homomorphism $\theta_1:\M_0\rightarrow \mathcal M_1$ given by $\theta_1(x)=u\Theta(x)p_0u^*$, for every $x\in \M_0$, and let $\sR_1=\theta_1(\P)'\cap \mathcal M_1$, where $\mathcal P:=\mathcal L(A_1^{(B_1)})$.  Letting $\omega_g=uw_gp_0u^*\in\sU(\sR_1)$, then \eqref{thetap_0} rewrites as
\begin{equation}\label{thetap}\text{$\theta_1(u_{g})\omega_g=\zeta_g(v_{\delta(g)}\otimes e)$, for every $g\in S$. }\end{equation}

We have $\Q_1\subset\sR_1$.
Since $(\theta_1(u_{g})\omega_g)_{g\in B_1}$ normalizes $\sR_1$ and $(\zeta_g)_{g\in S}\subset\sU(\Q_1)$, \eqref{thetap} implies that $(v_{\delta(g)}\otimes 1)_{g\in S}$ normalizes $\sR_1$.  Thus,  $\eta_g=\zeta_g\text{ad}(v_{\delta(g)}\otimes 1)(\omega _g^*)\in\sU(\sR_1)$ and
\begin{equation}\label{etah}
\theta_1(u_{g})=\eta_g(v_{\delta(g)}\otimes 1),\text{ for every }g\in S.
\end{equation}

To this end, we notice that since $S\leqslant B_1$ is finite index and $\delta$ is injective then $\delta(S)$ contains a product of icc property (T) groups. Also, as explained before, $C_{B_2}(g)$ is either amenable, or a product of the form $A\times B^2_2$ or $B_2^1\times A$ for an amenable group $A$. Thus in either case, every $g\in B_2\setminus\{1\}$, there is a sequence $(s_m)_m\subset S$ such that for every $g_1,g_2\in B_2$ and $k\in B_2\setminus\{1\}$ we have $g_1\text{ad}(\delta(s_m))(k)g_2\not=1$, for every $m$ large enough. Thus continuing, verbatim, as in the proof of \cite[Claim 4.3]{CIOS24} we further obtain that $\mathscr R_1 = \mathcal Q \otimes \mathcal T$ where $\mathcal T\subseteq \mathbb M_k(\mathbb C)$.

Thus relation \eqref{etah} implies  $\eta_g$ is a $1$-cocycle for the action
$\gamma_g ={\rm ad}(v_{\delta(g)}\otimes 1)$ on $\mathcal Q \otimes \mathcal T$. Since $B_1$ has a weakly normal property (T) group and the action ${\rm ad}(v_\delta(g))$ is a (direct sum of) Bernoulli action then using Popa \'s Cocycle Superigidity Theorem there exists a unitary $u\in \mathcal Q \otimes \mathcal T$  and group homomorphism $\rho: S\rightarrow \mathscr U(\mathcal T)$ such that $\eta_g = u (1\otimes \rho_g) \gamma_g(u^*)$ for all $g\in S$. This combined with equation \eqref{etah} implies that 

\begin{equation}
\label{etah1}
\theta_1(u_{g})= u (v_{\delta(g)}\otimes \rho_g) u^*,\text{ for every }g\in S,
\end{equation}
which concludes the first part of our proof.

For the moreover part one notice that since $S<B_1$ has finite index then so is $\Theta(\mathcal L(A_1^{(B_1)}\rtimes S))\subseteq \Theta (\mathcal L(G_1))$. As $\Theta$ is a virtual $\ast$-isomorphism, we conclude that $\Theta(\mathcal L(A^{(B_1)}_1\rtimes S))\subseteq \mathcal L(G_2)^t$ has finite index. Thus for every $p_i \in \Theta(\mathcal L(A^{(B_1)}_1\rtimes S))'\cap \widetilde {\mathcal M}$ the inclusion 
$w\Theta(\mathcal L(A^{(B_1)}_1\rtimes S))p_iw^*\subseteq wp_i\mathcal L(G_2)^tp_i w^*$ has finite index. Since the finite index property is hereditary, using the formulae of $\Theta$ from the first part of the conclusion and the fact that $\Theta (\mathcal L(A_1^{(B_1)}))\subseteq \mathcal A_2^{(B_2)}\otimes \mathbb D_n (\mathcal C)$ we further get that, letting  $e_i =wp_iw^*$, the inclusion  $e_i( \mathcal L(A_2^{(B_2)} \rtimes \delta_i(S)) \otimes \mathbb M_n(\mathbb C) )e_i \subseteq wp_i(\mathcal L(G_2) \otimes \mathbb M_n(\mathbb C)) p_i w^*$ has finite index. In particular we have that $\mathcal L(G_2)\prec_{\widetilde{\mathcal M}} \mathcal L(A_2^{(B_2)} \rtimes \delta_i(S)) $. Using  Lemma \ref{intsubgroups} this further entails existence of $h\in G_2$ satisfying  $[G_2: G_2 \cap h \delta_i(S)h^{-1}]<\infty$. This further implies that $[B_2: \delta_i(S)]<\infty$, as desired. \end{proof}

\begin{corollary}\label{embedding3}
For $i=1,2$ let $G_i = A_i\wr B_i$, where $A_i$ is an infinite abelian group, $B_i = B_i^1 \times B_i^2$ and for each $j=1,2$ we have $B_i^j = A_i^j\ast T_i^j$ with $A_i^j$ any infinite, finitely generated, amenable group and $T_i^j$ any torsion-free, hyperbolic group with property (T).  Suppose $t>0$ and $\theta:\Ll(G_1) \rightarrow \Ll(G_2)^t$ is any $*$-embedding.  \noindent Then $t\in \mathbb N$ and one can find a unitary $u\in \mathcal L(G_2)^t$ such that the restriction ${\rm ad}(u)\circ \Theta : \mathcal L(B_1)\rightarrow \mathcal L(B_2)^t$ is a group-like $\ast$-embedding.
\end{corollary}

\begin{proof} This follows immediately from the previous result by noticing that any free product between an amenable group and a hyperbolic groups is biexact \cite[Proposition 12]{Oza05}  and also weakly amenable, \cite{Ver22}.\end{proof}

\begin{theorem}\label{embedding4}
    For every $1 \leq i \leq 2$, let $G_i$ be a group hyperbolic relative to an exact, residually finite icc subgroup $H_i<G_i$. Also assume that $H_i =\mathbb Z \wr (B^i_1 \times B^i_2)$ for non-amenable icc groups $B_1^i, B_2^i$. Let $t>0$ and let $\Theta : \mathcal{L}(G_1) \rightarrow \mathcal{L}(G_2)^t$ be any $* -$embedding. Then one can find an integer $n$ and a unitary $u\in \Ll (G_2)\otimes \mathbb M_n(\mathbb C)$ such that ${\rm ad}(u)\circ \Theta : \Ll (H_1) \rightarrow \Ll(H_2)^t$ is a $\ast$-embedding. \\In particular, we have $\mathcal F_s(\Ll(G))\subseteq \mathcal F_s(\Ll(H))$. Thus, when $\mathcal F_s(\Ll(H))=\mathbb N$ we also have $\mathcal F_s(\Ll(G))=\mathbb N$.

\end{theorem}

\begin{proof} Realize $\Ll(G_2)^t= p(\Ll (G_2)\otimes \mathbb  M_n(\mathbb C))p$, where $n$ is an integer with $t\leq n$ and $p\in \Ll (G_2)\otimes \mathbb  M_n(\mathbb C)$ is a projection with $(\tau \otimes Tr)(p)=t$. Let $\Theta: \Ll(G_1)\rightarrow p(\Ll (G_2)\otimes \mathbb  M_n(\mathbb C))p$
be a unital $\ast$-embedding. 

  Since $B^1_1$ is infinite we have that $\A:=\Theta(\Ll(B^1_1))\nprec \mathbb M_n(\mathbb C)$. Moreover, as $\Ll(B_2^2)\otimes \mathbb M_n(\mathbb C)$ is a II$_1$ factor, using part 3) in Theorem \ref{commut2} one can find a unitary $u\in \Ll(G_2)\otimes \mathbb M_n(\mathbb C)$ satisfying $u\Theta (\Ll(B^1_2\times B^2_2))u^*\subseteq \Ll(H_2)\otimes \mathbb M_n(\mathbb C)$. 
 Since $u\Theta (\Ll(B_1^1\times B_2^1))u^*\nprec 1\otimes \mathbb M_n(\mathbb C)$ then Theorem \ref{peripheral absorption} implies that $u\Theta (\Ll(H_1))u^*\subseteq \Ll(H_2)\otimes \mathbb M_n(\mathbb C)$. Letting $q= u\Theta(1)u^*= upu^*$ we get a $\ast$-embedding ${\rm ad}(u)\circ \Theta : \Ll (H_1) \rightarrow q (\Ll(H_2)\otimes \mathbb M_n(\mathbb C))q=\Ll(H_2)^t$.

The second part of the statement follows immediately from the first.
\end{proof}

\begin{corollary}\label{isomperipheralalg}
    For every $1 \leq j \leq 2$ let $G_j$ be a icc group that is  hyperbolic relative to a family of icc, exact ressidually finite subgroups $\{H^j_1,\ldots,H^j_{n_j}\}$. Also assume that for every $1\leq j\leq 2$ and  $1\leq k\leq n_j$ we have $H^j_k = H_1^{j,k} \times H_2^{j,k}$, where $H_i^{j,k}$ are non-amenable groups. Let $0 \leq t \leq 1$ and let $\Theta : \mathcal{L}(G_1)^t \rightarrow \mathcal{L}(G_2)$ be a $* -$ isomorphism. \\Then $n_1 = n_2 $ and there exist a permutation $\sigma \in \mathfrak S_{n_1},$ unitaries $u_k \in \mathcal{L}(G_2)$ such that $u_k \Theta(\mathcal{L}(H_k^1))^t u_k^* = \mathcal{L}(H_{\sigma(k)}^2)$ for all $1 \leq k \leq n_1$.
\end{corollary}

\begin{proof}
    Fix $1 \leq k \leq n_1 .$ Since $H_1^{1,k}$ is icc, there exists $p_k \in \mathcal{L}(H^{1,k}_1)$ a projection with $\tau (p_k) = t.$ From hypothesis we have that $\theta (p_k \mathcal{L}(G_1)p_k) = \mathcal{L}(G_2).$ Since $QN_{G_1}(H_1^{1,k})= H^1_k$ we have that $\mathcal {QN}''_{\Theta (p_k \Ll(G_1)) p_k}(\Theta (p_k \Ll(H^{1,k}_1)) p_k)=\Theta (p_k \Ll(H^1_k) p_k)$.  Since $H_k^1$ and $H^2_k$  are icc, using Theorem \ref{commut2} there are $v_k \in \mathscr{U}(\mathcal{L}(G_2))$ and $1\leq \sigma (k) \leq n_2$ such that
    \begin{equation}\label{eq1s5}
        v_k \Theta (p_k \mathcal{L}(H_k^1)p_k) v_k^{*} \subseteq \mathcal{L}(H_{\sigma(k)}^2).
    \end{equation}
    Since $\theta^{-1}(\mathcal{L}(G_2)) \subseteq p_k \mathcal{L}(G
_1)p_k \subseteq \mathcal{L}(G_1)$, then using a similar argument in combination with Theorem \ref{commut2} one can find $w_k \in \mathscr{U}(\mathcal{L}(G_1))$ and $1 \leq \tau (\sigma (k)) \leq n_1 $ such that 
\begin{equation}\label{eq2s5}
    w_k \Theta^{-1}(\mathcal{L}(H_{\sigma (k) }^2))w_k^* \subseteq \mathcal{L}(H_{\tau (\sigma (k) )}^1)
\end{equation}
\eqref{eq1s5} and \eqref{eq2s5} imply that 
$ w_k \Theta^{-1}(v_k)p_k\mathcal{L}(H_k^1)p_k \Theta^{-1}(v_k^*)w_k^* \subseteq \mathcal{L}(H_{\tau (\sigma(k))}^1). $ Using Proposition \ref{nonintrhg} and Theorem \ref{quasinormalizercontrol} this further implies that $\tau (\sigma (k))= k$ and also $w_k \Theta^{-1}(v_k)p_k \in \mathcal{L}(H_k^1).$

Furthermore, combining this with \eqref{eq1s5} and \eqref{eq2s5} yields $v_k \Theta (p_k \mathcal{L}(H_i^1)p_k)v_k^* = \mathcal{L}(H_{\sigma(k)}^2).$ Since this holds every $k$, we have that $n_1 = n_2$ and $\sigma$ is a permutation of $\{1,\ldots,n_1\}.$
\end{proof}

\begin{corollary}\label{nonisom1}
    Let $\mathscr{B}$ be the category of groups $B= (A_1 \ast T_1)\times (A_2\ast T_2)$  with $A_i$ is nontrivial, ressidually finite and amenable and $T_i$ is ressidually finite, hyperbolic property (T) group.   For every $B\in \mathscr{B},$ let $G_B>\mathbb Z\wr B$ be any group that is hyperbolic relative to $\mathbb Z\wr B$.  Then $\mathscr G:=\{ \mathcal{L}(G_B)^t \; : \; B\in \mathscr{B}, t>0 \}$ consists of pairwise non-isomorphic II${_1}$ factors such that $\mathcal F_s(\Ll(G_B))=\mathbb N$. 
    
    \noindent Moreover, whenever $B_1, B_2\in \mathscr B$ are noncommensurable and $t_1,t_2>0$ then $\mathcal L(G_{B_1})^{t_1}$ is not virtually isomorphic to $\mathcal L(G_{B_2})^{t_2}$. 
\end{corollary}

\begin{proof}
    First, we notice that $\mathscr W=\{\mathbb Z\wr B\,:\,B\in\mathscr B\}$ consists of residually finite icc groups. Fix $K_1= \mathbb Z \wr B_1$, $K_2=\mathbb Z \wr B_2\in \mathscr W$ and assume there is $\theta : \mathcal{L}(G_{B_1})^{t_1} \rightarrow \mathcal{L}(G_{B_2})^{t_2}$ a $*-$isomorphism. Without any loss of generality we can assume that $t_2 = 1.$ Using Corollary \ref{isomperipheralalg} we have that $\mathcal{L}(K_1)^{t_1} \cong \mathcal{L}(K_2).$  However, using Popa's strong rigidity result, \cite[Theorem 0.1]{Pop03}  we must have that $t_1=1$ and $B_1 \cong B_2.$ Hence $t=1$ and $K_1 \cong K_2$  which yields the first part of the conclusion. The second part follows immediately from Theorem \ref{embedding2}, Corollary \ref{embedding3}, and Theorem \ref{embedding4}.

    For the moreover part, assume $\theta : \mathcal{L}(G_{B_1})^{t_1} \rightarrow \mathcal{L}(G_{B_2})^{t t_2}$ is a virtual $\ast$-isomorphism, i.e. $\theta (\mathcal L(G_{B_1}))^{t_1}\subseteq \mathcal L(G_{B_2})^{tt_2}$ has finite index. Using amplifications and changing $t$, if necessary, we can assume that $t_1=t_2=1$. Thus $t\in \mathbb N$ and from Theorems \ref{embedding4} and \ref{embedding2} we get that $B_1$ is commnesurable with $B_2$, as desired.  \end{proof}
We explain how the prior result can be used in combination with results in geometric group theory to provide new types of property (T) group factors that are pairwise non-isomorphic thus adding to the prior results, \cite{CH89, OP03, CDK19, CIOS24}. Moreover, these factors also have trivial one-sided fundamental semigroup, \cite{CDHK20, CIOS24}.

First, a recent result of Nikolov and Segal, \cite[Theorem 2]{NS21} shows there exist $2^{\aleph_0}$ pairwise non-isomorphic $4$-generator residually finite solvable groups of derived length $4$ which appear as quotients of $\mathbb Z \wr (C_2 \wr C_2 \wr C_\infty)$. Denote this family of groups by $\mathscr A_1$. Notice that Proposition \ref{contnoncommens} implies that there is a continuum subcollection $\mathscr A_0\subseteq \mathscr A_1$ consisting of pairwise non-commensurable groups. Fix a nonelementary group $H$ that is ressiually finite, hyperbolic and property (T); e.g. any uniform lattice $H< Sp(n,1)$ for $n\geq 2$. Then for $0\leq i\leq 1$ consider the groups $\mathscr S_i:=\{(A\ast H) \times (\mathbb Z \ast H) \,:\, A\in \mathscr A_i\}\subset\mathscr B$. When $i=1$ (resp.\ $i=0$) consists of a continuum of finitely generated, pairwise non-isomorphic (resp.\ non-commensurable) groups.  Fix $B\in \mathscr S_1$. Using \cite[Theorem 1.1]{AMO} one can find an icc property (T) supragroup $G_B > \mathbb Z\wr B$ which is hyperbolic relative to $\{\mathbb Z\wr B\}$. In particular, this construction implies that the class $\mathscr G$  in Corollary \ref{nonisom1} covers a continuum of relative hyperbolic groups with property (T).

\begin{corollary}
There exists a continuum $\mathscr{R}$ of icc relative hyperbolic groups with property (T) such that $\{\Ll(G)\}_{G \in \mathscr {R}}$ consists of pairwise non-virtually isomorphic II$_1$ factors such that $\mathcal F_s( \Ll(G))=\mathbb N$ (and hence $\mathcal F(\Ll(G))=1$) for every $G \in  \mathscr{R}$. 
\end{corollary}

\subsection{ A class of \texorpdfstring{$\mathcal{HT}$}{HT} factors} In this subsection, we compute the one-sided fundamental semigroup for a family of group $\mathcal {HT}$ factors in the sense of Popa. Unlike the fundamental results from \cite{Pop01} our factors are not group measure space constructions and implicitly do not appeal to computations of the fundamental groups of equivalence relations. Instead, we will use our prior relative solidity results together with height techniques \cite{IPV10, PV21},  and the notion of standard embedding as well as some techniques from \cite{PV21} that will be pointed out along the way. To properly state our main result, we need to recall a few concepts. 

The first is Popa-Vaes' notion of \emph{standard emebbeding} of a group into the unitary group of an amplification of a group von Neumann algebra, \cite{PV21}. Towards this, we will first establish some notation, and we will recall a basic construction. 

Let $G$ be a countable group and let $\mathcal G$ be any group. Assume that $\mathcal G_0\leqslant \mathcal G$ is a finite index subgroup with $[\mathcal G:\mathcal G_0]=m$. For every group homomorphism $\delta : \mathcal G_0 \rightarrow G$ and finite-dimensional unitary representation $\gamma : \mathcal G_0 \rightarrow \mathscr U(n)$ we denote by
$\pi_{\gamma,\delta} : \mathcal G_0 \rightarrow \mathscr U( \mathcal L(G) \otimes \mathbb M_n(\mathbb C))$ letting $\pi_{\gamma,\delta}(g) = u_{\delta(g)} \otimes \gamma(g)$ for all $g\in \mathcal G_0$. 

If $\mathcal N$ is a finite von Neumann algebra and $\pi_0 : \mathcal G_0 \rightarrow \mathscr U(\mathcal N)$ is a group homomorphism  we choose representatives $g_1, \ldots , g_m \in \mathcal G$ for the cosets $\mathcal G/\mathcal G_0$ and get a natural induction
$\pi:\mathcal G\rightarrow \mathscr U( \mathcal N \otimes \mathbb M_m (\mathbb C))$ defined as follows ${\pi (g)}_{ij} =\pi_0(h)$ if $gg_j =g_ih$ for some $h\in \mathcal G_0$ and ${\pi (g)}_{ij}=0$ if $gg_j \nin g_i\mathcal G_0$.

\begin{definition}[Definition 2.1 in \cite{{PV21}}]\label{standard} Let $G$ be an icc group and $t > 0$. Let $\mathcal G$ be a group. We say that a homomorphism $\pi : \mathcal G \rightarrow \mathscr U(\mathcal L(G)^t)$ is \emph{standard} if $t \in \mathbb N$ and if $\pi$ is unitarily conjugate to a finite direct sum of inductions to $G$ of homomorphisms of the form $\pi_{\gamma,\delta} : \mathcal G_0 \rightarrow \mathscr U( \mathcal L(G) \otimes \mathbb M_n(\mathbb C))$ where $\mathcal G_0 \leqslant\mathcal G$ is a finite index subgroup, $\delta : \mathcal G_0 \rightarrow G$ is a group homomorphism and $\gamma: \mathcal G_0 \rightarrow  \mathscr U(n)$ is an $n$-dimensional unitary representation.    
\end{definition}



The second concept is the Rips construction due to Belegradek-Osin. Following \cite[Theorem 2.1]{BO06} for every finitely generated group $Q$, there exists a property $(T)$ group $N$ such that $Q\subset Out (N)$ is a finite index subgroup and there exists a crossed product  $G=N\rtimes_{\sigma}Q$ such that $N\rtimes_\sigma Q$ is hyperbolic relative to \{Q\}. Moreover, if $Q$ is torsion-free, we can pick $N$ to be torsion-free as well, and hence $N\rtimes_\sigma Q$ is torsion-free. We denote by $Rip(Q)$ the class of all crossed product $N \rtimes_{\sigma}Q.$ as before. 

\begin{theorem}\label{thmA}
    Let $Q=Q_1\times Q_2, P=P_1\times P_2$ with $Q_i, P_i$ non-amenable, torsion-free, residually finite, biexact groups with Haagerup property. Fix $ G=N \rtimes Q \in Rip(Q), H = M \rtimes P \in Rip(P).$\\
    Let $t>0$ and assume that $\Theta : \mathcal L(G) \rightarrow \mathcal L(H)^t$ be a $*$-embedding.

    \noindent Then $t\in \mathbb N$. Moreover, there is a unitary $u\in \mathcal L(H)^t$ such that 

    \begin{enumerate}
        \item ${\rm ad}(u)\circ \Theta (\mathcal L(N))\subseteq \mathcal L(M)^t$, ${\rm ad}(u)\circ \Theta (\mathcal L(Q))\subseteq \mathcal L(P)^t$, and        
        \item the embedding ${\rm ad}(u)\circ\Theta \,:\, Q \rightarrow \mathscr U(\mathcal L(P)^t)$ is standard in the sense described above.    \end{enumerate}

\end{theorem}

\begin{proof}
Let $\mathcal M = \mathcal L (M\rtimes P)$ and let $k\geq t$ be an integer. Then realize $\mathcal L (M\rtimes P)^t = p(\mathcal M \otimes \mathbb M_k(\mathbb C))p$ for a projection $p\in \mathcal M \otimes \mathbb M_k(\mathbb C)$ with $\tau \otimes Tr(p)=\frac{t}{k}.$ Put $\mathcal N:=\Theta (\mathcal L(N))$, $\mathcal G_i:=\{ \Theta (u_{g}) \,:\, g\in Q_i \}$ and  $\mathcal S_i:=\mathcal G_i ^{''}$. Since $\mathcal S_1, \mathcal S_2 \subset p(\mathcal M \otimes \mathbb M_k(\mathbb C))p $ are nonamenable, commuting II$_1$ factors, and $\mathcal L(P) \otimes \mathbb M_k(\mathbb C)$ is a II$_1$ factor, using Theorem \ref{commut2}, one can find a unitary $u\in \mathcal M \otimes \mathbb M_k (\mathbb C )$ such that $u(\mathcal S_1 \vee \mathcal S_2)u^* \subseteq \mathcal L(P)\otimes \mathbb M_k(\mathbb C).$ Also since $\mathcal N \subseteq p(\mathcal M \otimes \mathbb M_k(\mathbb C))p=(\mathcal L(M)\otimes \mathbb M_k(\mathbb C))\rtimes P$ is a property $(T)$ subalgebra and $P$ has Haagerup's property, by \cite{Pop01} we have  $\mathcal N \preceq^{\rm s}\mathcal L(M)\otimes \mathbb M_k(\mathbb C).$ Thus, replacing $\Theta $ by ${\rm Ad}(u)\circ \Theta$, we can assume without loss of generality that:

\begin{enumerate}
        \item $\mathcal S_1 \vee \mathcal S_2 \subseteq p(\mathcal L(P) \otimes \mathbb M_k(\mathbb C))p$ and
        \item $\mathcal N \preceq^{\rm s} \mathcal L(M) \otimes \mathbb M_k(\mathbb C)$

    \end{enumerate}
     Next, we show the following 
    \begin{claim}\label{containment2} $\mathcal N \subseteq \mathcal L(M) \otimes \mathbb M_k(\mathbb C).$
     \end{claim}

\noindent\emph{Proof of Claim \ref{containment2}.} Since $P_1, P_2$ are biexact groups, part (1) implies that $\mathcal S_1 \preceq \mathcal L (P_i)\otimes \mathbb M_k(\mathbb C)$ and $\mathcal S_2 \preceq \mathcal L (P_{\hat{i}})\otimes \mathbb M_k(\mathbb C)$ for some $1\leq i \leq 2.$\\
Let $z_i \in \mathscr Z((\mathcal S_1 \vee \mathcal S_2)'\cap p(\mathcal L(P)\otimes \mathbb M_k(\mathbb C))p)$ be maximal projections satisfying 
\begin{enumerate}
    \item [a)] $\mathcal S_1 z_1 \preceq^{\rm s} \mathcal L(P_i)\otimes \mathbb M_k(\mathbb C),$
    \item[b)] $\mathcal S_2z_2 \preceq^{\rm s} \mathcal L(P_{\hat{i}})\otimes \mathbb M_k(\mathbb C).$
\end{enumerate}
Since $\mathcal S_1(1-z_1), S_2(1- z_1)\subseteq \mathcal L(P_i)\otimes M_k(\mathbb C)$ are commuting, nonamenable von Neumann algebras, we have that 
\begin{enumerate}
    \item[c)]$\mathcal S_1(1-z_1)\preceq \mathcal L(P_i)\otimes \mathbb M_k(\mathbb C)$ or
    \item[d)] $\mathcal S_1(1-z_1)\preceq \mathcal L(P_{\hat{i}}) \otimes \mathbb M_k(\mathbb C).$
\end{enumerate}
Since $z_1$ is chosen maximal satisfying a), we have that c) cannot hold; hence we have d). If $z_2(1-z_1) \neq 0,$ then d), b) and \cite[Proposition 4.4]{CD-AD21} further imply that $(\mathcal S_1 \vee \mathcal S_2)z_2(1-z_1) \preceq \mathcal L(P_{\hat{i}})\otimes \mathbb M_k(\mathbb C).$ However, this contradicts that $P_{\hat{i}}$ is biexact. Hence $z_2(1-z_1)=0$ and thus  $z_2\leq z_1.$ A similar argument shows that $z_1(1-z_2)=0$ which implies $z_1\leq z_2,$ and hence $z_1 = z_2.$ In conclusion, there is $\{i_1, i_2\}$ a permutation of $\{1,2\}$ such that the following hold: 

\vspace{-13pt}
\begin{equation}\label{eq1s6}
    \begin{split}
         \mathcal S_1 z_1 & \preceq^{\rm s} \mathcal L(P_{i_1})\otimes \mathbb M_k(\mathbb C),\\
         \mathcal S_2z_1 & \preceq^{\rm s} \mathcal L(P_{i_2})\otimes \mathbb M_k(\mathbb C),\\
         \mathcal S_1(1-z_1) & \preceq \mathcal L(P_{i_2}) \otimes \mathbb M_k(\mathbb C),\\
        \mathcal S_2(1-z_1) & \preceq \mathcal L(P_{i_1})\otimes \mathbb M_k(\mathbb C).
    \end{split}
\end{equation}
\vspace{-13pt}

Henceforth, for every $g\in \mathcal G_1 \times \mathcal G_2$ let $\Psi_g = {\rm Ad}(g):\mathcal N \rightarrow \mathcal N$ be the induced a $*-$isomorphism.
Since $\mathcal N \preceq^{\rm s} \mathcal L(N) \otimes \mathbb M_k(\mathbb C)$ for every $\varepsilon > 0$ there there is a finite set $K \subset P$ such that 

\vspace{-13 pt}
\begin{equation}\label{sintertwining}\lVert \mathbb P_K(n) - n \rVert_2 \leq \varepsilon\text{ for all }n\in (\mathcal N)_1.\end{equation}
\vspace{-13 pt}

\noindent Consider the finite sets  $F \subset (P_1 \setminus \{1\})\times (P_2 \setminus \{1\})$ and $F_1, F_2 \subset P_1 \times P_2$ such that 

\vspace{-15 pt}
\begin{equation}\label{disjoint}Ks \cap t(K\setminus \{1\}) = \emptyset\text{ for every }s,t\in P_1 \times P_2 \setminus F_1(\bigcup_{r\in F}C_P(r))F_2.\end{equation} 
\vspace{-15 pt}

Here $C_P(r)$ is the the centralizer of $r$ in $P$. Fix $\varepsilon >0$. Since the centralizers of nontrivial elements of $P_i$ are amenable, biexactness of $P_i$ implies that $\mathcal S_1 \vee \mathcal S_2 \nprec \mathcal L(C_P(r))\otimes \mathbb M_k(\mathbb C)$ for all $r\in F$. Thus one can find $g=g_1g_2\in \mathcal G$ with $g_1\in \mathcal G_1, g_2\in \mathcal{G}_2$ such that $\lVert \mathbb P_{F_1(\bigcup_{k\in F}C(k))F_2}(g)\rVert_2\leq \frac{\varepsilon}{\lvert K\rvert^2}$. Thus, letting $g_{K} := g-\mathbb P_{F_1(\bigcup_{k\in F}C(k))F_2}(g)$, we have $\lVert g_K \rVert_{\infty}\leq 1+\lvert F \rvert  \lvert F_1 \rvert \lvert F_2 \rvert$ and also 

\vspace{-15 pt}
\begin{equation}\label{est1}\lVert g - g_K \rVert_2 \leq \frac{\varepsilon}{\lvert K \rvert^2}.\end{equation}
\vspace{-15 pt}

As  $\Psi_g(n)g = gn$ for all $n\in \mathcal{N}$, using the prior relations and basic norm computations we have 

\begin{equation*}
    \begin{split}
        4\varepsilon^2 & \buildrel \eqref{sintertwining} \over \geq \lVert \mathbb P_K(\Psi_g(n))g -g\mathbb P_K(n) \rVert_2^2 = \lVert \mathbb P_K(\Psi_g(n))g -g\mathbb P_{K\setminus 1}(n) -gE_{\mathcal{L}(M)\otimes \mathbb M_K(\mathbb C)}(n) \rVert_2^2\\
        & \buildrel (\perp) \over= \lVert \mathbb P_K(\Psi_g(n))g -g\mathbb P_{K\setminus 1}(n) \rVert_2^2 + \lVert\mathbb P_K(\Psi_g(n))g-gE_{\mathcal{L}(M)\otimes \mathbb M_K(\mathbb C)}(n) \rVert_2^2 -\| \mathbb P_K(\Psi_g(n))g\|^2_2\\
        &\geq \lVert \mathbb P_K(\Psi_g(n))g -g\mathbb P_{K\setminus 1}(n) \rVert_2^2 -\| \mathbb P_K(\Psi_g(n))g\|^2_2 \\
        &= \lVert \mathbb P_K(\Psi_g(n))(g-g_K) -(g-g_K)\mathbb P_{K\setminus 1}(n) + \mathbb P_K \Psi_g(n)g_K - g_K\mathbb P_{K\setminus 1}(n) \rVert_2^2 -\| \mathbb P_K(\Psi_g(n))g\|^2_2        \\
        & \buildrel\eqref{est1}\over \geq \left (\lVert \mathbb P_K(\Psi_g(n))g_K - g_K \mathbb P_{K\setminus 1}(n)\rVert_2 - \frac{2\varepsilon}{|K|}\right )^2 -\| \mathbb P_K(\Psi_g(n))g\|^2_2 \\
        &\buildrel \eqref{disjoint}\over = \left (\sqrt{\lVert \mathbb P_K(\Psi_g(n))g_K\rVert_2^2 + \lVert g_K\mathbb P_{K\setminus 1}(n)\rVert_2^2}-\frac{2\varepsilon}{|K|}\right )^2 - \lVert \mathbb P_K(\Psi_g(n))g\rVert_2^2 \\
        & \geq \lVert \mathbb P_K(\Psi_g(n))g_K\rVert_2^2 + \lVert g_K\mathbb P_{K\setminus 1}(n)\rVert_2^2 - 8\varepsilon + \frac{4\varepsilon^2}{|K|^2} - \lVert \mathbb P_K(\Psi_g(n))g\rVert_2^2\\
        & \buildrel \text{rev. triang. ineq.} \over \geq -2\lvert K \rvert\lVert \mathbb P_K(\Psi_g(n))(g-g_K)\rVert_2   + \lVert g_K\mathbb P_{K\setminus 1}(n) \rVert_2^2 -8\varepsilon + \frac{4\varepsilon^2}{|K|^2}\\
        & \buildrel \eqref{est1} \over \geq -2\lvert K \rvert^2.\frac{\varepsilon}{\lvert K \rvert^2} + \lVert g_K\mathbb P_{K\setminus 1}(n) \rVert_2^2 -8\varepsilon + \frac{4\varepsilon^2}{|K|^2}\\
        &\buildrel \eqref{est1} \&\text{ triang. ineq.} \over\geq  \left (\lVert \mathbb P_{K\setminus 1}(n) \rVert_2 -\frac{\varepsilon}{|K|} \right)^2 -10\varepsilon - \frac{4\varepsilon^2}{|K|^2}\\
        &\geq \lVert \mathbb P_{K\setminus 1}(n) \rVert^2_2- 12\varepsilon- \frac{3\varepsilon^2}{|K|^2}.
    \end{split}
\end{equation*}
Hence, $\lVert \mathbb P_{K\setminus 1}(x) \rVert_2^2 \leq 12\varepsilon +7\varepsilon^2.$ Using \eqref{sintertwining} together with  triangle inequality we get 
$\lVert n - E_{\mathcal L(M)\otimes \mathbb M_K(\mathbb C)}(n) \rVert_2 \leq \lVert n - \mathbb P_K(n) + \mathbb P_{K\setminus 1}(n) \rVert_2 \leq 13\varepsilon +7\varepsilon^2.$ As $\varepsilon > 0$ is arbitrary, we get $n\in \mathcal{L}(M)\otimes \mathbb M_k(\mathbb C)$ for every $n\in \mathcal{N}.$ Hence $\mathcal{N} \subseteq \mathcal{L}(M) \otimes \mathbb M_k(\mathbb C)$, finishing the proof of our claim. \hfill$\blacksquare$

\vskip 0.08in

In conclusion, part (1) and Claim \ref{containment2} entail $\mathcal{S}_1 \vee \mathcal S_2 \subset \mathcal L(Q) \otimes \mathbb M_k(\mathbb C)$ and $\mathcal N \subset \mathcal L(M)\otimes \mathbb M_k(\mathbb C)$; in particular, $p=\Theta (1) \in (\mathcal L(M)\otimes \mathbb M_k(\mathbb C)) \bigcap (\mathcal L(Q)\otimes \mathbb M_k(\mathbb C)) = \mathbb M_k(\mathbb C).$ Hence $t\in \mathbb N$ and thus by shrinking $k$ if necessary we can assume that $p=1.$ Moreover we have unital inclusions $$\mathcal{S}_1 \vee \mathcal S_2 \subset \mathcal L(P) \otimes \mathbb M_k(\mathbb C); \qquad \mathcal N \subset \mathcal L(M)\otimes \mathbb M_k(\mathbb C).$$ 
\vskip 0.03in
We continue with the following 

\begin{claim}\label{catomic} The relative commutant $\mathcal N'\bigcap (\mathcal M \otimes \mathbb M_k(\mathbb C))$ is completely atomic.

\end{claim}

\noindent\emph{Proof of Claim \ref{catomic}} Assume by contradiction there is a projection $0\neq z \in\mathscr Z(\mathcal N'\cap (\mathcal M \otimes \mathbb M_k(\mathbb C))$ such that $(\mathcal N'\cap (\mathcal M \otimes \mathbb M_k(\mathbb C)))z$ is diffuse. Since $\mathcal N z$ is a non-amenable factor, Theorem \ref{commut2} implies the existence of a unitary $u\in \mathcal M \otimes \mathbb M_k(\mathbb C)$ such that $u(\mathcal N \vee \mathcal N'\cap (\mathcal M \otimes \mathbb M_k(\mathbb C)))z u^*\subseteq \mathcal L(P)\otimes \mathbb M_k(\mathbb C)$. Since $\mathcal N \subset \mathcal L(M)\otimes \mathbb M_k(\mathbb C)$, using \cite[Lemma 2.8]{DHI19} we have $\mathcal N \prec \mathbb M_k(\mathbb C)$. Thus $\mathcal N$ has a nontrivial atomic corner, a contradiction. \hfill$\blacksquare$

\begin{claim}\label{finitesupport} There exists a finite set $F\subset P$ such that for every elements  $(x_h)_{h\in P} \subset (\mathbb M_k(\mathbb C))_1$ satisfying $x_hu_h \in \mathcal N^{'} \cap (\mathcal M \otimes \mathbb M_k(\mathbb C))$ we have $x_h = 0$ for all $h\in P \setminus F.$
\end{claim}
\noindent \emph{Proof of Claim \ref{finitesupport}.}
    Since by Claim \ref{catomic} $\mathcal N^{'} \cap (\mathcal M \otimes \mathbb M_k(\mathbb C))$ is completely atomic, for every $\varepsilon > 0$ there is a finite set $F_{\varepsilon} \subset P$ such that \begin{equation}\label{sintertwining3}\lVert x - \mathbb P_{F_{\varepsilon}}(x) \rVert_2 < \varepsilon \text{ for all }x\in (\mathcal N^{'} \cap (\mathcal M \otimes \mathbb M_k(\mathbb C)))_1.\end{equation}  Fix $(x_h)_{h\in P}\subset (\mathbb M_k(\mathbb C))_1$ with $x_hu_h \in \mathcal N^{'} \cap (\mathcal M \otimes \mathbb M_k(\mathbb C))$. Let $x_h = \lvert x_h^* \rvert v_h$ be its co-polar decomposition. As $x_hu_h \in \mathcal N^{'} \cap (\mathcal M \otimes \mathbb M_k(\mathbb C))$ then $x_hx_h^* \in \mathcal N^{'} \cap (\mathcal M \otimes \mathbb M_k(\mathbb C))$ and hence $\lvert x_h^* \rvert \in \mathcal N^{'} \cap (\mathcal M \otimes \mathbb M_k(\mathbb C)).$ Thus $v_hu_h \in \mathcal N^{'} \cap (\mathcal M \otimes \mathbb M_k(\mathbb C))$ for all $h\in P.$ \\
    Using \eqref{sintertwining3} for $0<\varepsilon <\frac{1}{\sqrt{2k}}$ and $x=v_hu_h$, one can find a finite set $F_\varepsilon\subset P$ such that for every $h\in P\setminus F_{\varepsilon}$ we have
    \begin{equation*}
        \begin{split}
            \frac{1}{2k} \geq \varepsilon^2 & \geq \lVert v_hu_h - \mathbb P_{F_{\varepsilon}}(v_hu_h) \rVert_2^2 = \lVert v_hu_h \rVert_2^2 - \lVert \mathbb P_{F_{\varepsilon}}(v_hu_h)\rVert_2^2\\
            & = \tau (v_h^* v_h) - \sum_{k\in F_{\varepsilon}}\lVert E_{\mathcal M \otimes \mathbb M_k(\mathbb C)}(v_hu_{hk^{-1}})\rVert_2^2\\
            & = \tau (v_h^* v_h) - \sum_{k\in F_{\varepsilon}} \delta_{hk^{-1},1}\lVert v_h \rVert_2^2\\
            & =\tau (v_h^*v_h).
        \end{split}
    \end{equation*}
Since $v_h^*v_h \in \mathbb M_k(\mathbb C)$ is a projection we get $v_h^*v_h = 0.$ Thus $v_h =0$ and hence $x_h =0$ for all $h\in P\setminus F_{\varepsilon}.$ \hfill $\blacksquare$
\vskip 0.09in

Fix $g\in \mathcal G$ and let $\Psi_g:\mathcal N \rightarrow \mathcal N$ be the isomorphism given by 
\begin{equation}\label{eq1sc6}
    \Psi_g(x)g=gx, \; \text{for every} \;x\in \mathcal N.
\end{equation}

Let $g=\sum_{h\in Q}x_hu_h$ with $x_h\in \mathbb M_k(\mathbb C)$ be the Fourier decomposition of $g.$ Thus $\Psi_g(x)x_h = x_h \sigma_h(x)$ for every $x\in \mathcal N, h\in P.$ Hence $\Psi_g(x)x_hu_h = x_hu_hx$ for every $x\in \mathcal N, h\in P.$ Using these for every $h_1, h_2 \in P$ we have $u_{h_1^{-1}}x_{h_1}^*x_{h_2}u_{h_2} \in \mathcal N^{'} \cap (\mathcal M \otimes \mathbb M_k(\mathbb C))$; equivalently, we have $\sigma_{h^{-1}_1}(x_{h_1}^*x_{h_2})u_{h_1^{-1}h_2}\in \mathcal N^{'}\cap(\mathcal M \otimes \mathbb M_k(\mathbb C))$ for every $h_1,h_2 \in P.$ By Claim \ref{finitesupport} there is a finite set $F\subset P$ (independent of $g$) such that $x_{h_1}^*x_{h_2}=0$ for all $h_1,h_2 \in P$ with $h_1^{-1}h_2 \notin F.$ Hence $x_{h_1}x_{h_1}^*x_{h_2}x_{h_2}^*=0$ for every $h_1^{-1}h_2 \notin F.$ Therefore, letting $e_h = sup (\lvert x_h^* \rvert),$ we get  
\begin{equation}\label{eq8sc6}
    e_{h_1}e_{h_2}=0\; \text{for every}\; h_1,h_2 \in P\; \text{with} \; h_1^{-1}h_2 \notin F.
\end{equation}

\begin{claim}\label{heightlowerbound} For every $g \in \mathcal G$ there is a finite set $F_g\subset P$ with $\lvert F_g \rvert \leq k\lvert F \rvert $ such that $g \subseteq \mathbb M_k(\mathbb C) F_g.$ This yields the following lower bound for the height of $\mathcal G$ with respect to $P$, \cite{IPV10}:
$$\inf_{g\in \mathcal G}\left (\max_{h\in P}|\tau( g u_h)|\right )=: h_P(\mathcal G)\geq \frac{1}{\sqrt{k|F|}}>0.$$
\end{claim}

\noindent \emph{Proof of Claim \ref{heightlowerbound}}.
  Let $g= \sum_h x_hu_h$ be the Fourier expansion and let $x_h = \lvert x_h^* \rvert v_h$ and $v_hv_h^*=e_h=supp(\lvert x_h^* \rvert).$ Let $h_1$ such that $x_{h_1}\neq 0.$ Thus $e_{h_1}\neq 0.$ If $supp (g) \subset h_1F^{-1}.$ Then $F_g =h_1F^{-1}$ will do the job. If not there is $h_2 \in P \setminus h_1F^{-1}$ with $x_{h_2}\neq 0$ and hence $e_{h_2}\neq 0$. From \eqref{eq8sc6} we have $e_{h_2}e_{h_1}=0.$ Now consider the set $\{h_1, h_2 \}F^{-1}.$ If $ supp (g) \subset \{ h_1, h_2\}F^{-1}$ then $F_g = \{ h_1, h_2 \}F^{-1}$ will do the job. If not, there is $h_3 \in P\setminus \{ h_1, h_2 \}F^{-1}$ with $x_{h_3}\neq 0$ and thus $e_{h_3}\neq 0$. From \eqref{eq8sc6} we have $e_{h_3}e_{h_2}=e_{h_3}e_{h_1}=0$ and hence $\{ e_{h_1}, e_{h_2}, e_{h_3} \}$ are mutually orthogonal projections. Continuing in this fashion, after $k$ steps, we have either $ supp (g) \subset \{ 
h_1, \dots, h_k \}F^{-1}$ or there is a projections $e_{h_{k+1}}$ such that $\{ e_{h_1}, \dots , e_{h_k},e_{h_{k+1}} \}$ are nonzero mutually orthogonal projections in $\mathbb M_k(\mathbb C).$ Since the later is obviously impossible, we conclude that $supp(g) \subseteq \{ h_1,\ldots,h_k\}F^{-1}:=F_g$, finishing the proof of the first assertion.
\vskip 0.05in

To see the remaining part, notice the first assertion implies that for every $g \in \mathcal G$ we have

\begin{equation*}1=\lVert g \rVert_2^2 = \sum_{h\in F_g} \lVert E_{\mathbb M_k(\mathbb C)}(gu_{h^{-1}}) \rVert_2^2 \leq |F_g| h_{P}^2(\mathcal G) \leq k|F| h^2_{P}(\mathcal G), \end{equation*}
and hence  $h_P(\mathcal G) \geq \frac{1}{\sqrt{\lvert F \rvert k}}>0.$ \hfill$\blacksquare$
\vskip 0.06in

To derive our conclusion, we will use, almost mot \`a mot, the technique introduced in  \cite[Lemma 2.4, Theorem 2.3]{PV21} to show the embedding $\mathcal G \subseteq \mathscr U(\mathcal L(Q) \otimes \mathbb M_k(\mathbb C))$ is standard as in Definition \ref{standard}. While we encourage the reader to consult these results beforehand, for convenience, we also include all the details pertaining to our specific situation. 
\vskip 0.06in
Let $\mathcal A_0$ be the $*-$algebra of all $a\in \mathcal L(P) \otimes \mathbb M_k(\mathbb C)$ such that the $span \{ gag^{-1} | g\in \mathcal G \}$ is finite dimensional. Consider $\mathcal A := \mathcal A_0 ^{''} \subseteq \mathcal L (P) \otimes \mathbb M_k(\mathbb C)$ and notice $\A$ is $\mathcal G-$invariant. Moreover, the action of $\mathcal G$ by conjugation on $\mathcal A$ is compact. Since $P=P_1 \times P_2$ with $P_i$ non-amenable biexact, it follows from \cite{CSU11} that $\mathcal A \preceq \mathcal L(P_{j}) \otimes \mathbb M_k(\mathbb C)$ for some $1\leq j \leq 2.$ By \cite[Lemma 2.4]{DHI19} there is a projection $z_3 \in \mathscr{Z}((\mathcal A \vee \mathcal G)^{'} \cap \mathcal L(P)\otimes \mathbb M_k(\mathbb C))$ such that 
\begin{equation}\label{eq2sc6}
   \mathcal A z_3 \preceq^{\rm s} \mathcal L(P_{j}) \otimes \mathbb M_k(\mathbb C).
\end{equation}

Using \eqref{eq1s6} there is a projection $z$ such that

\begin{equation}\label{eq3sc6}
    \begin{split}
        \mathcal S_1 z & \preceq \mathcal L(P_j) \otimes \mathbb M_k(\mathbb C) \; \text{and}\\
        \mathcal S_2 (1-z) & \preceq^{\rm s} \mathcal L(P_j) \otimes\mathbb M_k(\mathbb C). 
    \end{split}
\end{equation}

Hence we have either $z_3z \neq 0$ or $z_3(1-z) \neq 0.$ Without loss of generality, assume that $z_3z \neq 0.$ Then by \eqref{eq2sc6},\eqref{eq3sc6} and \cite[Proposition 4.3]{CD-AD21} we get that $(\mathcal A \vee \mathcal S_1)z_1z_3 \preceq^{\rm s} \mathcal L(P_j) \otimes \mathbb M_k(\mathbb C).$ Since $P_j$ is biexact we get that $\mathcal A$ is discrete. Thus center $\mathscr Z(\mathcal A)$ is atomic and $\mathcal G$ acts on $\mathscr Z(\mathcal A)$ by $\alpha .$ Thus we can find nonzero projections $(z_n)_{n\in I} \subset \mathscr Z(\mathcal A)$ such that:
\begin{itemize}
    \item $\sum_{n\in I}z_n =1$
    \item $\alpha(h)z_n = z_n$
    \item  $\mathscr Z(\mathcal A)z_n$ is finite dimensional and the fixed point of $\alpha|_{\mathscr Z(\mathcal A)z_n}$ is $\mathbb C z_n,$ for every $n \in I$.\end{itemize}
It suffices to prove $\pi: g \longmapsto gz_n$ is standard. Multiplying all data by $z_n$ we can assume that $\mathscr{Z} (\mathcal A)$ is finite-dimensional with $\mathscr{Z}(\mathcal A)^{\alpha}=\mathbb C 1.$\\
Fix $z\in \mathscr{Z}(\mathcal A)$ a minimal projection.  Consider the group $\mathcal G_0=\{ g\in \mathcal G \; : \; \alpha (g)z=z \}$ and notice it has finite index in $\mathcal G$. We find that $1=\sum_{g\in \mathcal G/\mathcal G_0}\alpha(g)(z)$ and that $\pi$ is induced from the $\pi_0 : \mathcal G_0 \rightarrow \mathcal L(P)\otimes \mathbb M_k(\mathbb C)$ with $\pi_0(g) = \pi(g)z.$ As $\mathcal G_0 \leqslant \mathcal G$ has finite index Claim \ref{heightlowerbound} still implies that $h_P(\pi_0(\mathcal G_0))>0$. In addition, $\pi_0(\mathcal G_0)$ still has no amenable direct summand. By construction we have $\mathcal Az \cong \mathbb M_l(\mathbb C)$. Thus we may realize $z(\mathcal L(P)^k)z$ as $ \mathcal L(P)^s \otimes \mathbb M_l(\mathbb C)$ in such a way that $\mathcal A z$ corresponds to $1\otimes \mathbb M_l(\mathbb C).$ Thus for every $g\in \mathcal G_0$ we have that $\alpha (g)$ restricts to an automorphism of $\mathbb M_l(\mathbb C).$ Hence we can pick unitaries $\gamma (g) \in \mathbb M_l(\mathbb C)$ and $\pi_1(g) \in \mathcal L(P)^s$ such that $\pi_0(g) =  \pi_1(g) \otimes \gamma(g).$ Since the unitaries $\gamma (g)$ are uniquely determined up to a scalar, then we can find $\mathcal G_3 = \mathbb T \pi_1 (\mathcal G_0)\leqslant \mathscr U(\mathcal L(P)^s)$ is a subgroup. One can also check that $h_P(\mathcal G_3)>0$ and $\mathcal G_3^{''}$ has no amenable direct summand. 
\vskip 0.06 in
With these notations at hand we prove the following \

\begin{claim}\label{nonintcentr1} $\mathcal G_3^{''} \npreceq \mathcal L(C_P(h))$ for every $h\in P\setminus \{ 1 \}.$\end{claim}
\noindent \emph{Proof of Claim \ref{nonintcentr1}.}
If $h=(h_1,h_2)$ with $h_1 \neq h_2 \neq 1,$ then $C_P(h)$ amenable as $P_i$ are biexact. Since $\mathcal G_3^{''}$ has no amenable summand the claim follows. \\
Now assume that $h_1 = 1$ or $h_2 = 1.$ Due to symmetry it is enough to treat only one case so assume wlog that $h_2=1.$ Hence $h_1 \neq 1$ and $C_P(h) = C_{P_1}(h_1)\times P_2.$ Letting $\mathcal G_0^{1} = \mathcal G_0 \cap \mathcal G_1, \mathcal G_0^2 = \mathcal G_0 \cap \mathcal G_2$ we can notice that $[\mathcal G_1 : \mathcal G_0^1] < \infty, [\mathcal G_2 : \mathcal G_0^2] < \infty.$ Hence $\mathcal S_0^1 := (\mathcal G_0^{1})^{''}\subseteq \mathcal S_1$ and $\mathcal S_0^2 := (\mathcal G^{2}_0)^{''} \subseteq \mathcal S_2$ are finite index von Neumann subalgebras and thus they have no amenable direct summand. \\
Assume by contradiction that $\mathcal G_3 ^{''}\preceq \mathcal L(C_{P_1}(h_1) \times P_2).$ Hence $\pi_1(\mathcal G_0^{1})^{''}\preceq \mathcal L(C_{P_1}(h_1)\times P_2)$ and $\pi_1(\mathcal G_0^2)^{''}\preceq \mathcal L(C_P(h_1)\times P_2).$ Since $\gamma (\mathcal G_0^{1}) , \gamma (\mathcal G_0^2) \subseteq \mathbb M_l(\mathbb C)$ this further implies that 
\begin{equation}\label{eq4sc6}
    \begin{split}
        \mathcal S_0^{1} & \preceq \mathcal L (C_{P_1}(h_1) \times P_2) \; \text{and}\\
        \mathcal S_0^2 & \preceq \mathcal L (C_{P_1}(h)\times P_2).
    \end{split}
\end{equation}
From \eqref{eq1s6} we have that
\begin{equation}\label{eq5sc6}
    \begin{split}
        \mathcal S_j z_j & \preceq^{\rm s}  \mathcal L(P_2)\otimes \mathbb M_k(\mathbb C)\\
        \mathcal S_{\hat{j}}(1-z_j) & \preceq^{\rm s}  \mathcal L(P_2) \otimes \mathbb M_k(\mathbb C)
    \end{split}
\end{equation}
Thus \eqref{eq4sc6} and \eqref{eq5sc6} implies that either $\mathcal S_j \preceq \mathcal L(C_{P_1}(h_1))$ or $\mathcal S_{\hat{j}} \preceq \mathcal L(C_{P_1}(h_1)).$ Since $C_{P_1}(h_1)$ are amenable these lead to a contradiction.\hfill$\blacksquare$

\begin{claim}\label{mixing} The representation given by ${\rm Ad}(g) = \alpha_g: \mathcal G_1 \curvearrowright L^2(\mathcal L(P)^s \ominus \mathbb C1)$ is weak mixing.\end{claim}
\noindent \emph{Proof of Claim \ref{mixing}.}
    This follows mot \`a mot as in the last part of the proof of \cite[Theorem 2.3]{PV21}. 
    \hfill$\blacksquare$

    Thus by \cite[Lemma 2.4]{PV21} we get that $s=1$ and 
    $\mathcal G_1 \subseteq \mathbb T P.$ Hence $\mathcal L(P)^kz =  \mathcal L(P)\otimes \mathbb M_l(\mathbb C)$ and $\pi_0(g) =  u_{\delta (g)}\otimes \gamma(g)$ for every $g\in \mathcal G_0.$ This forces $\gamma : \mathcal G_0 \rightarrow \mathscr U(\mathbb C^l)$ to be a unitary representation and $\delta: \mathcal G_0 \rightarrow P$ be a group homomorphism. Since $\Pi$ is an induction of $\pi_0,$ the statement follows. \end{proof}

Finally we end this section by recording the following immediate consequence of the previous result. 
\begin{corollary}
  Let $G=N\rtimes Q\in Rip(Q)$ as in Theorem \ref{thmA}. Then $\mathcal F_s(\mathcal L (G))=\mathbb N.$ In particular, $\mathcal F(\mathcal L(G))=1$ and $\mathcal I(\mathcal L(G))\subseteq \overline{\mathbb N}.$ If we let $Q_1=Q_2=\mathbb F_n$ for $n\geq 0$, then we get $\mathcal I(\mathcal L(G))= \overline{\mathbb N}.$
\end{corollary}

\section{An infinite family of non-stably isomorphic group factors. } 



In this short section we use Belegradek-Osin' Rips construction \cite{BO06}  in conjunction with some von Neumann algebra techniques to construct an infinite family of commensurable icc hyperbolic groups, whose II$_1$ factors are pairwise non isomorphic.   Interestingly, these also appear as $\mathcal{HT}$ factors, \cite{Pop01}.
\vskip 0.08in


\begin{theorem}\label{thm1sc7}
    Let $Q$, $P$ be any finitely generated icc  groups in one of the classes:
    \begin{enumerate}
    \item infinite Haagerup groups,
    \item non-canonical amalgamated free products over a Haagerup subgroup, or 
    \item graph products of Haagerup groups.\end{enumerate}
    Let $G=N\rtimes Q \in Rip(Q)$, $ H=M\rtimes P \in Rip(P)$ and put $\mathcal N = \mathcal L(N\rtimes Q)$, $\mathcal M =\mathcal L(M\rtimes P).$ 
    
   \noindent  Let $0\leq t \leq 1$ and assume that $\Theta : \mathcal N \rightarrow \mathcal M^t$ is a $*-$isomorphism. 
    \vskip 0.03in
    \noindent Then one can find a unitary $u \in \mathcal M^t ,$  a projection $p \in \mathcal L(M)$ with $\tau (p)=t,$ a group virtual isomorphism $\delta : Q \rightarrow P$, and a map  $\eta : \delta (Q) \rightarrow \mathscr{U}(p\mathcal L(M) p)$ such that \begin{enumerate} 
    \item  $ {\rm ad} (u)\circ \Theta (u_g) = \eta (g) v_{\delta(g)}$ for every $g\in Q$;    
    \item $({\rm ad}(u)\circ \Theta) (\mathcal L(N)) = p\mathcal L(M)p$;    \item $\eta (g) \sigma_g(\eta (h))=\eta (gh), \eta(g) \eta(g)^* = p, \eta (g)^*\eta (g) = \sigma_g(p) \; \text{for every} \; g,h \in \delta (Q)$. 
    \end{enumerate}
    
    \noindent If $t=1$ then $ \delta$ is a group isomorphism.
\end{theorem}

\begin{proof}
Let  $p \in \mathcal L(M)$ be a projection and let $\Theta : \Nn \rightarrow p\M p$ be a $*-$isomorphism. Thus $\Theta (\mathcal L(N)) \subseteq p\mathcal L(M\rtimes P)p$ has relative property (T) and since $P$ is amalgamated free products of Haagerup groups, it follows that $\Theta (\mathcal{L}(N)) \preceq^{\rm s} p\mathcal L(M)p.$ A symmetric argument shows that $p\mathcal L(M)p \preceq^{\rm s} \Theta (\mathcal L(M)).$ Since $Q, P$ are icc groups it follows from \cite[Lemma 8.4]{IPP05} 
(see also \cite[Theorem 6.1]{CDD22})   that one can find a unitary $u \in \mathcal M^t$ such that 
\begin{equation}\label{eq4sc7}
    u\Theta (\mathcal L(N))u^* = p\mathcal L(M)p.
\end{equation}
Replacing $\Theta$ by $ad(u)\circ \Theta,$ we can assume that $\theta (\mathcal L(N))= p\mathcal L(M)p.$ Let $\Psi_g = ad(\Theta (u_g)):\Theta (\mathcal L (N)) \rightarrow \Theta (\mathcal L (N)) *-$automorphism. Then we have $\Psi_g(y) \Theta (u_g) = \Theta (u_g) y$ for every $y\in \Theta (\mathcal L(N)).$ Now consider Fourier expansion $\Theta (u_g)=\sum_{h\in P}x_hv_h.$
Using this, we get 
\begin{equation}\label{eq1sc7}
\Psi_g(y)x_h =x_h \alpha_h(y) \; \text{for every} \; y\in \Theta (\mathcal L(N)).
\end{equation}
Since $p\Theta (u_g) p = \Theta (u_g)$ we have that  $px_h\sigma_h(p)=x_h$ for all $h\in P.$ Using relation \eqref{eq1sc7} we have that $x_k^*x_h = \alpha_k(y^*)x_k^*x_h\alpha_h(y),$ and hence $v_{k^{-1}}x_k^*x_hv_h = y^*v_{k^{-1}}x_k^*x_hv_hy$ for every $y\in p\mathcal L(M) p.$ This implies that $v_{k^{-1}}x_k^*x_hv_h \in (p\mathcal L(M) p)^{'}\cap p\mathcal M p =\mathbb C p,$ and hence $v_{k^{-1}}x_k^*x_hv_h = \lambda_{k,h}p$ for some $\lambda_{k,h} \in \mathbb C.$ Thus we have $\lambda_{k,h} p = v_{k^{-1}}x_k^*x_hv_h =\alpha_k (x_k^* x_h)v_{k^{-1}h}$ for every $k,h\in P.$ Taking the expectation onto $\mathcal L(M)$ we see that $\lambda_{k,h}=0$ for every $k\neq h,$ and this implies 
\begin{equation}\label{eq2sc7}
    x_k^*x_h=0 \; \text{for all} \; k\neq h.
\end{equation}
Moreover, \eqref{eq1sc7} implies that $x_h x_h^* \in \Theta (\mathcal L(N))^{'} \cap \Theta (\mathcal N) = (p\mathcal L(M) p)^{'} \cap p \mathcal M p =\mathbb C p;$ then $x_h x_h^* = \mu_h p$ for some $\mu_h \in \mathbb C.$ Using this in combination with \eqref{eq2sc7} we get $0=x_kx_k^*x_hx_h^* = \mu_k \mu_h p,$ and hence $\mu_h =0$ or $\mu_k = 0$ for all $h,k.$ In conclusion for every $g\in Q$ there exists unique $\delta (g) \in P$ such that
\begin{equation}\label{eq3sc7}
    \Theta (u_g) = x_{\delta (g)}v_{\delta (g)}.
\end{equation}

Clearly, since $x_{\delta(gl)}v_{\delta(gl)} = \Theta (u_{gl})=\Theta (u_g)\Theta (u_l) = x_{\delta(g)}v_{\delta (g)}x_{\delta (l)}v_{\delta((l)},$ we get\\
$$x_{\delta (g)}\alpha_{\delta (g)}(x_{\delta(l)})v_{\delta(g) \delta (l)} = x_{\delta (gl)}v_{\delta (g)\delta(l)}$$

This implies $\delta: Q \rightarrow P$ is a group homomorphism and $\eta: \delta (Q) \rightarrow \mathscr{U}(p\mathcal L(M) p)$ is a generalized 1-cocycle. Furthermore, \eqref{eq3sc7} and \eqref{eq4sc7} together imply that $\mathcal M^t = \mathcal L (N\rtimes Q) \subseteq \mathcal L (M) \rtimes P$ and hence $[P:\delta(Q)]<\infty.$\\
Now let $K=ker(\delta)<Q.$ Therefore, \eqref{eq3sc7} and \eqref{eq4sc7}
imply that $\Theta (\mathcal L(N\rtimes K)) \subseteq p\mathcal L(M)p$ and also $\Theta (\mathcal L(N)) = p\mathcal L(M) p.$ In particular, we have $\Theta (\mathcal L(N\rtimes K)) \preceq_{\Theta (\mathcal N)} \Theta (\mathcal L (N))$ which implies that $\mathcal L(N\rtimes K) \preceq_{\mathcal N} \mathcal L(N).$ Thus $K$ is a finite group. In conclusion, $\delta: Q \rightarrow P$ is a virtual isomorphism. \end{proof}

We note that \cite[Theorem 6.4]{CH89} and \cite[Theorem 0.2]{IPP05} already revealed infinite families of icc hyperbolic groups whose factors are pairwise non-isomorphic; e.g.\ any collection of uniform lattices $H< Sp(n,1)$ for infinitely many different $n\geq 2$ or any collection of free products $H_1 \ast \cdots \ast H_n$ of property (T) hyperbolic groups $H_i$ for infinitely many values of  $n\geq 1$.   However, prior results (\cite[Theorem 0.4]{IPP05}, \cite[Theorem 1.10]{Jol14}) also show that both these classes consist of pairwise non-measure equivalent groups. In sharp contrast, using Theorem \ref{thm1sc7}, we can construct infinitely many icc hyperbolic groups which are all commensurable to each other (and thus measure equivalent) but still yield pairwise non-isomorphic factors. This should be compared with \cite[Theorem]{CI09}, \cite[Corollary C]{CdSS15} or  \cite[Corollary 6.2]{CDK19}.

\begin{corollary} Let $n\geq 2$ be an integer and fix a hyperbolic group $G\in Rip(\mathbb F_n)$. Then for every $k\in \mathbb N$ there is a subgroup $G_k \leqslant G$ with $ [G:G_k]=k$ such that the collection $\{\mathcal L(G_k)  \,: \, k\in \mathbb N\}$ consists of pairwise non-isomorphic  factors.  
    
\end{corollary}
\begin{proof} Let $G=N \rtimes \mathbb F_n$. Since $G$ is hyperbolic relative to $\mathbb F_n$ \cite{BO06} and the latter is hyperbolic it follows from \cite[Corollary 1.14]{DS05} that $G$ is hyperbolic.  By Schreier theorem, for every $k\in \mathbb N$ there is a free subgroup $\mathbb F_{k(n-1)+1}< \mathbb F_n$ with $[\mathbb F_n:\mathbb F_{k(n-1)+1}]=k$. Now consider the semidirect product restriction subgroup $G_k:=N\rtimes \mathbb F_{k(n-1)+1}<G$ and notice $[G:G_k]=[\mathbb F_n:\mathbb F_{k(n-1)+1}]=k$. Since $G$ is hyperbolic so are all subgroups $G_k$.  Then Theorem \ref{thm1sc7} yields the desired conclusion.    \end{proof}

\begin{corollary}
    For every prime $p\geq 3$ fix a hyperbolic group $G_p\in Rip(\mathbb Z_p\ast \mathbb Z_p)$. Then the collection $\{\mathcal L(G_p)  \,: \, p \text{ prime}\}$ consists of pairwise non-stably isomorphic  factors.  
    
\end{corollary}
\begin{proof} As $\mathbb Z_p \ast \mathbb Z_p$ is hyperbolic, the same argument from the previous proof shows that $G_p$ is icc and hyperbolic. Notice for every primes $p\neq q$, any group homomorphism between $\mathbb Z_p \ast \mathbb Z_p \rightarrow \mathbb Z_q\ast \mathbb Z_q$ is trivial; in particular, $\mathbb Z_p \ast \mathbb Z_p$ is never virtually isomorphic to $\mathbb Z_q \ast \mathbb Z_q$. The conclusion then follows from Theorem \ref{thm1sc7}. \end{proof}

We end by highlighting a couple of closely related, intriguing  questions which we were not able to address in this investigation.

\begin{problem}
    Find an example of icc hyperbolic group $G$ such that $\mathcal F(\mathcal L(G))=1$. How about $\mathcal F^f_s(\mathcal L(G))=\mathbb N$ or $\mathcal F_s(\mathcal L(G))=\mathbb N$? Conjecturally, the first assertion should hold for all hyperbolic groups $G$.\end{problem}

\begin{problem}
    Does there exist an infinite family of pairwise commensurable icc hyperbolic groups whose factors are pairwise non-stably isomorphic?\end{problem}
We end this paper with the following intriguing question. 
\begin{problem} Does there exist a non-elementary icc hyperbolic group $G$ that is W$^*$-superrigid in the sense of Popa, \cite{Pop06}.    
\end{problem}


\begin{thebibliography}{prime}















		
 \bibitem[AMO07]{AMO} G. Arzhantseva, A. Minasyan, and D. Osin, The SQ-universality and residual properties of relatively hyperbolic groups, \emph{J. Algebra}, {\bf 315} (2007), no. 1, 165--177.

\bibitem[BC14]{BC14} R. Boutonnet, and A. Carderi. Maximal amenable von Neumann subalgebras arising from maximal amenable subgroups, \emph{Geom. Funct. Anal.} \textbf{25} (2015), no. 6, 1688-1705.

 \bibitem[BHR12]{BHR12} R. Boutonnet, C. Houdayer, and S. Raum, Amalgamated free product type III factors with at most one Cartan subalgebra, \emph{Compos. Math.} \textbf{150} (2014), no. 1, 143-174.
 
 \bibitem[BO06]{BO06} I. Belegradek, and D. Osin, Rips construction and Kazhdan property (T), \emph{Groups, Geom., Dynam.}, \textbf{2} (2008), 1-12.

\bibitem[Bou12]{Bou12} R. Boutonnet, On solid ergodicity for Gaussian actions, \emph{J. Funct. Anal.}, \textbf{263} (2012), no. 4, 1040-1063.

\bibitem[Bou14]{Bou14} R. Boutonnet, Several rigidity features of von Neumann algebras, \emph{PhD diss., Ecole normale supérieure de lyon-ENS LYON}, (2014).

\bibitem[BV13]{BV13} M. Berbec, and S. Vaes, W$^*-$superrigidity for group von Neumann algebras of left–right wreath products, \emph{Proc. Lond. Math. Soc.}, \textbf{108} (2014), no. 5, 1116-1152.

\bibitem[CD-AD20]{CD-AD20} I. Chifan, A. Diaz-Arias, and D. Drimbe, New examples of W$^*$ and C$^*$-superrigid groups, \emph{Adv. Math.}, \textbf{412} (2023), 108797.

\bibitem[CD-AD21]{CD-AD21} I. Chifan, A. Diaz-Arias, and D. Drimbe, W$^*$ and C$^*$ -superrigidity results for coinduced groups, \emph{J. Funct. Anal.}, \textbf{284} (2023), no. 1, 109730.

 \bibitem[CDD22]{CDD22} I. Chifan, M. Davis, and D. Drimbe, Rigidity for von Neumann algebras of graph product groups. I. Structure of automorphisms, (2024), arXiv preprint arXiv:2209.12996.

 \bibitem[CDHK20]{CDHK20} I. Chifan, S. Das, C. Houdayer, and K. Khan, Examples of property (T) $\rm II_1$ factors with trivial fundamental group, \emph{Amer. J. Math.}, { \bf 146 } (2024), no. 2, 435-465.

 \bibitem[CDK19]{CDK19} I. Chifan, S. Das, and K. Khan, Some Applications of Group Theoretic Rips Constructions to the Classification of von Neumann Algebras, \emph{Anal. PDE}, {\bf 16} (2023), no. 2, 433-476.

 
 \bibitem[CDS23]{CDS23} I. Chifan, S. Das, and B. Sun, Invariant subalgebras of von Neumann algebras arising from negatively curved groups, \emph{J. Funct. Anal.}, \textbf{285} (2023), no. 9, 110098.

\bibitem[CdSS15]{CdSS15}I. Chifan, R. de Santiago, and T. Sinclair, W -rigidity for the von Neumann algebras of products of hyperbolic groups, \textit{Geom. Funct. Anal.} \textbf{26} (2016), 136–159.
 
 \bibitem[CdSS17]{CdSS17} I. Chifan, R. de Santiago and W. Suckpikarnon, Tensor product decompositions of $\rm II_1$ factors arising from extensions of amalgamated free product groups, \emph{Comm. Math. Phys.}, \textbf{364} (2018), no.3, 1163-1194.

\bibitem[CFQOT24]{CFQOT24} I. Chifan, A. Fernandez Quero, D. Osin, and H. Tan,  W$^*$-superrigidity for property (T) groups with infinite center, Preprint 2025, arXiv:2503.12742.

\bibitem[CH89]{CH89} M. Cowling, and U. Haagerup, Completely bounded multipliers of the Fourier algebra of a simple Lie group of real rank one, \textit{Invent. Math.} \textbf{96} (1989), 507-549.

 \bibitem[CH08]{CH08} I. Chifan, and C. Houdayer, Bass-Serre rigidity results in von Neumann algebras, { \it Duke Math. J.}, \textbf{153} (2010), no. 1, 23-54.

\bibitem[CI09]{CI09}I. Chifan, and A. Ioana, On a question of D. Shlyakhtenko, \textit{Proc. Amer. Math. Soc.}, \textbf{139}(2011), no. 3, 1091–1093.


 \bibitem[CI17]{CI17} I. Chifan, and A. Ioana, Amalgamated free product rigidity for group von Neumann algebras,\textit{ Adv. Math.}, \textbf{329} (2018), 819-850.

\bibitem[CIOS22a]{CIOS22a} I. Chifan, A. Ioana, D. Osin, and  B. Sun, Wreath-like products of groups and their von Neumann algebras $\rm I$: $W^{*}$-superrigidity, \emph{Ann. of Math.}, \textbf{198} (2023), no. 3, 1261-1303.

 \bibitem[CIOS23b]{CIOS23b} I. Chifan, A. Ioana, D. Osin, and B. Sun, Wreath-like product groups and their von Neumann algebras II: Outer automorphisms II, (2023), to appear in \emph{Duke Math. J.}, arXiv:2304.07457.

 \bibitem[CIOS24]{CIOS24} I. Chifan, A. Ioana, D. Osin, and B. Sun, Wreath-like products of groups and their von Neumann algebras $ \rm III $: Embeddings, (2024), arXiv preprint arXiv:2402.19461.

 \bibitem[CKP14]{CKP14}
I. Chifan, Y. Kida, and S. Pant, Primeness results for von Neumann algebras associated with surface braid groups, \emph{Int. Math. Res. Not.} \textbf{2016.16} (2016), 4807-4848.

\bibitem[CL25]{CL25}  I. Chifan, J. Lim,  Jones Index Set for Property (T) Group Factors, \textit{Work in progress} (2024).




 \bibitem[CS11]{CS11}I. Chifan, and T. Sinclair, On the structural theory of $\rm II_1$ factors of negatively curved groups,\emph{ Ann. Sci. \'Ec. Norm. Sup\'er.}, \textbf{46} (2013), no. 1, pp. 1–33.


 \bibitem[CSU11]{CSU11} I. Chifan, T. Sinclair, and B. Udrea, On the structural theory of $\rm  II_1$ factors of negatively curved groups, $\rm II$: Actions by product groups, \emph{Adv. Math.}, \textbf{245} (2013), 208-236.

 \bibitem[CSU16]{CSU16} I. Chifan, T. Sinclair, and B. Udrea, Inner amenability for groups and central sequences in factors, \emph{Ergodic Theory Dynam. Systems}, \textbf{36} (2016), no. 4, 1106--1029.
		
 \bibitem[CW22]{CW22} H. K. Chong and D. T. Wise, An uncountable family of finitely generated residually finite groups, \emph{J. Group Theory}, \textbf{25} (2022), no. 2, 207--216.

 
\bibitem[dlH95]{dlH95} P. de la Harpe, Operator algebras, free groups and other groups, \textit{Ast\'{e}risque}, tome 232 (1995), p. 121-153.

\bibitem[DGO17]{DGO17} F. Dahmani, V. Guirardel, and D. Osin, Hyperbolically embedded subgroups and rotating families in groups acting on hyperbolic spaces, \emph{American Mathematical Society}, \textbf{245} (2017), no. 1156.

\bibitem[Dri20]{Dri20} D. Drimbe, Prime II$_1$ factors arising from actions of product groups, \emph{J. Funct. Anal.} 278 (2020), no. 5,  108366.

\bibitem[DHI19]{DHI19} D. Drimbe, D. Hoff, and A. Ioana,  Prime $\rm II_1$ factors arising from irreducible lattices in products of rank one simple Lie groups, \emph{ J. Reine Angew. Math.}, \textbf{757} (2019), 197--246.

\bibitem[DK23]{DK23} F. Dahmani, and S. M. S. Krishna, Relative hyperbolicity of hyperbolic-by-cyclic groups, (2023), arXiv preprint arXiv:2006.07288.

\bibitem[DS05]{DS05} C. Drutu,
M. Sapir, Tree-graded spaces and asymptotic cones of groups. With an appendix by Denis Osin and Mark Sapir. \emph{Topology} \textbf{44} (2005), no. 5, 959-1058.

\bibitem[FGS10]{FGS10} J. Fang, S. Gao, and R. R. Smith, The Relative Weak Asymptotic Homomorphism Property for Inclusions of Finite von Neumann Algebras, \emph{Internat. J. Math.}, \textbf{22} (2011), no. 07, 991--1011.

\bibitem[FM77]{FM77} J. Feldman, and C. C. Moore, Ergodic equivalence relations, cohomology, and von Neumann algebras. II,\textit{ Trans. Amer. Math. Soc.} 234 (1977), no. 2, 325-359.

\bibitem[Ge96]{Ge96} L. Ge: {\it Applications of free entropy to finite von Neumann algebras, II},  Ann. of Math. {\bf 147} (1998), 143-157.

\bibitem[Gho18]{Gho18} P. Ghosh, Relatively irreducible free subgroups in Out($\mathbb {F} $), (2018), arXiv preprint arXiv:1802.05705 .

\bibitem[Gho23]{Gho23} P. Ghosh, Relative hyperbolicity of free-by-cyclic extensions, \emph{Compos. Math.}, \textbf{159} (2023), no. 1, 153–183.

\bibitem[GG23]{GG23}P. Ghosh, F. Gultepe Relative hyperbolicity of free extensions of free groups, (2023), arXiv preprint arXiv:2307.09674.

\bibitem[HV12]{HV12} C. Houdayer, and S. Vaes, Type III factors with unique Cartan decomposition, \emph{Journal de Mathématiques Pures et Appliquées}, \textbf{100} (2013), no. 4, 564-590.

\bibitem[IM19]{IM19} Y. Isono, and A. Marrakchi, Tensor product decompositions and rigidity of full factors, \emph{Ann. Sci. \'Ec. Norm. Sup\'er.}, \textbf{55} (2022),109-139.

\bibitem[Ioa11]{Ioa11} A. Ioana, $W^*$–superrigidity for Bernoulli actions of property (T) groups, \emph{J. Amer. Math. Soc.}, \textbf{24} (2011), no. 4, 1175-1226.

\bibitem[Ioa12]{Ioa12} A. Ioana, Cartan subalgebras of amalgamated free product II $ _1 $ factors, \emph{Ann. Sci. \'Ec. Norm. Sup\'er.}, \textbf{48} 
(2015), no. 1, 71-130. 

\bibitem[IPV10]{IPV10} A. Ioana, S. Popa, and S. Vaes, A class of superrigid group von Neumann algebras, \textit{ Ann. of Math.} (2) {\bf 178} (2013), no. 1,  231-286.

\bibitem[IPP05]{IPP05} A. Ioana, J. Peterson, and S. Popa, Amalgamated free products of weakly rigid factors and calculation of their symmetry groups, \emph{Acta Math.}, \textbf{200} (2008), 85--153.

\bibitem[Iso14]{Iso14} Y. Isono, Some prime factorization results for free quantum group factors, \emph{J. Reine Angew. Math.} 722 (2017),
215–250.

\bibitem[Iso16]{Iso16} Y. Isono, On fundamental groups of tensor product II$_1$ factors, \emph{Journal of the Institute of Mathematics of Jussieu} \textbf{19} (2020), no. 4, 1121-1139.


 
\bibitem[Jol14]{Jol14} P. Jollisaint, Proper cocycles and weak forms of amenability, (2014),  arXiv preprint arXiv:1403.0207.


\bibitem[Li18]{Li18}  R. Li, Relative hyperbolicity of suspensions of free products, \emph{PhD diss., Université Grenoble Alpes}, (2018).

\bibitem[LS77]{LS77} R. C. Lyndon and P. E. Schupp, Combinatorial group theory, Springer-Verlag, Berlin,
1977.

\bibitem[MvN43]{MvN43} F. J. Murray, and J. von Neumann, On rings of operators. IV, \emph{ Ann. of Math.}, (1943), 716-808.

\bibitem[NS21]{NS21}N. Nikolov and D. Segal, Constructing uncountably many groups with the same profinite completion, Preprint 2021, arXiv:2107.08877.

\bibitem[OP03]{OP03} N. Ozawa, and S. Popa, Some prime factorization results for type $\rm II_1$ factors, \emph{ Invent. Math.}, \textbf{156} (2004), no. 2, 223-234.


 \bibitem[Osi06]{Osi06} D. Osin, Small cancellations over relatively hyperbolic groups and embedding theorems, \emph{Ann. of Math.}, \textbf{172} (2010), 1--39. 

\bibitem[Oya23]{Oya23} K. Oyakawa, Bi-exactness of relatively hyperbolic groups, \emph{J. Funct. Anal.}, \textbf{284} (2023), no. 9, 109859.

 \bibitem[Oza03]{Oza03} N. Ozawa: {\it Solid von Neumann algebras}, Acta Math. {\bf 192} (2004), no. 1, 111-117.
 
 \bibitem[Oza04]{Oza04} N. Ozawa, Boundary amenability of relatively hyperbolic groups, \emph{Topology and its Applications}, \textbf{153} (2006), no. 14, 2624--2630.

\bibitem[Oza05]{Oza05} N. Ozawa, A Kurosh-type theorem for type  II$_1$  factors, \emph{Int. Math. Res. Not.} 2006, Art. ID 97560, 21 pp.

\bibitem[Pan99]{Pan99} A. E. Pankratoev, Hyperbolic products of groups, \emph{Vestnik Moskov. Univ. Ser. I Mat. Mekh.} 72 (1999), no. 2, 9–13.

\bibitem[Pet06]{Pet06} J. Peterson: {\it L$^2$-rigidity in von Neumann algebras}, Invent. Math. {\bf 175} (2009), no. 2, 417-433.

\bibitem[Pop83]{Pop83} S. Popa: Orthogonal pairs of $\ast$-subalgebras in finite von Neumann algebras, \emph{J. Operator Theory} {\bf 9} (1983),
no. 2, 253-268.

\bibitem[Pop99]{Pop99} S. Popa, Some properties of the symmetric enveloping algebra of a subfactor, with applications to amenability and property (T), \emph{Doc. Math.} {\bf 4} (1999), 665-744.

\bibitem[Pop01]{Pop01} S. Popa, On a Class of Type $\rm II_1$ Factors with Betti Numbers Invariants, \emph{Ann. of Math.}, \textbf{163} (2006), no 3, 809–899.
 
\bibitem[Pop03]{Pop03} S. Popa, Strong rigidity of $\rm II_1$ factors arising from malleable actions of w-rigid groups. $\rm I$, \emph{Invent. Math.}, {\bf 165} (2006), no. 2, 369-408.

\bibitem[Pop04]{Pop04} S. Popa, Strong rigidity of $\rm II_1$ factors arising from malleable actions of w-rigid groups, $\rm II$, \emph{Invent. Math.}, {\bf 165} (2006), no. 2, 409-451.

\bibitem[Pop06]{Pop06} S. Popa, Deformation and rigidity for group actions and von Neumann algebras, \emph{International Congress of Mathematicians}. Vol. I, 445-477, Eur. Math. Soc., Z¨urich, 2007.

\bibitem[Pop06b]{Pop06b} S. Popa,  On Ozawa’s property for free group factors, \emph{Int. Math. Res. Not.} \textbf{2007} (2007), no. 11, Art. ID rnm036, 10 pp.

\bibitem[PP86]{PP86} M. Pimsner and S. Popa,  Entropy and index for subfactors, {\it Ann. Sci. \'Ec. Norm. Sup\'er.}, {\bf 19} (1986), no. 1, 57-106.

\bibitem[PV06]{PV06}
S. Popa, and S. Vaes, Strong rigidity of generalized Bernoulli actions and computations of their symmetry groups, \textit{Adv. Math.} \textbf{217} (2008), no. 2, 833-872. 

\bibitem[PV11]{PV11} S. Popa, and S. Vaes, Unique Cartan decomposition for $\rm II_1$ factors arising from arbitrary actions of free groups, \emph{Acta Math. 212 (1)} (2014), 141-198.

\bibitem[PV12]{PV12} S. Popa, and S. Vaes, Unique Cartan decomposition for $\rm II_1$ factors arising from arbitrary actions of hyperbolic groups, \emph{  J. Reine Angew. Math.}, \textbf{694} (2014), 215-239.

\bibitem[PV21]{PV21} S. Popa, and S. Vaes, $W^*-$Rigidity Paradigm for Embeddings of $\rm II_1$ Factors, \emph{Comm. Math. Phys.}, \textbf{395}(2022), no.2, 907-961.

\bibitem[Sin10]{Sin10} T. Sinclair, Strong solidity of group factors from lattices in $SO(n,1)$ and $SU(n,1)$, \emph{J. Funct. Anal.}, {\bf 260} (2011), no. 11, 3209--3221.

\bibitem[Sun20]{Sun}
B. Sun, Cohomology of group theoretic Dehn fillings I: Cohen-Lyndon type theorems,
\emph{J. Algebra}, \textbf{542} (2020), 277-307.


\bibitem[SW12]{SW12} J. O. Sizemore, and A. Winchester, Unique prime decomposition results for factors coming from wreath product groups, \emph{Pacific J. Math.}, \textbf{265} (2013), no. 1, 221-232.

\bibitem[Vae08]{Vae08} S. Vaes, Explicit computations of all finite index bimodules for a family of $\rm II_1 $ factors, \emph{Ann. Sci. \'Ec. Norm. Sup\'er.}, \textbf{41} (2008), no. 5, 743-788.

\bibitem[Vae10]{Vae10} S. Vaes, One-cohomology and the uniqueness of the group measure space decomposition of a $\rm II_1$ factor, \emph{Math. Ann.}, \textbf{355} (2013), no. 2, 661-696.

\bibitem[Vae13]{Vae13} S. Vaes, Normalizers inside amalgamated free product von Neumann algebras, \emph{Publications of the RIMS}, \textbf{50} (2014), no. 4, 695-721.

\bibitem[Ver22]{Ver22} I. Vergara, Weak amenability of free products of hyperbolic and amenable groups,
\emph{Glasg. Math. J.} 64 (2022), no. 3, 698-701.

\bibitem[VV14]{VV14} S. Vaes, and P. Verraedt, Classification of type III Bernoulli crossed products, \emph{Advances in Mathematics} \textbf{281} (2015), 296-332.

\end{thebibliography}
\end{document}